\documentclass[12pt, reqno, a4paper]{amsart}

\usepackage{amssymb}
\usepackage{color}
\definecolor{darkgreen}{rgb}{0,0.45,0} 
\usepackage[
colorlinks,citecolor=blue,linkcolor=blue]{hyperref}
\usepackage[matrix, arrow, curve]{xy}
\usepackage{bbm}
\usepackage{xcolor}
\usepackage{comment}
\usepackage{enumitem}

\usepackage[english]{babel}

\binoppenalty=\maxdimen
\relpenalty=\maxdimen

\theoremstyle{plain}
\newtheorem{theorem}{Theorem}[section]
\newtheorem{lemma}[theorem]{Lemma}
\newtheorem{proposition}[theorem]{Proposition}
\newtheorem{corollary}[theorem]{Corollary}

\theoremstyle{remark}
\newtheorem{remark}[theorem]{Remark}
\newtheorem{remarks}[theorem]{Remarks}

\theoremstyle{definition}
\newtheorem{example}[theorem]{Example}

\newtheorem{definition}[theorem]{Definition}

\numberwithin{equation}{section}

\DeclareMathOperator{\id}{id}

\DeclareMathOperator{\End}{End}

\DeclareMathOperator{\supp}{supp}
\DeclareMathOperator{\cosupp}{cosupp}

\DeclareMathOperator{\Hom}{Hom}

\newcommand{\LIO}{\mathsf{LIO}}

\newcommand{\Ker}{{\sf Ker}\,}

\newbox\skewpullbackbox
\setbox\skewpullbackbox=\hbox{\xy 0;<1mm,0mm>: \POS(4,0)\ar@{-} (-4,0) \ar@{-} (8,4)
	\endxy}

\newbox\skwepullbackbox
\setbox\skwepullbackbox=\hbox{\xy 0;<1mm,0mm>: \POS(16,0)\ar@{-} (10,0) \ar@{-} (12,4)
	\endxy}

\newbox\ksewpullbackbox
\setbox\ksewpullbackbox=\hbox{\xy 0;<1mm,0mm>: \POS(0,-8)\ar@{-} (0,-4) \ar@{-} (4,-4)
	\endxy}

\newbox\pullbackbox
\setbox\pullbackbox=\hbox{\xy 0;<1mm,0mm>: \POS(4,0)\ar@{-} (0,0) \ar@{-} (4,4)
	\endxy}
\newcommand{\pullback}{\copy\pullbackbox}

\newbox\pullbackabox
\setbox\pullbackabox=\hbox{\xy 0;<1mm,0mm>: \POS(-4,-6)\ar@{-} (-8,-6) \ar@{-} (-4,-2)
	\endxy}

\newbox\pullbackbbox
\setbox\pullbackbbox=\hbox{\xy 0;<1mm,0mm>: \POS(-4,-4)\ar@{-} (-8,-4) \ar@{-} (-4,0)
	\endxy}

\newbox\pullbackcbox
\setbox\pullbackcbox=\hbox{\xy 0;<1mm,0mm>: \POS(4,-7)\ar@{-} (0,-7) \ar@{-} (4,-3)
	\endxy}

\newbox\pullbackdbox
\setbox\pullbackdbox=\hbox{\xy 0;<1mm,0mm>:\POS(-10,-6)\ar@{-} (-14,-6) \ar@{-} (-10,-2)
	\endxy}

\newbox\pushoutbox
\setbox\pushoutbox=\hbox{\xy 0;<1mm,0mm>: \POS(0,4)\ar@{-} (0,0) \ar@{-} (4,4)
	\endxy}

\newbox\pushoutabox
\setbox\pushoutabox=\hbox{\xy 0;<1mm,0mm>: \POS(2,8)\ar@{-} (2,4) \ar@{-} (8,8)
	\endxy}

\DeclareRobustCommand{\No}{\ifmmode{\nfss@text{\textnumero}}\else\textnumero\fi} 

\newcounter{dummy}
\makeatletter
\newcommand\myitem[1][]{\item[(#1)]\refstepcounter{dummy}\def\@currentlabel{#1}}
\makeatother

\usepackage{graphicx}
\usepackage[geometry]{ifsym}

\renewcommand{\square}{\text{\normalfont\scalebox{.6}{\SmallSquare}}}

\makeatletter
\@namedef{subjclassname@2020}{\textup{2020} Mathematics Subject Classification}
\makeatother

\begin{document}

\title[Dualities for universal (co)acting Hopf monoids]{Dualities for universal (co)acting Hopf monoids}

\author{A.L. Agore}
\address{Simion Stoilow Institute of Mathematics of the Romanian Academy, P.O. Box 1-764, 014700 Bucharest, Romania}
\address{Max Planck Institut f\"{u}r Mathematik, Vivatsgasse 7, 53111 Bonn, Germany}
\email{ana.agore@gmail.com, ana.agore@imar.ro}

\author{A.S. Gordienko}
\address{Department of Higher Algebra,
	Faculty of Mechanics and Mathematics,
	M.\,V.~Lomonosov Moscow State University,
	Leninskiye Gory, d.1, Main Building, GSP-1, 119991 Moskva, Russia }
\email{alexey.gordienko@math.msu.ru}

\author{J. Vercruysse}
\address{D\'epartement de Math\'ematiques, Facult\'e des sciences, Universit\'e Libre de Bruxelles, Boulevard du Triomphe, B-1050 Bruxelles, Belgium}
\email{jvercruy@ulb.ac.be}


\begin{abstract}	In general, universal (co)measuring (co)monoids and universal (co)acting bi/Hopf monoids, which prove to be a useful tool in the classification of quantum symmetries, do not always exist. In order to ensure their existence, the support of a given object was recently introduced in \cite{AGV3} and used to restrict the class of objects considered when defining universal (co)acting objects. It is well-known that, in contrast with the universal coacting Hopf algebra, for actions on algebras over a field it is usually difficult to describe the universal acting Hopf algebra explicitly and this turns the duality theorem into an important investigation tool. In the present paper we establish duality results for universal (co)measuring (co)monoids and universal (co)acting bi/Hopf monoids in pre-rigid braided monoidal categories $\mathcal C$. In addition, when the base category $\mathcal C$ is closed monoidal, we provide a convenient uniform approach to the aforementioned universal objects in terms of the cosupports, which in this case become subobjects of internal hom-objects. In order to explain our constructions, we use the language of locally initial objects. Known results from the literature are recovered when the base category is the category of vector spaces over a field. New cases where our results can be applied are explored, including categories of (co)modules over (co)quasitriangular Hopf algebras, Yetter~--- Drinfel’d modules and dg-vector spaces.
\end{abstract}


\thanks{The study carried out by the second author was conducted under the state assignment of Lomonosov Moscow State University. The third author thanks the FNRS for support via the MIS project ``Antipode'' (Grant F.4502.18)}

\maketitle

\tableofcontents

\section{Introduction}

The geometrical problem of classifying quantum symmetries of a given algebra $A$ has an algebraic correspondent in the classification of (co)module structures on $A$. Indeed, recall that if an affine algebraic group $G$ is acting on an affine algebraic variety $X$, the morphic action $G\times X \to X$ corresponds to a homomorphism of algebras $\mathcal O(X) \to \mathcal O(X) \otimes \mathcal O(G)$ where $\mathcal O(X)$ and $\mathcal O(G)$ are the algebras of regular functions on $X$ and $G$, respectively. Moreover, the group structure on $G$ endows $\mathcal O(G)$ with a Hopf algebra structure and $\mathcal O(X)$ becomes an $\mathcal O(G)$-comodule algebra. Furthermore, if $\mathfrak g$ is the Lie algebra  of $G$, then $\mathcal O(X)$ is a $U(\mathfrak g)$-module algebra where $U(\mathfrak g)$ denotes the universal enveloping algebra of $\mathfrak g$. Replacement of the commutative algebras $\mathcal O(G)$ and $\mathcal O(X)$ (respectively the cocommutative Hopf algebra $U(\mathfrak g)$) with an arbitrary Hopf algebra $H$ (co)acting on an arbitrary algebra $A$, leads to an action of a quantum group by quantum symmetries on an algebraic variety $X$ whose algebra of regular functions is $A$. Furthermore, if these quantum groups act on certain cohomological invariants of the variety, one could hope to obtain more geometrical information.

A typical example of a comodule algebra structure is the one defined by group gradings on an algebra over a field. These are usually classified either up to an isomorphism by considering the grading group to be fixed or up to equivalence when it is not important by elements of which group the graded components are marked, see e.g.  \cite{ElduqueKochetov}. However, the universal group of a grading allows us to recover all groups that realize a concrete grading. The corresponding notions of equivalence and universal Hopf algebras of (co)module structures on algebras were introduced in~\cite{AGV1}, generalizing the aforementioned universal group of a grading. In a certain sense (\cite[Remark 4.15]{AGV1}) the aforementioned construction provides a refinement of Manin's universal coacting Hopf algebra (\cite{Manin}), an important and intensively studied (see e.g. \cite{HNUVVW, HNUVVW2, HNUVVW3}) symmetry object in noncommutative geometry. Furthermore, a unifying theory for universal Hopf algebras of (co)module structures and universal (co)acting bi/Hopf algebras
of Sweedler~-- Manin~-- Tambara (\cite{Sweedler, Manin, Tambara}) was introduced in~\cite{AGV2}, by considering $V$-universal (co)acting bi/Hopf algebras
where $V$ is a unital subalgebra of $\End_{\mathbbm k}(A)$ and $\mathbbm k$ is the base field. Motivation comes, on the one hand, from the fact that this unified theory simplifies and in certain cases even makes it possible at all to classify (co)module structures by means of duality results. On the other hand, since the universal coacting bi/Hopf algebras
of Manin~--- Tambara do not always exist~\cite[Section 4.5]{AGV2}, $V$ provides the necessary restriction on the class of comodule structures under consideration to ensure the existence of the universal Hopf algebra for this class.

In~\cite{AGV3} the authors introduced the categorical foundations for the notion of universal bi/Hopf algebra. Sufficient conditions for the existence of universal (co)measuring (co)monoids and universal (co)acting bi/Hopf monoids over a base  (braided or symmetric monoidal) category are given. Furthermore, it was shown that the existence problem for such universal objects is a particular case of what we call the \emph{Lifting Problem for locally initial objects} (see Section~\ref{SubsectionLIO} for the precise statement) as we will briefly explain. By definition, the (co)action of a universal Hopf monoid
is an initial object in some full subcategory of the category of all (co)actions on a fixed $\Omega$-magma.
In particular, for the whole category of (co)actions this universal (co)action is a \emph{locally initial object} (i.e. an object which
admits at most one morphism into any other object). That full subcategory consists of the objects whose images under some forgetful functor admit a morphism from a certain fixed locally initial object. Therefore, the existence problem for universal Hopf monoids can be described as a lifting problem for locally initial objects. The lifting itself is carried out in several steps.

The first duality result for the (absolute) universal comeasuring algebra and measuring coalgebra, which generalizes the classical adjunction between the finite dual (of an algebra over a field) functor and the dual algebra (of a coalgebra over a field) functor (see e.g., \cite[Theorem 1.5.22]{DNR}), was established by D.~Tambara in~\cite{Tambara}. This duality was extended to $V$-universal (co)measuring (co)algebras and (co)acting bi/Hopf algebras over fields in~\cite[Theorems 3.20, 4.14, 4.15]{AGV2}. As it was proven in~\cite{AGV2}, the $V$-universal coacting Hopf algebras over fields admit a transparent description in terms of free algebras and relations, while the construction of $V$-universal acting Hopf algebras is not as explicit, since it involves subcoalgebras of cofree coalgebras. For this reason the duality theorems for $V$-universal acting and coacting Hopf algebras become here one of the main tools of investigation. In the present paper we prove duality results for universal (co)measuring (co)monoids and universal (co)acting bi/Hopf monoids in pre-rigid braided monoidal categories $\mathcal C$, generalizing the aforementioned existing results. 

We use the previously introduced terminology to describe the outline of the paper. Throughout, we use the following notation: \begin{itemize}
\item $A$ and $B$ are $\Omega$-magmas;
\item $\mathbf{MorTens}(A,B)$ and $\mathbf{TensMor}(A,B)$ are the categories of, respectively, morphisms $A\to B \otimes Q$
and morphisms $P\otimes A \to B$ for some objects $P$ and $Q$; 
\item $\mathbf{Comeas}(A,B)$ and $\mathbf{Meas}(A,B)$ are the categories of, respectively, comeasurings $A\to B \otimes Q$
and measurings $P\otimes A \to B$ for some comonoids $P$ and monoids $Q$; 
\item $\mathbf{ComodStr}(A)$ and $\mathbf{ModStr}(A)$ the categories of, respectively, comodule structures $A\to A \otimes P$ and module structures $Q\otimes A \to A$ where $P$ is a comonoid and $Q$ a monoid;  
\item $\mathbf{Coact}(A)$ and $\mathbf{Act}(A)$ are the categories of bimonoid (co)actions on $A$;
\item $\mathbf{HCoact}(A)$ and $\mathbf{HAct}(A)$ are the categories of  Hopf monoid (co)actions on $A$;
\item functors $G$, $G_1$--$G_4$, $G'$, $G_1'$--$G_4'$ are the corresponding forgetful/embedding functors;
\item $\rho^\vee \colon Q^* \otimes A \to B $ is the morphism induced by
 $\rho \colon A \to B \otimes Q$ (here we assume that $\mathcal C$ is \textit{pre-rigid}, in particular, endowed with
 a functor $(-)^* \colon \mathcal C \to \mathcal C^{\mathrm{op}}$);
 \item $\varkappa_Q \colon Q^\circ \to Q^*$ is the natural morphism that relates the finite dual $Q^\circ$
 of a monoid $Q$ with $Q^*$.
\end{itemize}
For full details we refer the reader to Sections~\ref{SectionPreliminaries}--\ref{SectionDuality}. To start with, consider the following (not necessarily commutative) diagram:

$$\xymatrix{ \mathbf{Comeas}(A,B) \ar[d]_{G_1} \ar[rr]^{(-)^\vee (\varkappa_{(\ldots)} \otimes \id_A)} & \quad & \mathbf{Meas}(A,B)^{\mathrm{op}} \ar[d]^{G'_1} \\
	\mathbf{MorTens}(A,B) \ar[rr]^{(-)^\vee} &  & \mathbf{TensMor}(A,B)^{\mathrm{op}}
}$$

In~\cite[Theorems 4.24 and 5.19]{AGV3} 
 locally initial objects were lifted along the functors $G_1$ and $G_1'$, respectively.
Theorem~\ref{Theorem(Co)monUniv(Co)measDuality} of the present paper provides sufficient conditions
for the functor $(-)^\vee (\varkappa_{(\ldots)} \otimes \id_A) \colon \mathbf{Comeas}(A,B) \to \mathbf{Meas}(A,B)^{\mathrm{op}}$
to map universal comeasurings to universal measurings.

In addition, if $A=B$, one can consider bi/Hopf monoid (co)actions. The complete picture is captured by the following diagram which commutes except, possibly, for the lower central square: 

$$\xymatrix{
	& \mathbf{HCoact}(A)  \ar[d]_{G_4}  \ar[rr]^{(-)^\vee (\varkappa_{(\ldots)} \otimes \id_A)} & \quad & \mathbf{HAct}(A)^{\mathrm{op}} \ar[d]^{G'_4} \\
\mathbf{Comeas}(A,A) \ar[d]_{G_1}	& \mathbf{Coact}(A) \ar[l]_{G_3} \ar[d]_{G_2} \ar[rr]^{(-)^\vee (\varkappa_{(\ldots)} \otimes \id_A)} & \quad & \mathbf{Act}(A)^{\mathrm{op}} \ar[d]^{G'_2} \ar[r]^{G'_3} & \mathbf{Meas}(A,A)^{\mathrm{op}} \ar[d]^{G'_1}   \\
\mathbf{MorTens}(A,A)	& \mathbf{ComodStr}(A) \ar[l]_G    \ar[rr]^{(-)^\vee} & \quad & \mathbf{ModStr}(A)^{\mathrm{op}}  
\ar[r]^{G'} & \mathbf{TensMor}(A,A)^{\mathrm{op}}
}$$	

In~\cite[Corollaries~4.37 and~5.24]{AGV3} 
 objects in $\mathbf{ComodStr}(A)$ and $\mathbf{ModStr}(A)^{\mathrm{op}}$, whose images under $G$
and $G'$  are locally initial, are lifted along the functors $G_2$ and $G_2'$, respectively.
The actual lifting is made along $G_3$ and $G_3'$.
Theorem~\ref{TheoremBimonUniv(Co)actDuality} is the corresponding duality result.

Furthermore, the locally initial objects obtained in $\mathbf{Coact}(A)$ and $\mathbf{Act}(A)^{\mathrm{op}}$,
are lifted to $\mathbf{HCoact}(A)$ and $\mathbf{HAct}(A)^{\mathrm{op}}$, respectively, in~\cite[Theorems~4.42 and~5.27]{AGV3}. 
The corresponding duality result is proved in Theorem~\ref{TheoremHopfMonUniv(Co)actDuality}.

In the dualizable theory developed in~\cite{AGV3}, coactions were restricted by their supports while actions were restricted by their cosupports. However, if $\mathcal C$ is not only pre-rigid, but closed monoidal, there exists an isomorphism of categories $$K \colon  \mathbf{TensMor}(A,B) \mathrel{\widetilde\to} (\mathcal C \downarrow [A,B]),$$
which makes it possible to segregate classes of both measurings and comeasurings in terms of 
their cosupports. The latter correspond to locally initial objects in 
the comma category $(\mathcal C \downarrow [A,B])^{\mathrm{op}}$, i.e. just monomorphisms.
This approach, which proves to be fruitful for the classification of (co)module structures on a given $\Omega$-magma,
is undertaken in Section~\ref{SectionCosupportDualityInMonoidalClosedCategories}. 

The categorical approach carried out in~\cite{AGV3} and in the present paper makes it possible not only to recover known results,
when the base category is the category of vector spaces over a field, but also to consider new cases, e.g. when the base category itself is a category of (co)modules over a (co)quasitriangular Hopf algebra, Yetter~--- Drinfel'd modules or dg-vector spaces. All these examples are considered in full detail in Section~\ref{SectionApplications}.

\section{Preliminaries}\label{SectionPreliminaries}

Throughout the paper we assume familiarity with (braided monoidal) categories as covered for instance in \cite{AHS, MacLaneCatWork} (resp. \cite{JoyStr}) and Hopf algebras where the relevant background is available in a number of sources such as \cite{DNR, Montgomery, Radford, Sweedler}.
We start by briefly recalling definitions and notation from~\cite{AGV3} that will be used in the sequel.

\subsection{Locally initial objects}\label{SubsectionLIO}
An object $x_0$ in a category $X$ is \textit{locally initial} if for every object $x$ in $X$ there exists
at most one morphism $x_0 \to x$. 

Locally initial objects obviously generalize the classical initial object of a category. However, as opposed to an initial object, locally initial objects may not be unique up to isomorphism. 

Furthermore, locally initial objects form a preorder $\LIO(X)$ by considering $x_1 \succcurlyeq x_2$ if there exists a morphism $x_1 \to x_2$. If $x_0$ is a given locally initial object we denote by $X(x_0)$ the full subcategory of $X$ consisting of all objects $x$ such that there exists a morphism $x_0 \to x$. We can easily see that $x_0$ is the initial object in $X(x_0)$.

Consider now another category $Y$ and a functor $G \colon Y \to X$. Given $x_0 \in \LIO(X)$ denote by  $Y_G(x_0)$ the full subcategory of $Y$ consisting of all objects $y$
such that $Gy$ is an object in $X(x_0)$. In this context, we formulate:

\noindent\textbf{Lifting Problem.} \textit{Given $x_0 \in \LIO(X)$, find an initial object $y_0$ in $Y_G(x_0)$.}

Note that the condition $Gy_0 = x_0$ is not assumed to hold. 

\subsection{Absolute values}\label{SubsectionLIOAbsValue}
For a given object $x$ of the category $X$ denote by $|x|$ an object in $\LIO(X)$
such that $x$ is an object in $X(|x|)$ and for any other object $x_1$ in $\LIO(X)$
such that $x$ is an object in $X(x_1)$ we have $|x| \preccurlyeq x_1$.
In other words, $|x|$ is the absolute minimum of such $x_1 \in \LIO(X)$
for which there exists a morphism $x_1 \to x$:
$$\xymatrix{ x &  \ar[l] \ar@{-->}[d] x_1 \\
	& |x| \ar[lu]
}$$
The terminology used is justified by \cite[Example 2.7]{AGV3}.

It can be easily seen that for every $x \in \LIO(X)$ we have $x = |x|$. Furthermore, note that the arrow $|x| \to x$ is just the terminal object in the comma category~$(\LIO(X)\downarrow x)$.

Now if $G \colon Y \to X$ is a functor, for a given object $y$ in $Y$ we denote $|y|_G := |Gy|$.

 For objects $y_1, y_2$ in $Y$ we write $y_1 \preccurlyeq y_2$
and say that $y_1$ is \textit{coarser} than $y_2$ and $y_2$ is \textit{finer} than $y_1$ if $|y_1|_G \preccurlyeq |y_2|_G$.
If $y_1 \preccurlyeq y_2$ and $y_2 \preccurlyeq y_1$, then we say that $y_1$ and $y_2$ are \textit{(support) equivalent}.
Since the absolute value is defined up to an isomorphism, $y_1$ and $y_2$ are equivalent if and only if $|y_1|_G = |y_2|_G$.

\begin{remark} 
	If absolute values of all objects in $X$ exist, then, given $x_0 \in \LIO(X)$, the category $Y_G(x_0)$ consists of all objects $y$ in $Y$ such that $|y| \preccurlyeq x_0$.
\end{remark}

\begin{definition}
Let $K \colon X \to X'$ be a functor for some categories~$X, X'$. We say that
\begin{enumerate}
\item[1)] $K$ \textit{preserves absolute values}
	if for every $x$ in $X$ with an absolute value $|x|$ we have $K\left(|x|\right) = |Kx|$;
\item[2)] $K$ \textit{reflects the preorder} (on locally initial objects) if the following two conditions hold:
	\begin{itemize}
		\item the functor $K$ maps $\LIO(X)$ to $\LIO(X')$;
		\item $Kx_1 \preccurlyeq Kx_2$ in $X'$ for some $x_1, x_2 \in \LIO(X)$
		if and only if $x_1 \preccurlyeq x_2$ in $X$.
	\end{itemize}
\end{enumerate}	
\end{definition}

\begin{proposition}\label{PropositionLIOKPreservesPreorder}
	Let $G \colon Y \to X$ and $K \colon X \to X'$ be functors for some categories~$X, X',Y$ such that there exist absolute values of all objects in $X$ and $K$ preserves absolute values and reflects the preorder. Let $x_0$ be an object in $X$ such that
	$Kx_0 \in \LIO(X')$.
	Then the categories $Y_G(|x_0|)$ and $Y_{KG}(Kx_0)$ coincide.
	As a consequence, if $y_0$ is the initial object in $Y_G(|x_0|)$, then  $y_0$ is the initial object in~$Y_{KG}(Kx_0)$.
\end{proposition}
\begin{proof} Note that $Y_G(|x_0|)$ is a full subcategory of $Y_{KG}(Kx_0)$
	since $ Kx_0=|Kx_0|=K|x_0|$.
	Suppose now that $y$ is an object in $Y_{KG}(Kx_0)$. Then there exists an arrow
	$Kx_0 \to KGy$ in $X'$. By our assumptions, $K|Gy|$ is the absolute value of $KGy$ in $X'$.
	Hence $Kx_0 \to KGy$ factors through $K|Gy|$. In particular,
	there exists an arrow $K|x_0|\to K|Gy|$ in $X'$
	and $ K|Gy| \preccurlyeq K|x_0|$. Thus
	$|Gy|  \preccurlyeq |x_0|$ and $y$ is an object in $Y_G(|x_0|)$.
\end{proof}

\subsection{$\Omega$-magmas}\label{SubsectionOmegaMagmas}

Let $\Omega$ be a set together with maps $s,t \colon \Omega \to \mathbb Z_+$. 

An \textit{$\Omega$-magma} in a monoidal category $\mathcal C$ is 
	an object $A$ endowed with morphisms $\omega_A \colon A^{\otimes s(\omega)} \to
	A^{\otimes t(\omega)}$ for every $\omega \in \Omega$. We will usually drop the subscript $A$ and denote
	the map just by $\omega$. We use the convention $A^{\otimes 0} := \mathbbm{1}$, the monoidal unit in $\mathcal C$.
Note that here we do not require from $\omega_A$ to satisfy any identities. $\Omega$-magmas in $\mathbf{Vect}_\mathbbm{k}$ are called \textit{$\Omega$-algebras} over $\mathbbm{k}$ (\cite{AGV2}).

Examples include many familiar algebraic structures: algebras (either unital or nonunital, associative or non-associative) and coalgebras over a field are $\Omega$-magmas in $\mathbf{Vect}_\mathbbm{k}$ for different $\Omega$'s, ordinary monoids are $\Omega$-magmas in $\mathbf{Sets}$, any object $A$ endowed with a braiding $c_A \colon A \otimes A \to A \otimes A$ is an $\Omega$-magma (see \cite[Examples 3.5]{AGV3} for more details). 





\subsection{(Co)measurings}\label{Subsection(Co)Measurings}

Fix a braided monoidal category $\mathcal C$ with a braiding $c$ and a monoidal unit~$\mathbbm{1}$. Let $P$ be a comonoid in $\mathcal C$ with a comultiplication
$\Delta \colon P \to P \otimes P$ and a counit $\varepsilon \colon P \to \mathbbm{1}$. Consider the monoidal category $\mathsf{TensMor}(P)$ where the objects are morphisms $P\otimes A \to B$ and morphisms between objects $\psi_1 \colon P\otimes A_1 \to B_1$ and $\psi_2 \colon P\otimes A_2 \to B_2$
are pairs of morphisms $\alpha \colon A_1 \to A_2$ and $\beta \colon B_1 \to B_2$ making the diagram below
commutative:

$$\xymatrix{  P\otimes A_1 \ar[d]_{\id_P\, \otimes \, \alpha }\ar[r]^(0.6){\psi_1} & B_1 \ar[d]^{\beta } \\
	P\otimes A_2 \ar[r]^(0.6){\psi_2} & B_2
}$$

The monoidal product $\psi_1 \mathbin{\widetilde\otimes} \psi_2 \colon P \otimes (A_1 \otimes A_2) \to B_1 \otimes B_2$
of objects $\psi_1 \colon P\otimes A_1 \to B_1$ and $\psi_2 \colon P\otimes A_2 \to B_2$
in $\mathsf{TensMor}(P)$ is defined as the composition of the morphisms below:

$$\xymatrix{ P \otimes (A_1 \otimes A_2) \ar[rr]^(0.45){\Delta \, \otimes \,  \id_{A_1 \otimes A_2}} &  & (P \otimes P) \otimes (A_1 \otimes A_2) \ar[rr]^{\id_{P} \, \otimes \,  c_{P, A_1} \, \otimes \,  \id_{A_2}} & \quad & (P \otimes A_1) \otimes (P \otimes A_2)
	\ar[d]^{\psi_1 \otimes \psi_2} \\ & & & & B_1 \otimes B_2 }$$

The monoidal unit of $\mathsf{TensMor}(P)$ is the composition $P \otimes \mathbbm{1} \mathop{\widetilde{\to}} P \xrightarrow{\varepsilon} \mathbbm{1}$.
The axioms of a monoidal category for $\mathsf{TensMor}(P)$ are consequences of those for $\mathcal C$ and the fact that $P$ is a comonoid.

A \textit{measuring} of $\Omega$-magmas is an $\Omega$-magma $\psi \colon P\otimes A \to B$ in the category $\mathsf{TensMor}(P)$. Note that the structure of an $\Omega$-magma on $\psi$ endows the objects $A$ and $B$ with structures
of $\Omega$-magmas in $\mathcal C$ and the morphism $\psi$ relates these structures in a special way.

The classical definition of a measuring of (now not necessarily associative) algebras as introduced in \cite[Chapter VII]{Sweedler} can be recovered for $\mathcal C = \mathbf{Vect}_\mathbbm{k}$ for a field $\mathbbm{k}$ and a certain $\Omega$ (see \cite[Examples 3.6]{AGV3}). 



For a monoidal category~$\mathcal C$ denote by $\mathsf{Mon}(\mathcal C)$ and $\mathsf{Comon}(\mathcal C)$
the categories of monoids and comonoids in~$\mathcal C$, respectively.

Recall that if the category $\mathcal C$ is braided, then $\mathsf{Mon}(\mathcal C)$ is a monoidal category too.
Objects of the category $\mathsf{Comon}(\mathsf{Mon}(\mathcal C))$ (which is isomorphic to $\mathsf{Mon}(\mathsf{Comon}(\mathcal C))$) are called \textit{bimonoids} in $\mathcal C$.

If $P$ is a bimonoid, then the category ${}_P\mathsf{Mod}$ of left $P$-modules
is a subcategory of $\mathsf{TensMor}(P)$ that inherits from $\mathsf{TensMor}(P)$ the monoidal structure.
An $\Omega$-magma in ${}_P\mathsf{Mod}$ is called a \textit{$P$-module $\Omega$-magma}
and the corresponding morphism $\psi \colon P\otimes A \to A$ is called a \textit{$P$-action} on $A$.

The classical definition of a (now not necessarily associative) module algebra over a bialgebra can be recovered for $\mathcal C = \mathbf{Vect}_\mathbbm{k}$ for a field $\mathbbm{k}$ and a certain $\Omega$ (see \cite[Examples 3.7]{AGV3}).



Dually,  let $Q$ be a monoid in $\mathcal C$ with a multiplication
$\mu \colon Q \otimes Q \to Q$ and a unit $u \colon \mathbbm{1} \to Q$. Consider the monoidal category $\mathsf{MorTens}(Q)$ where the objects are morphisms $ A \to B \otimes Q$ and morphisms between objects $\rho_1 \colon 
A_1 \to B_1 \otimes Q$ and $\rho_2 \colon 
A_2 \to B_2 \otimes Q$
are pairs of morphisms $\alpha \colon A_1 \to A_2$ and $\beta \colon B_1 \to B_2$ making the diagram below
commutative:

$$\xymatrix{  A_1 \ar[d]_{ \alpha }\ar[r]^(0.4){\rho_1} & B_1 \otimes Q \ar[d]^{\beta \, \otimes \, \id_Q} \\
	A_2 \ar[r]^(0.4){\rho_2} & B_2 \otimes Q
}$$

The monoidal product $\rho_1 \mathbin{\widetilde\otimes} \rho_2 \colon A_1 \otimes A_2 \to (B_1 \otimes B_2) \otimes Q$
of objects $\rho_1 \colon A_1 \to B_1  \otimes Q$ and $\rho_2 \colon A_2 \to B_2  \otimes Q$
in $\mathsf{MorTens}(Q)$ is defined as the composition of the morphisms below:

$$\xymatrix{ A_1 \otimes A_2 \ar[r]^(0.35){\rho_1 \otimes \rho_2} &
	(B_1 \otimes Q) \otimes (B_2 \otimes Q) \ar[rr]^{\id_{B_1} \, \otimes \,  c_{Q, B_2} \, \otimes \,  \id_{Q}}
	&\qquad & (B_1 \otimes B_2) \otimes (Q \otimes Q)   	\ar[d]^{\id_{B_1 \otimes B_2} \, \otimes \, \mu}
	\\ & & & (B_1 \otimes B_2) \otimes Q }$$

The monoidal unit of $\mathsf{MorTens}(Q)$ is the composition $ \mathbbm{1}  \xrightarrow{u}
Q \mathop{\widetilde{\to}} \mathbbm{1} \otimes Q$.
The axioms of a monoidal category for $\mathsf{MorTens}(Q)$ are consequences of those for $\mathcal C$ and the fact that $Q$ is a monoid.

A \textit{comeasuring} of $\Omega$-magmas is an $\Omega$-magma $\rho \colon A \to B \otimes Q$ in the category $\mathsf{MorTens}(Q)$. Note that the structure of an $\Omega$-magma on $\rho$ endows the objects $A$ and $B$ with structures
of $\Omega$-magmas in $\mathcal C$ and the morphism $\rho$ relates these structures in a special way.



If $Q$ is a bimonoid, then the category $\mathsf{Comod}^Q$ of right $Q$-comodules
is a subcategory of $\mathsf{MorTens}(Q)$ that inherits from $\mathsf{MorTens}(Q)$ the monoidal structure.
An $\Omega$-magma in $\mathsf{Comod}^Q$ is called a \textit{$Q$-comodule $\Omega$-magma}
and the corresponding morphism $\rho \colon  A \to A \otimes Q$ is called a \textit{$Q$-coaction} on $A$.

The classical definition of a (now not necessarily associative) comodule algebra over a bialgebra can be recovered for $\mathcal C = \mathbf{Vect}_\mathbbm{k}$ for a field $\mathbbm{k}$ and a certain $\Omega$ (see \cite[Examples 3.9]{AGV3}).



\subsection{Conditions on the base category}\label{SubsectionSupportCoactingConditions}

Now we list the conditions on the base category $\mathcal C$ from~\cite[Section 4.2 and Section 5.1]{AGV3} we will refer to in the theorems below:
\begin{enumerate}
	\item\label{PropertySmallLimits} there exist all small limits in $\mathcal C$;
	\item\label{PropertyFiniteAndCountableColimits} there exist finite and countable colimits in $\mathcal C$;
	\item\label{PropertyEpiExtrMonoFactorizations} $\mathcal C$ is (Epi, ExtrMono)-structured;
	\item\label{PropertySubObjectsSmallSet} $\mathcal C$ is wellpowered;	
	\item\label{PropertyMonomorphism} for every monomorphism $f$ and every object $M$
	both
	$f \otimes \id_M$ and $\id_M \otimes f$ are monomorphisms too; 
	\myitem[5a]\label{PropertyExtrMonomorphism} for every extremal monomorphism $f$ the morphism
	$f \otimes f$ is an extremal monomorphism too; 
	\item\label{PropertyLimitsOfSubobjectsArePreserved} for every object $M$ the functor $ M \otimes (-)$ preserves limits (= intersections) of extremal subobjects in $ \mathcal C$;
	\item\label{PropertyTensorPullback}
	for every object~$M$ the functor
	$ M \otimes (-)$
	preserves preimages, i.e.
	for every
	pullback $$\xymatrix{ P\strut \ar[r]^{t} \ar@{>->}[d]_{h} \ar@{}[rd]|<{\pullback} & A\strut \ar@{>->}[d]^f \\
		C \ar[r]^g & B}
	$$
	where $f$ is an arbitrary monomorphism and $g$ is an arbitrary morphism having the same codomain $B$
	(recall that in this case $h$ is automatically a monomorphism too) the diagram below is a pullback too:
	$$
	\xymatrix{ M \otimes P
		\ar[r]^{\id_M \otimes\, t} \ar@{->}[d]_{\id_M \otimes\,  h} \ar@{}[rd]|<{\pullback} & M \otimes A
		\ar@{->}[d]^{\id_M \otimes\,  f} \\
		M \otimes C \ar[r]^{\id_M \otimes\,  g} & M \otimes B}
	$$
	\item\label{PropertySwitchProdTensorIsAMonomorphism}
	for any nonempty small set $\Lambda$ and any objects $M$ and $A_\alpha$, $\alpha \in \Lambda$, the morphism $$\xymatrix{M \otimes \prod\limits_{\alpha\in\Lambda}A_\alpha \ar[rr]^(0.45){({ \id_M} \otimes {\pi_\alpha})_{\alpha\in\Lambda}}
		& \qquad\quad& \prod\limits_{\alpha\in\Lambda} (M \otimes A_\alpha),}$$ where $\pi_\alpha$ is the projection from $\prod\limits_{\alpha\in\Lambda}A_\alpha$ to $A_\alpha$, $\alpha \in \Lambda$,
	is a monomorphism;	
	\item\label{PropertyEqualizers}
	for every object $M$ the functor $ M \otimes (-)$
	preserves all equalizers;
	\item\label{PropertyFreeMonoid} the forgetful functor $\mathsf{Mon}(\mathcal C) \to \mathcal C$
	has a left adjoint $\mathcal F \colon  \mathcal C \to \mathsf{Mon}(\mathcal C)$.
\end{enumerate}
\begin{remark} Property~\ref{PropertyEpiExtrMonoFactorizations}
	follows from Properties~\ref{PropertySmallLimits} and~\ref{PropertySubObjectsSmallSet}.
	Property~\ref{PropertyEqualizers} follows from 	Properties~\ref{PropertySmallLimits}, \ref{PropertyTensorPullback}  and~\ref{PropertySwitchProdTensorIsAMonomorphism} (see e.g.~\cite[Proposition 4.2]{AGV3}).
\end{remark}

\subsection{Supports of morphisms $A \to B \otimes Q$}\label{SubsectionSupportMorTensAB}

Let $\mathcal C$ be a monoidal category.
For given objects $A,B$ in $\mathcal C$ denote by $\mathbf{MorTens}(A,B)$
the comma category $(A\downarrow B\otimes(-))$, i.e.
the category where
\begin{itemize}
	\item the objects are all morphisms $\rho \colon A \to B \otimes Q$ for arbitrary objects $Q$;
	\item the morphisms between $\rho_1 \colon A \to B \otimes Q_1$ and $\rho_2 \colon   A \to  B \otimes Q_2$
	are morphisms $\tau \colon Q_1 \to Q_2$
	making the diagram below commutative:
	$$ \xymatrix{	 A \ar[r]^(0.4){\rho_1} \ar[rd]_{\rho_2} &  B \otimes Q_1 \ar[d]^{\id_{ B} \otimes\, \tau} \\
		&  B \otimes Q_2}$$
\end{itemize}

\begin{remark}
	The category $\mathbf{MorTens}(A,B)$ defined above should not be confused with the category
	$\mathsf{MorTens}(Q)$ defined in Section~\ref{Subsection(Co)Measurings} in order to introduce comeasurings. Both contain
	$\rho \colon A \to B \otimes Q$ as objects, but in $\mathbf{MorTens}(A,B)$ the objects $A$ and $B$ are fixed and 
	in $\mathsf{MorTens}(Q)$ we fix $Q$, the objects $A$ and $B$ may be arbitrary.
\end{remark}	

For morphisms $\rho \colon A \to B \otimes Q$ we are going to use the terminology and the notation from Sections~\ref{SubsectionLIO} and~\ref{SubsectionLIOAbsValue} with respect to $X=\mathbf{MorTens}(A,B)$.

\begin{definition}
	We say that a morphism $\rho \colon A \to B \otimes Q$ is a \textit{tensor epimorphism}
	if $\rho \in \LIO(\mathbf{MorTens}(A,B))$, i.e. if
	for every $f,g \colon Q \to R$, such that
	\begin{equation*}
		({\id_B} \otimes f)\rho = ({\id_B} \otimes g) \rho,\end{equation*}
	we have $f=g$.
\end{definition}

If there exists $|\rho| \colon  A \to B \otimes \tilde Q$ for some $\rho$, then we call the object $\supp \rho := \tilde Q$
the \textit{support} of $\rho$. From the definition of the absolute value it follows that $\supp \rho$ is defined up to an isomorphism compatible with $|\rho|$.

\begin{theorem}[{\cite[Theorem 4.13]{AGV3}}]\label{TheoremAbsValueSupportExistence}
	Let $\mathcal C$ be a monoidal category satisfying Properties~\ref{PropertySmallLimits}, \ref{PropertySubObjectsSmallSet}--\ref{PropertyLimitsOfSubobjectsArePreserved} and \ref{PropertyEqualizers}
	of Section~\ref{SubsectionSupportCoactingConditions}. Then for every objects $A,B$ in $\mathcal C$ there exist absolute values of all objects in the category $\mathbf{MorTens}(A,B)$ and, consequently, there exist supports for all morphisms $\rho \colon A \to B \otimes Q$ in~$\mathcal C$. 
\end{theorem}

\subsection{Cosupports of morphisms $P \otimes A \to B$}

Cosupports of morphisms are introduced in the dual way.

Let $\mathcal C$ be a monoidal category.
For given objects $A,B$ in $\mathcal C$ denote by $\mathbf{TensMor}(A,B)$
the comma category $( (-)\otimes A \downarrow B)$, i.e.
the category where
\begin{itemize}
	\item the objects are all morphisms $\psi \colon  P \otimes A \to B$ for arbitrary objects $P$;
	\item the morphisms between $\psi_1 \colon P_1 \otimes A \to B$ and $\psi_2 \colon P_2 \otimes  A \to B$ are morphisms $\tau \colon P_1 \to P_2$
	making the diagram below commutative:
	$$ \xymatrix{	P_1 \otimes A \ar[d]_{\tau\, \otimes\, \id_{A}} \ar[r]^(0.6){\psi_1}  &  B   \\
		P_2 \otimes A \ar[ru]_{\psi_2} &  }$$
\end{itemize}

For morphisms $\psi \colon P \otimes A \to B$ we are going to use the terminology and the notation from Sections~\ref{SubsectionLIO} and~\ref{SubsectionLIOAbsValue} with respect to $X=\mathbf{TensMor}(A,B)^\mathrm{op}$. 

\begin{definition}
	We say that a morphism $\psi \colon P \otimes A \to B$ is a \textit{tensor monomorphism}
	if $\psi \in \LIO(\mathbf{TensMor}(A,B)^\mathrm{op})$, i.e. if for every $f,g \colon R \to P$ such that
  \begin{equation*}
		\psi (f \otimes {\id_A}) = \psi (g \otimes {\id_A}) \end{equation*}
	we have $f=g$.
\end{definition}

If there exists $|\psi| \colon  \tilde P \otimes A \to B$ for some $\psi$, then we call the object $\cosupp \psi := \tilde P$
the \textit{cosupport} of $\psi$. From the definition of the absolute value it follows that $\cosupp \psi$ is defined up to an isomorphism compatible with $|\psi|$.

\begin{theorem}[{\cite[Theorem 5.15]{AGV3}}]\label{TheoremDUALAbsValueCosupportExistence}
	Let $\mathcal C$ be a monoidal category satisfying the properties dual to Properties~\ref{PropertySmallLimits}, \ref{PropertySubObjectsSmallSet}--\ref{PropertyLimitsOfSubobjectsArePreserved} and \ref{PropertyEqualizers}
	of Section~\ref{SubsectionSupportCoactingConditions}. Then for every objects $A,B$ in $\mathcal C$
	there exist absolute values of
	all objects in the category $\mathbf{TensMor}(A,B)^\mathrm{op}$ and, consequently, there exist cosupports for all morphisms 
	$\psi \colon P \otimes A \to B$ in $\mathcal C$. 
\end{theorem}

\subsection{Universal objects}\label{SubsectionUnivObjects}

\subsubsection{Universal comeasuring monoids and universal measuring comonoids}\label{SubsectionUnivComeasExistence}

Fix $\Omega$-magmas $A$ and $B$ in a braided monoidal category $\mathcal C$. 

Consider the category $\mathbf{Comeas}(A,B)$ where
\begin{itemize}
	\item the objects are all comeasurings $\rho \colon A \to B \otimes Q$ for arbitrary monoids $Q$;
	\item the morphisms from $\rho_1 \colon A \to B \otimes Q_1$
	to $\rho_2 \colon A \to B \otimes Q_2$  are monoid homomorphisms $\varphi \colon Q_1 \to Q_2$
	making the diagram below commutative:
	
	$$\xymatrix{ A \ar[r]^(0.4){\rho_1} 
		\ar[rd]_{\rho_2}
		& B \otimes Q_1 \ar[d]^{\id_B \otimes \varphi} \\
		& B \otimes Q_2} $$
\end{itemize}

Denote by $G_1$ the forgetful functor $\mathbf{Comeas}(A,B)\to \mathbf{MorTens}(A,B)$. Given a tensor epimorphism $\rho_U \colon A \to B \otimes U$ for some object $U$ in $\mathcal C$, let us call the monoid $\mathcal{A}^\square(\rho_U)$ corresponding to the initial object $\rho_U^\mathbf{Comeas} \colon A \to B \otimes \mathcal{A}^\square(\rho_U)$ in $\mathbf{Comeas}(A,B)_{G_1}(\rho_U)$ (if it exists) the \textit{$U$-universal comeasuring monoid} from $A$ to $B$.

Consider the category $\mathbf{Meas}(A,B)$ where
\begin{itemize}
	\item the objects are all measurings $\psi \colon P\otimes A \to B$ for arbitrary comonoids $P$;
	\item the morphisms from $\psi_1 \colon P_1\otimes A \to B$
	to $\psi_2 \colon P_2\otimes A \to B$  are monoid homomorphisms $\varphi \colon P_1 \to P_2$
	making the diagram below commutative:
	\begin{equation*}\xymatrix{
			{P_1} \otimes A \ar[r]^(0.6){\psi_1}  \ar[d]_{\varphi\otimes{\id_A}} & B  \\
			{P_2} \otimes A \ar[ru]_(0.6){\psi_2} &
	}\end{equation*}	
\end{itemize}

Denote by $G_1'$ the forgetful functor $\mathbf{Meas}(A,B)\to \mathbf{TensMor}(A,B)$.
Let $\psi_V \colon V \otimes A \to B$ be a tensor monomorphism  for some object $V$.
We call the comonoid ${}_\square\mathcal{C}(\psi_V)$ corresponding to the initial object $$\psi_V^\mathbf{Meas} \colon {}_\square\mathcal{C}(\psi_V) \otimes A \to B$$ in $\mathbf{Meas}(A,B)^{\mathrm{op}}_{G_1'}(\psi_V)$ (if it exists) the \textit{$V$-universal measuring comonoid} from $A$ to $B$.



\subsubsection{Universal (co)acting bi/Hopf monoids}\label{SubsectionUnivBiHopfCoactingExistence}
Fix an $\Omega$-magma $A$ in a braided monoidal category~$\mathcal C$.

Consider the category $\mathbf{Coact}(A)$ where \begin{itemize}
	\item the objects are all coactions $\rho \colon A \to A \otimes B$ for arbitrary bimonoids $B$;
	\item  the morphisms from $\rho_1 \colon A \to A \otimes B_1$
	to $\rho_2 \colon A \to A \otimes B_2$  are bimonoid homomorphisms $\varphi \colon B_1 \to B_2$
	making the diagram below commutative:
	$$\xymatrix{ A \ar[r]^(0.4){\rho_1} 
		\ar[rd]_{\rho_2}
		& A \otimes B_1 \ar[d]^{{\id_A} \otimes \varphi} \\
		& A \otimes B_2} $$
\end{itemize}

Let $U$ be a comonoid and let a tensor epimorphism $\rho_U \colon A \to A \otimes U$
define on $A$ a structure of a $U$-comodule.  Denote by $G_2$ the forgetful functor 
$\mathbf{Coact}(A) \to \mathbf{ComodStr}(A)$.
Let us call the bimonoid  corresponding to the initial object
in $\mathbf{Coact}(A)_{G_2}(\rho_U)$ (if it exists) the \textit{$U$-universal coacting bimonoid} on $A$.

Consider the category $\mathbf{Act}(A)$ where \begin{itemize}
	\item the objects are all actions $\psi \colon B\otimes A \to A$ for arbitrary bimonoids $B$;
	\item the morphisms from $\psi_1 \colon B_1\otimes A \to A$
	to $\psi_2 \colon B_2\otimes A \to A$  are bimonoid homomorphisms $\varphi \colon B_1 \to B_2$
	making the diagram below commutative:
	\begin{equation*}\xymatrix{
			{B_1} \otimes A \ar[r]^(0.6){\psi_1}  \ar[d]_{\varphi\otimes{\id_A}} & A  \\
			{B_2} \otimes A \ar[ru]_(0.6){\psi_2} &
	}\end{equation*}	
\end{itemize}

Now let $V$ be a monoid and let a tensor monomorphism $\psi_V \colon V \otimes A \to A$
define on $A$ a structure of a $V$-module. Denote by $G_2'$ the forgetful functor 
$\mathbf{Act}(A) \to \mathbf{ModStr}(A)$.
Let us call the bimonoid 
corresponding to the initial object in $\mathbf{Act}(A)^{\mathrm{op}}_{G_2'}(\psi_V)$ (if it exists) the \textit{$V$-universal acting bimonoid} on $A$.

Consider the category $\mathbf{HCoact}(A)$ where \begin{itemize}
	\item  the objects are all coactions $\rho \colon A \to A \otimes H$ for arbitrary Hopf monoids $H$;
	\item  the morphisms from $\rho_1 \colon A \to A \otimes H_1$
	to $\rho_2 \colon A \to A \otimes H_2$  are Hopf monoid homomorphisms $\varphi \colon H_1 \to H_2$
	making the diagram below commutative:
	$$\xymatrix{ A \ar[r]^(0.4){\rho_1} 
		\ar[rd]_{\rho_2}
		& A \otimes H_1 \ar[d]^{\id_A \otimes \varphi} \\
		& A \otimes H_2} $$
\end{itemize}

Let $U$ be a comonoid and let a tensor epimorphism $\rho_U \colon A \to A \otimes U$
define on $A$ a structure of a $U$-comodule.
Denote by $G_4$ the forgetful functor $\mathbf{HCoact}(A)\to \mathbf{Coact}(A)$.
 We call the Hopf monoid 
corresponding to the initial object
in $\mathbf{HCoact}(A)_{G_2G_4}(\rho_U)$ (if it exists) the \textit{$U$-universal coacting Hopf monoid} on $A$.

Consider the category $\mathbf{HAct}(A)$ where \begin{itemize}
	\item the objects are all actions $\psi \colon H\otimes A \to A$ for arbitrary Hopf monoids $H$;
	\item the morphisms from $\psi_1 \colon H_1\otimes A \to A$
	to $\psi_2 \colon H_2\otimes A \to A$  are Hopf monoid homomorphisms $\varphi \colon H_1 \to H_2$
	making the diagram below commutative:
	\begin{equation*}\xymatrix{
			{H_1} \otimes A \ar[r]^(0.6){\psi_1}  \ar[d]_{\varphi\otimes{\id_A}} & A  \\
			{H_2} \otimes A \ar[ru]_(0.6){\psi_2} &
	}\end{equation*}	
\end{itemize}

Let $V$ be a monoid and let a tensor monomorphism $\psi_V \colon V \otimes A \to A$
define on $A$ a structure of a $V$-module. Denote by $G_4'$ the forgetful functor $\mathbf{HAct}(A)\to \mathbf{Act}(A)$.
Let us call the Hopf monoid ${}_\square\mathcal{H}(\psi_V)$
corresponding to the initial object $\psi_V^\mathbf{HAct}$
in $\mathbf{HAct}(A)^{\mathrm{op}}_{G_2'G_4'}(\psi_V)$ (if it exists) the \textit{$V$-universal acting Hopf monoid} on $A$.

\subsection{Monoidal functors}

Let $(\mathcal C, \otimes, \mathbbm{1}, a, l, r)$ and $(\mathcal C^\wr, \otimes^\wr, \mathbbm{1}^\wr, a^\wr, l^\wr, r^\wr)$
be two monoidal categories. Recall that a functor $F \colon \mathcal C \to \mathcal C^\wr$ is called a \textit{(lax) monoidal functor} if there exist fixed natural transformations $J_{A,B} \colon FA \mathbin{\otimes^\wr} FB \to F(A\otimes B)$
and a morphism $\varphi \colon \mathbbm{1}^\wr \to F\mathbbm{1}$ making the diagrams below commutative for all objects $A,B,C$ in $\mathcal C$:

$$
\xymatrix{(FA \mathbin{\otimes^\wr} FB)  \mathbin{\otimes^\wr} FC \ar[rr]^{a^\wr_{FA,FB,FC}}
	\ar[d]_{J_{A,B} \mathbin{\otimes^\wr} \id_{FC}} & &
	FA \mathbin{\otimes^\wr} (FB  \mathbin{\otimes^\wr} FC) 
	\ar[d]^{  \id_{FA} \mathbin{\otimes^\wr} J_{B,C}}
	\\
	F(A \otimes B)  \mathbin{\otimes^\wr} FC
	\ar[d]_{J_{A\otimes B,C}}
	& &
	FA \mathbin{\otimes^\wr} F(B  \otimes C) 
	\ar[d]^{J_{A,B\otimes C}}
	\\
	F((A \otimes B)  \otimes C) \ar[rr]^{Fa_{A,B,C}} & &
	F(A \otimes (B  \otimes C))  \\
}
$$
$$
\xymatrix{ 
	\mathbbm{1}^\wr \mathbin{\otimes^\wr} FA \ar[d]_{\varphi \mathbin{\otimes^\wr} \id_{FA}}
	\ar[r]^{l_{FA}^\wr} & FA \\
	F\mathbbm{1} \mathbin{\otimes^\wr} FA \ar[r]^{J_{\mathbbm{1},A}} & F(\mathbbm{1} \otimes A) \ar[u]_{Fl_A}}
\qquad
\xymatrix{ FA \mathbin{\otimes^\wr} \mathbbm{1}^\wr \ar[d]_{ \id_{FA}  \mathbin{\otimes^\wr} \varphi}
	\ar[r]^{r_{FA}^\wr} & FA \\
	FA \mathbin{\otimes^\wr} 	F\mathbbm{1}  \ar[r]^{J_{A,\mathbbm{1}}} & F( A \otimes \mathbbm{1}) \ar[u]_{Fr_A}}
$$

If both $\mathcal C$ and $\mathcal C^\wr$ are braided with braidings $c$ and $c^\wr$, respectively, then
$F$ is called \textit{braided} if for every objects $A$ and $B$ in $\mathcal C$ the diagram below is commutative:
$$\xymatrix{
	FA \mathbin{\otimes^\wr} FB \ar[rr]^{c^\wr_{FA,FB}} \ar[d]_{J_{A,B}} & & FB \mathbin{\otimes^\wr} FA \ar[d]^{J_{B,A}} \\
	F(A\otimes B) \ar[rr]^{Fc_{A,B}} & & F(B\otimes A)
}$$

A functor $F$ is called \textit{op-monoidal} if $F^\mathrm{op}$ is monoidal. 

If for a monoidal functor $F$ all $J_{A,B}$ and $\varphi$ are isomorphisms, then $F$ is called \textit{strong}
and if all $J_{A,B}$ and $\varphi$ are identity morphisms, then $F$ is called \textit{strict}.

\section{Duality}\label{SectionDuality}

\subsection{Pre-rigid categories}

A monoidal category $(\mathcal C, \otimes, \mathbbm{1}, a, l, r)$ is called \textit{pre-rigid}~\cite{ArdGoyMen2022, GoyVer}
if for every object $A$ in $\mathcal C$ there exists a fixed object $A^*$
and a morphism $\mathrm{ev}_A \colon A^* \otimes A \to \mathbbm{1}$
such that for every object $B$ the map
\begin{equation}\label{EqPrerigidBijection}
\mathcal C(B, A^*) \to \mathcal C(B \otimes A, \mathbbm{1}),\qquad
f \mapsto \mathrm{ev}_A(f\otimes \id_A),
\end{equation}
is a bijection. Note that the correspondence~\eqref{EqPrerigidBijection} is automatically natural in $B$.

\begin{remark}
	Given an object $A$, the object $A^*$ is unique up to an isomorphism compatible with $\mathrm{ev}_A$ since $\mathrm{ev}_A$
	is the terminal object in the comma category $\bigl((-) \otimes A \downarrow \mathbbm{1}\bigr)$.	
\end{remark}

For a given morphism $f \colon A \to B$ in $\mathcal C$ define $f^*$ to be the morphism
$B^* \to A^*$ corresponding to $\mathrm{ev}_B(\id_{B^*}\otimes f)$ under the bijection~\eqref{EqPrerigidBijection}.
Then $(-)^*$ becomes a contravariant functor such that the bijection~\eqref{EqPrerigidBijection} is natural in $A$ too
and $\mathrm{ev}_A$ is dinatural in $A$.

Now we are going to define a monoidal structure on $(-)^*$. 

Let $\mathcal C$ be a pre-rigid braided monoidal category with a braiding $c_{A,B}\colon A\otimes B \to B\otimes A$. For given objects $A$ and $B$ in $\mathcal C$ denote by $\theta_{A,B} \colon A^* \otimes B^* \to (A\otimes B)^*$ the morphism corresponding under the bijection~\eqref{EqPrerigidBijection} to the composition $$\xymatrix{A^* \otimes B^* \otimes A\otimes B \ar[rr]^{{\id_A} \otimes {c^{-1}_{A,B^*}} \otimes {\id_B}}
& \qquad  & A^* \otimes A\otimes B^* \otimes B \ar[rr]^(0.6){{\mathrm{ev}_A} \otimes {\mathrm{ev}_B}} & &
\mathbbm{1} \otimes \mathbbm{1} \ar[r]^{\sim}
& \mathbbm{1}
}$$

By $\iota \colon \mathbbm 1 \to \mathbbm 1^*$ denote the morphism corresponding under the bijection~\eqref{EqPrerigidBijection} to the identification $\mathbbm{1} \otimes \mathbbm{1} \mathrel{\widetilde{\to}} \mathbbm{1}$.

As it was noticed in~\cite[Section 4.3]{ArdGoyMen2022}, the following theorem holds:
\begin{theorem}\label{TheoremStarBraidedLaxMonoidal} $\theta_{A,B}$ is a natural transformation and
	$(-)^* \colon \mathcal C \to \mathcal C^{\mathrm{op}}$ together with $\theta$ and $\iota$ is a braided op-monoidal functor
	where $\mathcal C^{\mathrm{op}}$ is the braided monoidal category opposite to $\mathcal C$ as an ordinary category with the same monoidal product as in $\mathcal C$, the associativity constraint $a'_{A,B,C}:= a^{-1}_{A,B,C}$ and the braiding $c'_{A,B}:= c_{B,A}$.
\end{theorem}

Using the braiding, we get a self-adjunction of $(-)^*$:
\begin{equation}\label{EqDualityAdjunctionStarObjects}\xymatrix{
	\mathcal C(A, B^*) \ar[r]^\sim & \mathcal C(A \otimes B, \mathbbm{1}) \ar[rr]^{\mathcal C(c_{B,A}, \id_\mathbbm{1})} & \quad &
	C(B \otimes A, \mathbbm{1})  \ar[r]^\sim & C(B, A^*)}\end{equation}

For a given morphism $f \colon A \to B^*$ define $f^\sharp \colon B \to A^*$ as the image
of $f$ under the composition of bijective natural transformations~\eqref{EqDualityAdjunctionStarObjects}.

Analogously, for a morphism
$g \colon B \to A^*$ define $g^\flat \colon A \to B^*$ as the preimage of $g$ under the composition~\eqref{EqDualityAdjunctionStarObjects} above. Obviously, $f=(f^{\sharp})^\flat$ for every $f \colon A \to B^*$
and $g=(g^{\flat})^\sharp$ for every $g \colon B \to A^*$.

Denote by $\alpha_A \colon A\to A^{**}$ the counit of the adjunction~\eqref{EqDualityAdjunctionStarObjects}, i.e. $\alpha_A := \id_{A^*}^\flat$ for every object $A$ in $\mathcal C$.

The functor  $(-)^*$ induces a contravariant functor $$\mathsf{Comon}(\mathcal C) \to \mathsf{Mon}(\mathcal C),\qquad
(C,\Delta,\varepsilon) \mapsto (C^*, \mu, u)$$ where $\mu = \Delta^*\theta_{C,C}$, $u=\varepsilon^*\iota$.
We denote this functor again by $(-)^*$. 

\begin{example}
One of the important examples of pre-rigid categories are closed monoidal categories (we refer to Section~\ref{SectionApplications} for concrete examples).
A monoidal category $\mathcal C$ is \textit{closed} if
for every object $A$ in $\mathcal C$ the functor $(-) \otimes A$ has a left adjoint $[A,-]$, which is called the \textit{internal hom}. Note that the functor $ (-) \otimes \mathbbm{1}$
is isomorphic to the identical functor $\id_{\mathcal C}$. Hence we may assume that $[\mathbbm{1},-] = \id_{\mathcal C}$ too.
Since $\otimes$ is a bifunctor  $\mathcal C \times \mathcal C \to \mathcal C$,  the adjunction
$$\mathcal C (A \otimes B, C) \cong \mathcal C (A, [B, C])$$
 defines a bifunctor $[-,-] \colon \mathcal C^\mathrm{op} \times \mathcal C \to \mathcal C$ in a unique way.
  Let $\mathrm{ev}_{B,A} \colon [B,A]\otimes B \to A$ be the counit of this adjunction.
Then $\mathcal C$ is pre-rigid with respect to $(-)^*:= [-,\mathbbm{1}]$ and $\mathrm{ev}_A := \mathrm{ev}_{A,\mathbbm{1}}$.
\end{example}

\begin{lemma}\label{LemmaFFlatEvBraiding} For every morphism $f \colon P \to U^*$ the diagram below is commutative:
	$$ 
	\xymatrix{ P\otimes U \ar[d]_(0.4){\id_P \otimes f^\flat} \ar[rr]^{f\otimes \id_U} & & U^*\otimes U \ar[d]^{\mathrm{ev}_U} \\
		P\otimes P^* \ar[r]^{c_{P,P^*}} & P^* \otimes P \ar[r]^{\mathrm{ev}_P} & \mathbbm{1}\\
	}
	$$
\end{lemma}
\begin{proof}
	Consider the following diagram:
	$$ 
\xymatrix{ P\otimes U \ar[dd]^(0.4){\id_P \otimes f^\flat} \ar[rr]^{f\otimes \id_U} \ar[rd]^{c_{P,U}} & & U^*\otimes U \ar[dd]^{\mathrm{ev}_U} \\
	& U\otimes P \ar[d]^{ f^\flat  \otimes \id_P} & \\
	P\otimes P^* \ar[r]^{c_{P,P^*}} & P^* \otimes P \ar[r]^{\mathrm{ev}_P} & \mathbbm{1}\\
}
$$
The polygon on the right is commutative by the definition of $(-)^\flat$.
The lower left quadrilateral is commutative by the naturality of the braiding.

Therefore the outer rectangle is commutative too and the lemma is proven.	
\end{proof}	

\subsection{Correspondence between morphisms}

Now we introduce several maps defined on morphisms in a pre-rigid braided monoidal category $\mathcal C$ and prove some of their properties.

For a given morphism $\rho \colon A \to B \otimes Q$, where $A,B,Q$ are some objects, define the morphism $\rho^{\vee} \colon Q^* \otimes A \to B$ by the commutative diagram
$$\xymatrix{   Q^* \otimes B \otimes Q \ar[rr]^{{c_{Q^*,B}} \otimes {\id_{Q}}} & &  B \otimes Q^* \otimes Q \ar[rr]^{ \id_B \otimes \mathrm{ev}_{Q}} & &
	B \otimes \mathbbm{1}  \ar[d]^{\sim}  \\
	Q^* \otimes A \ar[u]^{\id_{Q^*}\otimes \rho} \ar@{-->}[rrrr]^{\rho^\vee}& & & & B
}$$

For a given morphism $\rho \colon A \to B \otimes P^*$, where $A,B,P$ are some objects,
define the morphism $\rho^\nabla \colon P \otimes A \to B$
 by the commutative diagram
$$\xymatrix{  P \otimes B \otimes P^* \ar[rr]^{c_{P,B \otimes  P^*}} & & B  \otimes P^* \otimes  P \ar[rr]^{ \id_B \otimes \mathrm{ev}_{P}} & & B  \otimes \mathbbm{1} \ar[d]^\sim \\
	P \otimes A \ar@{-->}[rrrr]^{\rho^\nabla} \ar[u]^{\id_{P}\otimes \rho} & & & & B}$$

\begin{proposition}\label{PropositionDualityNablaVeeSharp} Let $\rho \colon A \to B \otimes U$ and $f \colon P \to U^*$ be some morphisms in $\mathcal C$.
	Then 
	$$\bigl((\id_B \otimes f^\flat)\rho\bigr)^\nabla = \rho^\vee (f \otimes \id_A).$$
\end{proposition}
\begin{proof} Consider the diagram below:
	$$\xymatrix{ P \otimes A \ar[rr]^{f\otimes \id_A} \ar[d]_{\id_P \otimes \rho} & & U^* \otimes A \ar[rr]^(0.4){\id_{U^*}\otimes \rho} & & U^* \otimes B \otimes U \ar[d]^{c_{U^*,B}\otimes \id_U} \\
		P \otimes B \otimes U \ar[urrrr]^(0.35){f\otimes \id_{B\otimes U}}
		\ar[rr]_{c_{P,B}\otimes \id_U}
			\ar[d]_{\id_{P\otimes B}\otimes f^\flat}
		 & & B \otimes P \otimes U \ar[rr]^{\id_B \otimes f \otimes \id_U}
		 \ar[d]^{\id_{B\otimes P}\otimes f^\flat}
		  & & B \otimes U^* \otimes U \ar[dd]^{\id_B \otimes
		 \mathrm{ev}_U} \\
		P \otimes B \otimes P^* \ar[rr]^{c_{P,B}\otimes \id_{P^*}} \ar[d]_{c_{P, B\otimes P^*}} & & B \otimes P \otimes P^*
		\ar[lld]^(0.3){\quad \id_B\otimes c_{P,P^*}} 
		 \\
		B \otimes P^* \otimes P \ar[rrrr]^(0.7){\id_B \otimes \mathrm{ev}_P} & & & & B \otimes \mathbbm{1} \ar[d]^\sim \\
	                          	& &                      &  & B
	}$$

The upper left triangle and the left square in the middle are commutative by the functoriality of $\otimes$,
the upper right triangle is commutative by the naturality of the braiding. The lower left triangle is commutative by the corresponding property of the braiding and the lower right polygon is commutative by Lemma~\ref{LemmaFFlatEvBraiding}.
	Therefore the outer square is commutative too, which implies the proposition since the compositions on the left and on the right equal
	$\bigl((\id_B \otimes f^\flat)\rho\bigr)^\nabla$ and $\rho^\vee (f \otimes \id_A)$, respectively.
\end{proof}	
The proposition below can be proved directly by a similar diagram chase, however now it is easier to deduce it from the proposition above:
\begin{proposition}\label{PropositionVeeUnderComposition}
	Let $\rho \colon A \to B \otimes P$ and $f \colon P \to U$ be some morphisms in $\mathcal C$.
	Then $$\bigl((\id_B \otimes f)\rho\bigr)^\vee = \rho^\vee(f^* \otimes \id_A).$$
\end{proposition}
\begin{proof} By the naturality of $(-)^\flat$ we have 
	$$f^{*\flat} = (f^{*}\id_{U^*})^\flat = (\id_{U^*})^\flat f.$$
	Applying Proposition~\ref{PropositionDualityNablaVeeSharp} two times, we get
	\begin{equation*}\begin{split}
	\rho^\vee(f^* \otimes \id_A) =\bigl((\id_B \otimes f^{*\flat})\rho\bigr)^\nabla
	= \bigl((\id_B \otimes  (\id_{U^*})^\flat f)\rho\bigr)^\nabla
	\\= \bigl((\id_B \otimes (\id_{U^*})^\flat)(\id_B \otimes f)\rho\bigr)^\nabla
	= ((\id_B \otimes f)\rho\bigr)^\vee ( \id_{U^*} \otimes \id_A) = ((\id_B \otimes f)\rho\bigr)^\vee.
	\end{split}
	\end{equation*}
\end{proof}		

Below we deduce from the universal property of $\mathrm{ev}$ some cancellation property for morphisms. Recall that $\alpha_B \colon B \to B^{**}$ is a monomorphism for every object $B$ in $\mathcal C$ if and only if the functor $(-)^*$ is faithful.

\begin{lemma}\label{LemmaRhoABPStarPartialTensoring} Let $\rho_1, \rho_2 \colon A \to B \otimes P^*$ be morphisms in $\mathcal C$ for some objects $A,B,P$
	such that $\alpha_B \otimes \id_{P^*}$ and $\theta_{B^*,P}$ are monomorphisms
	and \begin{equation}\label{EqRhoABPStarPartialTensoring}
	(\rho_1 \otimes \id_{P})(\id_B \otimes \mathrm{ev}_{P}) = (\rho_2 \otimes \id_{P})(\id_B \otimes \mathrm{ev}_{P}).
	\end{equation}
	Then $\rho_1 = \rho_2$.
\end{lemma}	
\begin{proof}
	Consider the following diagram:
 $$\xymatrix{
	A \otimes B^* \otimes P 
	\ar[d]^{\id_{A}\otimes c_{B^*,P}} \ar@<2pt>[rrrr]^{\rho_1 \otimes \id_{B^*\otimes P}} \ar@<-2pt>[rrrr]_{\rho_2\otimes \id_{B^*\otimes P}}
	& & 
		& & B \otimes P^* \otimes B^* \otimes P \ar[ddd]^{\alpha_B \otimes \id_{P^*\otimes B^* \otimes P}} 
	\ar[lllldd]^{\qquad\id_{B\otimes P^*}\otimes c_{B^*,P}}
	\\
	A \otimes P \otimes B^* 
	\ar@<2pt>[d]^{\rho_2\otimes \id_{P\otimes B^*}}
	\ar@<-2pt>[d]_{\rho_1\otimes \id_{P\otimes B^*}}\\
	B \otimes P^* \otimes P \otimes B^* \ar[d]^{\id_B \otimes \mathrm{ev}_P \otimes \id_{B^*}}
	\ar[rrd]^{\quad\alpha_B \otimes \id_{P^* \otimes P \otimes B^*}}
	\\
	B  \otimes \mathbbm{1} \otimes B^* \ar[d]^{\alpha_B \otimes \id_{\mathbbm 1 \otimes B^*}} & & B^{**}  \otimes P^* \otimes P \otimes B^* \ar[lld]^(0.45){\qquad \id_{B^{**}}\otimes \mathrm{ev}_P \otimes \id_{B^*}}
	\ar[d]^(0.4){\id_{B^{**}}\otimes c^{-1}_{B^*, P^* \otimes P}}
	&\quad & B^{**}  \otimes P^* \otimes B^* \otimes P
	\ar[lld]^(0.4){\qquad \id_{B^{**}}\otimes c^{-1}_{B^*, P^*} \otimes \id_P}
	\ar[ll]_(0.55){\qquad\id_{B^{**}\otimes P^*}\otimes c_{B^*,P}}
	\ar[d]^{\theta_{B^*, P} \otimes \id_{B^* \otimes P}}
	\\
	B^{**}  \otimes \mathbbm{1} \otimes B^* \ar[d]^(0.4){\id_{B^{**}}\otimes c^{-1}_{B^*, \mathbbm{1}}} & & B^{**} \otimes B^*  \otimes P^* \otimes P \ar[lld]^(0.45){\qquad\id_{B^{**} \otimes B^*}\otimes \mathrm{ev}_P}
	\ar[d]^{\mathrm{ev}_{B^*} \otimes \mathrm{ev}_P}
	& & (B^{*}  \otimes P)^* \otimes B^* \otimes P \ar[d]^{\mathrm{ev}_{B^{*}  \otimes P}}  \\
	B^{**} \otimes B^*  \otimes \mathbbm{1} \ar[rr]_{\mathrm{ev}_{B^*}\otimes \id_{\mathbbm 1}} &  &
	\mathbbm{1}  \otimes \mathbbm{1} \ar[rr]^\sim & &
	\mathbbm{1} \\
}$$

By~\eqref{EqRhoABPStarPartialTensoring} the composition on the left and lower edges does not depend on whether we take $\rho_1$ or $\rho_2$.
The lower left quadrilateral is commutative by the naturality of the braiding.
The triangle in the lower right part of the diagram is commutative by the corresponding property of the braiding.
The lower right polygon is commutative by the definition of $\theta$.
The other inner polygons are commutative by the functoriality of $\otimes$.
Therefore, the composition on the right and upper edges does not depend on whether we take $\rho_1$ or $\rho_2$ either.
Now recall that both $\alpha_B\otimes \id_{P^*}$ and $\theta_{B^*,P}$ are monomorphisms. Hence the universal property of $\mathrm{ev}$
implies that $\rho_1=\rho_2$.
\end{proof}	
\begin{corollary}\label{CorollaryRhoABPStarEqualAfterNabla} Let $\rho_1, \rho_2 \colon A \to B \otimes P^*$ be morphisms in $\mathcal C$ for some objects $A,B,P$
	such that $\alpha_B \otimes \id_{P^*}$ and $\theta_{B^*,P}$ are monomorphisms
	and $\rho_1^\nabla = \rho_2^\nabla$.
	Then $\rho_1 = \rho_2$.
\end{corollary}	
\begin{proposition}\label{PropositionDualToTensorEpiIsMono}
	Let $\rho \colon A \to B \otimes U$
	be a tensor epimorphism in~$\mathcal C$ such that $\alpha_B \otimes \id_{P^*}$ and $\theta_{B^*,P}$ are monomorphisms for all objects $P$. Then $\rho^\vee \colon U^* \otimes A \to B$ is a tensor monomorphism.
\end{proposition}	
\begin{proof} Let $f_1, f_2 \colon P \to U^*$ be two morphisms such that 
	$\rho^\vee(f_1 \otimes \id_A) = \rho^\vee(f_2 \otimes \id_A)$. We claim that $f_1 = f_2$.
	
By Proposition~\ref{PropositionDualityNablaVeeSharp}
we have
$$ \bigl((\id_B \otimes f_1^\flat)\rho\bigr)^\nabla = \rho^\vee(f_1 \otimes \id_A) = \rho^\vee(f_2 \otimes \id_A)
= \bigl((\id_B \otimes f_2^\flat)\rho\bigr)^\nabla.$$
Now Corollary~\ref{CorollaryRhoABPStarEqualAfterNabla} implies that $(\id_{B}\otimes f_1^\flat)\rho=(\id_{B}\otimes f_2^\flat)\rho$. Thus $f_1^\flat=f_2^\flat$ since $\rho$ is a tensor epimorphism. Therefore $f_1=f_2$ and $\rho^\vee$ is a tensor monomorphism.
\end{proof}	

Below we prove the commutativity of diagrams that relate operations in a comonoid $P$ and the monoid $P^*$
dual to $P$:

\begin{lemma}\label{LemmaPStarPStarPDifferentWays}
	Let $(P,\Delta,\varepsilon)$ be a comonoid in~$\mathcal C$
	and let $(P^*,\mu,u)$ be the dual monoid. Then the diagrams below are commutative:
	\begin{equation}\label{LemmaPStarPStarPDifferentWaysUnity}\xymatrix{
		\mathbbm{1}\otimes P \ar[r]^(0.6)\sim \ar[d]_{u \otimes \id_P} & P  \ar[d]^\varepsilon\\
		P^* \otimes P \ar[r]^(0.6){\mathrm{ev}_{P}} &  \mathbbm{1}\\
}\end{equation}
	\begin{equation}\label{LemmaPStarPStarPDifferentWaysMul}\xymatrix{P^* \otimes P^* \otimes P \ar[rr]^(0.45){\id_{P^*\otimes P^*} \otimes \Delta}
		\ar[dd]_{\mu \otimes \id_P}
		& & P^* \otimes P^* \otimes P \otimes P \ar[d]^{\theta_{P,P} \otimes \id_{P\otimes P}} \\
		& &  (P \otimes P)^* \otimes P \otimes P \ar[d]^{\mathrm{ev}_{P \otimes P}} \\
		P^* \otimes P \ar[rr]^{\mathrm{ev}_P } & & \mathbbm{1}
	}\end{equation}
\end{lemma}
\begin{proof}
	Consider the diagram
	$$ \xymatrix{
		\mathbbm{1}\otimes P \ar[rdd]^{\iota\otimes \id_P} \ar[rrrd]^{\mathbbm{1}\otimes \varepsilon} \ar[rrrr]^\sim \ar[ddd]_{u \otimes \id_P} & & & & P  \ar[ddd]^\varepsilon\\
		& & &  \mathbbm{1} \otimes  \mathbbm{1} \ar[d]_{\iota \otimes \mathbbm{1}} \ar[rdd]^\sim & \\
		& \mathbbm{1}^*\otimes P \ar[ld]_(0.4){\varepsilon^* \otimes \id_P} \ar[rr]^{\id_{\mathbbm{1}^*} \otimes \varepsilon} & & \mathbbm{1}^*\otimes \mathbbm{1}
		\ar[rd]_{\mathrm{ev}_{\mathbbm{1}}} & \\
		P^* \otimes P  \ar[rrrr]^{\mathrm{ev}_{P}} & & & & \mathbbm{1}\\
	}
	$$
	The triangle on the left is commutative by the definition of $u$, the lower quadrilateral is commutative by the dinaturality
	of $\mathrm{ev}$, the central quadrilateral is commutative by the functoriality of $\otimes$, the upper right quadrilateral is commutative by the naturality of the transformation $\mathbbm{1} \otimes  M  \mathrel{\widetilde\to} M$,
	the lower right triangle is commutative by the definition of $\iota$.
	Hence the outer square is commutative too, which yields the commutativity of~\eqref{LemmaPStarPStarPDifferentWaysUnity}.
	
	Now consider the diagram
	$$\xymatrix{P^* \otimes P^* \otimes P \ar[rrr]^{\id_{P^*\otimes P^*} \otimes \Delta}
		\ar[d]_{\theta_{P,P} \otimes \id_P}
		& & & P^* \otimes P^* \otimes P \otimes P 
		\ar[d]^{\theta_{P,P} \otimes \id_{P\otimes P}\quad}
		\\
		(P \otimes P)^* \otimes P \ar[d]_{\Delta^* \otimes \id_P}
		\ar[rrr]^{\id_{(P \otimes P)^*}\otimes \Delta} & & & (P \otimes P)^* \otimes P \otimes P
		\ar[d]^{\mathrm{ev}_{P\otimes P}}
		\\
		P^* \otimes P \ar[rrr]^{\mathrm{ev}_P } & & & \mathbbm{1}
	}$$
	The upper square is commutative by the functoriality of $\otimes$.
	The lower square is commutative by the dinaturality of $\mathrm{ev}$.
	Hence the outer square is commutative too, which yields the commutativity of~\eqref{LemmaPStarPStarPDifferentWaysMul}.
\end{proof}

\begin{lemma}\label{LemmaDiagramsPPStarOperations}
Let $(P,\Delta,\varepsilon)$ be a comonoid in~$\mathcal C$
and let $(P^*,\mu,u)$ be the dual monoid.
Then the diagrams~\eqref{EqDiagramsPPStarOperations(Co)unit} and~\eqref{EqDiagramsPPStarOperations(Co)multiplication} below are commutative:
\begin{equation}\label{EqDiagramsPPStarOperations(Co)unit}
\xymatrix{
P \otimes \mathbbm{1} \ar[r]^\sim \ar[d]_{\id_P \otimes u} & P \ar[r]^{\varepsilon}& \mathbbm{1} \\
P\otimes P^* \ar[d]_\sim & & \mathbbm{1} \otimes \mathbbm{1} \ar[u]_\sim \\
P\otimes \mathbbm{1} \otimes  P^* \ar[rr]^{c_{P,\mathbbm{1}\otimes P^*}} & & \mathbbm{1} \otimes  P^* \otimes P
\ar[u]_{\id_{\mathbbm{1}}\otimes \mathrm{ev}_P} \\
}
\end{equation}
\begin{equation}\label{EqDiagramsPPStarOperations(Co)multiplication}
\xymatrix{P \otimes P^* \otimes P^* \ar[rrrrr]^{\id_P \otimes \mu}
	\ar[d]_{\Delta \otimes \id_{P^* \otimes P^*}}
	 & & & & & P\otimes P^* \ar[d]^{c_{P,P^*}}\\
P \otimes P \otimes P^* \otimes P^* \ar[d]_{\id_P \otimes c_{P,P^*} \otimes \id_{P^*}} & & & & & P^*\otimes P \ar[d]^{\mathrm{ev}_P} \\
P \otimes P^* \otimes P \otimes P^*\ar[rr]^{c_{P,P^*} \otimes c_{P,P^*}} & & P^* \otimes P \otimes P^* \otimes P
\ar[rr]^(0.6){\mathrm{ev}_P \otimes \mathrm{ev}_P}
 & & \mathbbm{1} \otimes \mathbbm{1} \ar[r]^\sim & \mathbbm{1}
}
\end{equation}
\end{lemma}
\begin{proof} First, consider the diagram
$$\xymatrix{
	P \otimes \mathbbm{1} \ar[r]^\sim \ar[rd]^{c_{P,\mathbbm{1}}} \ar[dd]_{\id_P \otimes u} & P \ar[r]^{\varepsilon} & \mathbbm{1} \\
	 & \mathbbm{1} \otimes  P \ar[u]^\sim \ar[d]_{u\otimes \id_P} \\
	P\otimes P^* \ar[d]_\sim \ar[r]^{c_{P,P^*}} &  P^* \otimes P \ar[rd]^\sim \ar[ruu]_{\mathrm{ev}_P} & \mathbbm{1} \otimes \mathbbm{1} \ar[uu]_\sim \\ 
	P\otimes \mathbbm{1} \otimes  P^* \ar[rr]^{c_{P,\mathbbm{1}\otimes P^*}} & & \mathbbm{1} \otimes  P^* \otimes P
	\ar[u]_{\id_{\mathbbm{1}}\otimes \mathrm{ev}_P} \\
}$$	
The inner polygons are commutative by the properties of the braiding and Lemma~\ref{LemmaPStarPStarPDifferentWays}. Therefore, the outer square is commutative too, which yields the commutativity of~\eqref{EqDiagramsPPStarOperations(Co)unit}.

Now consider the diagram
\begin{equation*}
\xymatrix{P \otimes P^* \otimes P^* \ar[rrd]^{c_{P, P^* \otimes P^*}}
		 \ar[rrrrr]^{\id_P \otimes \mu}
	\ar[dd]^{\Delta \otimes \id_{P^* \otimes P^*}}
	& & & & & P\otimes P^* \ar[d]^{c_{P,P^*}}
	 \\
	& & P^* \otimes P^* \otimes P \ar[d]_{\id_{P^* \otimes P^*} \otimes \Delta} \ar[rrr]^{\mu \otimes \id_P} & & & P^*\otimes P \ar[dd]^{\mathrm{ev}_P}\\
	P \otimes P \otimes P^* \otimes P^* \ar[d]^{\id_P \otimes c_{P,P^*} \otimes \id_{P^*}} 
	\ar[rr]^{c_{P \otimes P, P^* \otimes P^*}}
	& &
	 P^* \otimes P^* \otimes P \otimes P \ar[d]^{\id_{P^*} \otimes c^{-1}_{P,P^*} \otimes \id_P} \ar[rr]^{\theta_{P,P} \otimes \id_{P\otimes P}}
	 & & (P \otimes P)^* \otimes P \otimes P \ar[rd]^{\mathrm{ev}_{P\otimes P}} &  \\
	P \otimes P^* \otimes P \otimes P^*\ar[rr]^{c_{P,P^*} \otimes c_{P,P^*}} & & P^* \otimes P \otimes P^* \otimes P
	\ar[rr]^(0.6){\mathrm{ev}_P \otimes \mathrm{ev}_P}
	& & \mathbbm{1} \otimes \mathbbm{1} \ar[r]^\sim & \mathbbm{1}
}
\end{equation*}

The inner polygons are commutative by Lemma~\ref{LemmaPStarPStarPDifferentWays}, the definition of $\theta$ and the properties of the braiding. Therefore, the outer square is commutative too, which yields the commutativity of~\eqref{EqDiagramsPPStarOperations(Co)multiplication}.
\end{proof}	

Applying properties of the braiding, we deduce Lemma~\ref{LemmaNablaStrictMonoidalFunctor} below
from Lemma~\ref{LemmaDiagramsPPStarOperations}:

\begin{lemma}\label{LemmaNablaStrictMonoidalFunctor} For every comonoid $P$
	$$(-)^\nabla \colon \mathsf{MorTens}(P^*) \to \mathsf{TensMor}(P)$$
	is a strict monoidal functor.
	\end{lemma}

\begin{corollary}\label{CorrollaryDiagramsPPStarOperations} Let $\rho \colon A \to B \otimes P^*$ be a morphism in $\mathcal C$ for some comonoid $P$
	and objects $A$ and $B$. Then for every $m\in\mathbb Z_+$ the diagram below is commutative:
	$$\xymatrix{
		P \otimes A^{\otimes m} \ar[dd]_{\id_P \otimes \rho^{{}\mathbin{\widetilde\otimes} m}} \ar[rr]^{\left(\rho^\nabla\right)^{{}\mathbin{\widetilde\otimes} m}} & & B^{\otimes m}\\
		& & B^{\otimes m} \otimes \mathbbm{1} \ar[u]_{\sim}\\
		P \otimes B^{\otimes m} \otimes P^* \ar[rr]^{c_{P, B^{\otimes m} \otimes P^*} } &\qquad &  B^{\otimes m} \otimes P^* \otimes P
		\ar[u]_{\id_{B^{\otimes m}} \otimes \mathrm{ev}_P}
	}$$	
	where in the expressions $\left(\rho^\nabla\right)^{{}\mathbin{\widetilde\otimes} m}$ and $ \rho^{{}\mathbin{\widetilde\otimes} m}$ we use the monoidal products in the categories $\mathsf{TensMor}(P)$ and $\mathsf{MorTens}(P^*)$ defined in Section~\ref{Subsection(Co)Measurings}.
\end{corollary}

\begin{lemma}\label{LemmaNablaMeasComeas}
	Let $\rho \colon A \to B \otimes P^*$ be a morphism for some $\Omega$-magmas $A$ and $B$
	and comonoid $P$ in~$\mathcal C$.
	
	I. If $\rho$ is a comeasuring, then 
	$\rho^\nabla \colon P \otimes A \to B$ is a measuring.
	
	II. If $\rho^\nabla$ is a measuring and $\alpha_{B^{{} \otimes t(\omega)}} \otimes \id_{P^*}$ and $\theta_{\left(B^{{} \otimes t(\omega)}\right)^*,P}$ are monomorphisms for every $\omega\in\Omega$, then $\rho$ is a comeasuring.
\end{lemma}	
\begin{proof}
		For every $\omega \in \Omega$ consider the following diagram where $m=s(\omega)$ and $n=t(\omega)$:
	\begin{equation}\label{EqMeasComeasBijection}
	\hspace{-1cm}\xymatrix{
		P \otimes A^{\otimes m} \ar[ddddd]^{\id_P \otimes \omega_A} \ar[rdd]^{\id_P \otimes \rho^{{}\mathbin{\widetilde\otimes} m}} \ar[rrrrr]^{\left(\rho^\nabla\right)^{{}\mathbin{\widetilde\otimes} m}} & & & & & B^{\otimes m} \ar[ddddd]^{\omega_B}\\
		& & & & B^{\otimes m} \otimes \mathbbm{1} \ar[ru]_{\sim} \ar[ddd]^{\omega_B \otimes \id_\mathbbm{1}}\\
		& P \otimes B^{\otimes m} \otimes P^* 
		\ar[d]^{\id_P \otimes \omega_B \otimes \id_{P^*}}
		\ar[rr]^{c_{P, B^{\otimes m} \otimes P^*}} &\qquad &  B^{\otimes m} \otimes P^* \otimes P
		\ar[d]^{\omega_B \otimes \id_{P^* \otimes P}}
		\ar[ru]^{\id_{B^{\otimes m}} \otimes \mathrm{ev}_P}\\
		& P \otimes B^{\otimes n} \otimes P^* \ar[rr]^{c_{P, B^{\otimes n} \otimes P^*}} &\qquad &  B^{\otimes n} \otimes P^* \otimes P
		\ar[rd]_{\id_{B^{\otimes n}} \otimes \mathrm{ev}_P}\\
		& & & & B^{\otimes n} \otimes \mathbbm{1} \ar[rd]^{\sim}\\
		P \otimes A^{\otimes n} \ar[ruu]_{\id_P \otimes \rho^{{}\mathbin{\widetilde\otimes} n}} \ar[rrrrr]^{\left(\rho^\nabla\right)^{{}\mathbin{\widetilde\otimes} n}} & & & & & B^{\otimes n}\\
	}\end{equation}	
	
	The upper and the lower quadrilaterals are commutative by Corollary~\ref{CorrollaryDiagramsPPStarOperations}.
	The central square is commutative by the naturality of the braiding. The two quadrilaterals on the right are commutative
	by the functoriality of $\otimes$ and the naturality of the transformation $\xymatrix{C\otimes \mathbbm{1} \ar[r]^(0.6)\sim & C}$.
	
	Now we notice that the left quadrilateral of~\eqref{EqMeasComeasBijection} is the diagram~\eqref{EqRhoPStarComeasuring} below tensored by~$P$ from the left:
	
	\begin{equation}\label{EqRhoPStarComeasuring}
	\xymatrix{ A^{\otimes m} \ar[d]_{\omega_A} \ar[r]^(0.4){\rho^{{}\mathbin{\widetilde\otimes} m}}&  B^{\otimes m} \otimes P^* \ar[d]^{\omega_B \otimes \id_{P^*}}  \\
		A^{\otimes n} \ar[r]^(0.4){\rho^{{}\mathbin{\widetilde\otimes} n}} &  B^{\otimes n} \otimes P^*
	}
	\end{equation}
	
	Hence if $\rho$ is a comeasuring, then~\eqref{EqRhoPStarComeasuring}
	is commutative, the outer square of~\eqref{EqMeasComeasBijection} is commutative too,
	$\rho^\nabla \colon P \otimes A \to B$ is a measuring, and Part I is proved.
	
	Suppose now that $\rho^\nabla$ is a measuring and $\alpha_{B^{\otimes n}} \otimes \id_{P^*}$ and $\theta_{\left(B^{\otimes n}\right)^*,P}$ are monomorphisms. Then the outer square of~\eqref{EqMeasComeasBijection} is commutative and, by the naturality of the braiding and Lemma~\ref{LemmaRhoABPStarPartialTensoring}, \eqref{EqRhoPStarComeasuring}
	is commutative for every $\omega\in\Omega$, i.e. $\rho$ is a comeasuring.
\end{proof}	

Now we are ready to prove that under some conditions the map $\rho \mapsto \rho^\nabla$ is a bijection:

\begin{lemma}\label{Lemma(Co)monMeasComeasBijection}
	 Let $\rho_U \colon A \to B \otimes U$
be a tensor  epimorphism in~$\mathcal C$
for 
$\Omega$-magmas $A$ and $B$
and an object $U$. Suppose $\alpha_M \otimes \id_N$ and $\theta_{M,N}$ are monomorphisms
for all objects $M,N$. For a given comonoid $P$ in~$\mathcal C$
denote by $ \mathrm{Meas}(P,\rho_U^\vee)$
	the set of measurings $\psi \colon P \otimes A \to B$  such that there exists a morphism $\psi \to \rho_U^\vee$
	in $\mathbf{TensMor}(A,B)$
	and by $\mathrm{Comeas}(P^*,\rho_U)$
	the set of comeasurings 
	$\rho \colon A \to B \otimes P^*$
	  such that there exists a morphism $\rho_U \to \rho$ in $\mathbf{MorTens}(A,B)$.
	  Then the map $\rho \mapsto \rho^\nabla$ defines a bijection
	$\mathrm{Comeas}(P^*,\rho_U) \cong \mathrm{Meas}(P,\rho_U^\vee)$
	natural in the comonoid $P$ if we regard $\mathrm{Meas}(-,\rho_U^\vee)$ and $\mathrm{Comeas}((-)^*,\rho_U)$
	as functors $\mathsf{Comon}(\mathcal C) \to \mathbf{Sets}$.
\end{lemma}
\begin{proof} By Proposition~\ref{PropositionDualToTensorEpiIsMono}, 
	$\rho_U^\vee$ is a tensor monomorphism. Hence $\mathrm{Meas}(P,\rho_U^\vee)$ can 
	be identified with the subset of $\mathcal C(P,U^*)$ consisting 
	of such morphisms $g \colon P \to U^*$ that $\psi=\rho_U^\vee(g\otimes \id_A)$
	is a measuring. Analogously, $\mathrm{Comeas}(P^*,\rho_U)$ can 
	be identified with the subset of $\mathcal C(U,P^*)$ consisting 
	of such morphisms $f \colon U \to P^*$ that $\rho=(\id_B \otimes f) \rho_U$
	is a comeasuring. By Proposition~\ref{PropositionDualityNablaVeeSharp},
	\begin{equation}\label{Eq(Co)monMeasComeasBijection}
	\rho^\nabla=\bigl((\id_B \otimes f) \rho_U\bigr)^\nabla = 
	\rho_U^\vee (f^\sharp \otimes \id_B).
	\end{equation}
		 By Lemma~\ref{LemmaNablaMeasComeas}
	a morphism
	 $\rho \colon A \to B \otimes P^*$
	is a comeasuring if and only if $\rho^\nabla \colon P\otimes A \to B$ is a measuring.
		Now~\eqref{Eq(Co)monMeasComeasBijection} implies that
		the map $(-)^\nabla$ is a restriction of the natural bijection $(-)^\sharp \colon \mathcal C(U,P^*) \mathrel{\widetilde\to} \mathcal C(P,U^*)$.
\end{proof}	

\subsection{Correspondence between supports and cosupports}\label{SubsectionCorrespondenceSuppCosupp}

The naturality of the bijection $(-)^\flat$ implies Proposition~\ref{PropositionEpimorphismStarIsAMonomorphism} below:

\begin{proposition}\label{PropositionEpimorphismStarIsAMonomorphism}
	Let $f \colon A \to B$ be an epimorphism in a pre-rigid braided monoidal category $\mathcal C$. Then $f^* \colon B^* \to A^*$
	is a monomorphism.
\end{proposition}

However, below we require a somewhat dual property.
Namely, in Theorems~\ref{TheoremSupportsUnderVee}---\ref{TheoremVeeReflectsCoarserFiner} below $\mathcal C$
is a pre-rigid braided monoidal category  satisfying Properties~\ref{PropertySmallLimits},
 \ref{PropertySubObjectsSmallSet}--\ref{PropertyLimitsOfSubobjectsArePreserved},  \ref{PropertyEqualizers} of Section~\ref{SubsectionSupportCoactingConditions} and their duals, such that
 the functor $(-)^*$ maps extremal monomorphisms to extremal epimorphisms.  
 
 \begin{remark} The conditions above hold in $\mathcal C = \mathbf{Vect}_\mathbbm k$ for a field $\mathbbm k$.
	In Theorems~\ref{TheoremHModSatisfiesProperties} and~\ref{TheoremDgVectSatisfiesProperties}
	below we show that these conditions hold in $\mathcal C={}_H \mathsf{Mod}$, the category
	of left $H$-modules, for a quasitriangular Hopf algebra $H$ 
	and in $\mathcal C=\mathbf{dgVect}_\mathbbm{k}$, the category of differential graded vector spaces over a field $\mathbbm k$.
	In addition, the conditions above hold in $\mathcal C = \mathsf{Comod}^{\mathbbm kG}$, the category of $G$-graded vector spaces, for a field $\mathbbm k$ and a group $G$ (see Section~\ref{SubsectionComodCoquasitriangularHopfAlgebras}).
	In Remarks~\ref{left_YD_rem} (\ref{RemarkYDPolyMonoStarIsNotEpi}) and Remarks~\ref{Comod_coquasi_rmk} (\ref{RemarkPolyMonoStarIsNotEpiRComod}) below we give examples of categories where $f^*$ is not an epimorphism
	for some extremal monomorphism $f$.
\end{remark}

 These theorems will not be used until Section~\ref{SectionCosupportDualityInMonoidalClosedCategories}:

\begin{theorem}\label{TheoremSupportsUnderVee} 
	Let $\rho \colon A \to B \otimes Q$ be a morphism in $\mathcal C$ such that
	$\alpha_B \otimes \id_{P^*}$ and $\theta_{B^*,P}$ are monomorphisms for all objects $P$.
	Let $|\rho| \colon A \to B \otimes (\supp \rho)$ be the absolute value of $\rho$.
	Then $\cosupp(\rho^\vee) = (\supp \rho)^*$
	and $|\rho|^\vee \colon (\supp \rho)^* \otimes A \to B$
	is the absolute value of $\rho^\vee$.
\end{theorem}
\begin{proof}
	Recall that $|\rho|$ is a tensor epimorphism
	and
	$\rho = (\id_B \otimes \tau)|\rho|$
	for some extremal monomorphism $\tau \colon \supp \rho \rightarrowtail Q$ (see~\cite[Remark~4.15 and Lemma~4.17]{AGV3}).
	Hence by Proposition~\ref{PropositionVeeUnderComposition}
	$$\rho^\vee = \bigl((\id_B \otimes \tau)|\rho|\bigr)^\vee
	= |\rho|^\vee (\tau^* \otimes \id_A).$$
	By Proposition~\ref{PropositionDualToTensorEpiIsMono}
	the morphism $|\rho|^\vee$ is a tensor monomorphism
	and, by our assumptions, the morphism $\tau^*$ is an extremal epimorphism.
	Hence, by the proposition dual to~\cite[Proposition~4.19]{AGV3}, 
	 the morphism $|\rho|^\vee$
	indeed corresponds to the cosupport of $\rho^\vee$.
\end{proof}	

\begin{theorem}\label{TheoremVeePreservesCoarserFiner}
	Let $\rho_i \colon A \to B \otimes Q_i$, $i=1,2$, be some morphisms in $\mathcal C$
	such that
	$\alpha_B \otimes \id_{P^*}$ and $\theta_{B^*,P}$ are monomorphisms for all objects $P$. Then
	$\rho_1 \succcurlyeq \rho_2$ implies $\rho_1^\vee \succcurlyeq \rho_2^\vee$. 
\end{theorem}
\begin{proof} If $\rho_1 \succcurlyeq \rho_2$,
	then there exists a morphism $$\tau \colon \supp \rho_1 \to \supp \rho_2$$ such that $(\id_B \otimes \tau)|\rho_1| = |\rho_2|$.
	By Proposition~\ref{PropositionVeeUnderComposition} 
	we have 
	$$|\rho_1|^\vee(\tau^* \otimes \id_A) = |\rho_2|^\vee.$$
	Now Theorem~\ref{TheoremSupportsUnderVee} implies that $\rho_1^\vee \succcurlyeq \rho_2^\vee$.
\end{proof}

\begin{theorem}\label{TheoremVeeReflectsCoarserFiner}
	Let $\rho_i \colon A \to B \otimes Q_i$, $i=1,2$, be some morphisms in $\mathcal C$
	such that
	$\alpha_P$ is an extremal monomorphism and $\theta_{B^*,P}$ is a monomorphism for all objects $P$. Then
	$\rho_1^\vee \succcurlyeq \rho_2^\vee$ implies $\rho_1 \succcurlyeq \rho_2$. 
\end{theorem}
\begin{proof}  By Theorem~\ref{TheoremSupportsUnderVee},
	$|\rho_i^\vee|=|\rho_i|^\vee$. Suppose that $\rho_1^\vee \succcurlyeq \rho_2^\vee$.
		Then $$|\rho_1|^\vee(f \otimes \id_A) = |\rho_2|^\vee$$
	for some morphism $f \colon (\supp \rho_2)^* \to (\supp \rho_1)^*$.
		Proposition~\ref{PropositionDualityNablaVeeSharp} implies
	\begin{equation*}\begin{split}\bigl((\id_B \otimes f^\flat)|\rho_1| \bigr)^\nabla
			= |\rho_1|^\vee (f\otimes \id_A) = |\rho_2|^\vee 
			=  |\rho_2|^\vee(\id_{(\supp \rho_2)^*} \otimes \id_A) \\
			= \bigl((\id_B \otimes (\id_{\supp \rho_2})^\flat)|\rho_2| \bigr)^\nabla
			= \bigl((\id_B \otimes \alpha_{\supp \rho_2})|\rho_2| \bigr)^\nabla.
	\end{split}\end{equation*}
	
	By Corollary~\ref{CorollaryRhoABPStarEqualAfterNabla},
	$$(\id_B \otimes f^\flat)|\rho_1| = (\id_B \otimes \alpha_{\supp \rho_2})|\rho_2|.$$
	Hence $(\id_B \otimes \alpha_{\supp \rho_2})|\rho_2| \preccurlyeq \rho_1$.
	Since by our assumptions $\alpha_{\supp \rho_2}$ is an extremal monomorphism,
	\cite[Proposition~4.19]{AGV3} implies that
	$$|(\id_B \otimes \alpha_{\supp \rho_2})|\rho_2|\bigr| = 
	|\rho_2|,$$ and we get $\rho_2 \preccurlyeq \rho_1$.
\end{proof}

\subsection{Finite dual}
Let $\mathcal C$ be a pre-rigid braided monoidal category. Suppose that the functor $(-)^* \colon \mathsf{Comon}(\mathcal C) \to \mathsf{Mon}(\mathcal C)$ has an adjoint functor $(-)^\circ \colon \mathsf{Mon}(\mathcal C) \to \mathsf{Comon}(\mathcal C)$
such that there exists a natural bijection
$$\mathsf{Mon}(\mathcal C)(A, C^*) \cong  \mathsf{Comon}(\mathcal C)(C, A^\circ).$$

As usual, given an object $A$ in $\mathsf{Mon}(\mathcal C)$ the object $A^\circ$ is called the \textit{finite} or the \textit{Sweedler dual} of $A$.

\begin{remark}\label{RemarkCircSufficientConditionsForExistence}
	By \cite[Theorem 5.9]{AGV3} 
	 and \cite[Theorem 1.5]{ArdGoyMen1} the functor $(-)^\circ$ exists, for example, if
	$\mathcal C$ satisfies Properties~\ref{PropertySubObjectsSmallSet}, \ref{PropertyMonomorphism} and the properties dual to Properties~\ref{PropertySmallLimits}--\ref{PropertyEpiExtrMonoFactorizations} and \ref{PropertyFreeMonoid} of Section~\ref{SubsectionSupportCoactingConditions}.
\end{remark}

Let $A$ be a monoid in $\mathcal C$. Denote by $\varkappa_A \colon A^\circ \to A^*$ the morphism in $\mathcal C$
corresponding to $\id_{A^\circ}$ under the following composition of bijections and an embedding:
$$\xymatrix{\mathsf{Comon}(\mathcal C)(A^\circ,A^\circ) \cong  \mathsf{Mon}(\mathcal C)(A,A^{\circ *})
	\subseteq \mathcal C(A,A^{\circ *}) \ar[r]^(0.78){(-)^\sharp}& \mathcal C(A^{\circ},A^*)}.$$
Note that $\varkappa$ is natural in $A$ by the naturality of the maps above.

\begin{lemma}\label{LemmaQcircKappa} For every monoid $Q$ in $\mathcal C$
	the diagram below is commutative:
	$$ 
	\xymatrix{ Q^\circ\otimes Q \ar[d]_{\id_{Q^\circ} \otimes {\varkappa_Q^\flat}} \ar[rrr]^{{\varkappa_Q}\otimes \id_Q} & & & Q^*\otimes Q \ar[d]^{\mathrm{ev}_Q} \\
		{Q^\circ}\otimes {Q^{\circ*}} \ar[rr]^{c_{Q^\circ,Q^{\circ *}}} & & {Q^{\circ*}} \otimes {Q^\circ} \ar[r]^(0.6){\mathrm{ev}_{Q^\circ}} & \mathbbm{1}\\
	}
	$$
\end{lemma}
\begin{proof}
	Apply Lemma~\ref{LemmaFFlatEvBraiding} for $P=Q^\circ$, $U=Q$, $f=\varkappa_Q$.
\end{proof}	

\begin{proposition}\label{PropositionQcircMeasuring}
Let $\rho \colon A \to B \otimes Q$ be a comeasuring for some $\Omega$-magmas $A$ and $B$ and a monoid $Q$
in $\mathcal C$. Then $\rho^\vee(\varkappa_Q \otimes \id_A) \colon Q^\circ \otimes A \to B$
is a measuring.
\end{proposition}
\begin{proof} By Proposition~\ref{PropositionDualityNablaVeeSharp},
	$$\bigl((\id_B \otimes \varkappa_Q^\flat)\rho\bigr)^\nabla = \rho^\vee (\varkappa_Q \otimes \id_A).$$
	Now we use the fact that $ \varkappa_Q^\flat$ is a monoid homomorphism and apply Lemma~\ref{LemmaNablaMeasComeas}.
\end{proof}	

\subsection{Duality theorem for (co)measurings}

In Theorem~\ref{Theorem(Co)monUniv(Co)measDuality} below we not only show that 
$\mathcal{A}^\square(\rho_U)^\circ \cong {}_\square \mathcal{C}(\rho_U^\vee)$
but also provide an explicit isomorphism:

\begin{theorem}\label{Theorem(Co)monUniv(Co)measDuality}
Let~$\mathcal C$ be a pre-rigid braided monoidal category such that
\begin{itemize}
	\item there exists the functor $(-)^\circ$;
	\item $\alpha_M \otimes \id_N$ and $\theta_{M,N}$ are monomorphisms
	for all objects $M,N$.
\end{itemize}
	 Let $\rho_U \colon A \to B \otimes U$
	 be a tensor  epimorphism in~$\mathcal C$
	  for 
	 $\Omega$-magmas $A$ and $B$
	 and an object~$U$ such that there exist	 
	  $\mathcal{A}^\square(\rho_U)$ and ${}_\square \mathcal{C}(\rho_U^\vee)$.
	Then $$\left(\rho_U^\mathbf{Comeas}\right)^\vee(\varkappa_{\mathcal{A}^\square(\rho_U)}\otimes \id_A)
	\colon \mathcal{A}^\square(\rho_U)^\circ \otimes A \to B$$ is a measuring and
	 the unique comonoid homomorphism $\beta \colon \mathcal{A}^\square(\rho_U)^\circ \to {}_\square \mathcal{C}(\rho_U^\vee)$ making
	the diagram
\begin{equation}\label{Eq(Co)monUniv(Co)measDuality} \xymatrix{
\mathcal{A}^\square(\rho_U)^\circ \otimes A  \ar@{-->}[d]_{\beta \otimes \id_A}
\ar[rr]^{\varkappa_{\mathcal{A}^\square(\rho_U)}\otimes \id_A}
& & \mathcal{A}^\square(\rho_U)^* \otimes A \ar[d]^{\left(\rho_U^\mathbf{Comeas}\right)^\vee} \\
		 {}_\square \mathcal{C}(\rho_U^\vee) \otimes A \ar[rr]^(0.6){\left(\rho_U^\vee\right)^\mathbf{Meas}}& & B \\} 
	 \end{equation}  commutative is a comonoid isomorphism.
\end{theorem}
\begin{remark} Consider the following (not necessarily commutative) diagram where 
	$G_1$ and $G_1'$ are forgetful functors:
$$\xymatrix{ \mathbf{Comeas}(A,B) \ar[d]_{G_1} \ar[rr]^{(-)^\vee (\varkappa_{(\ldots)} \otimes \id_A)} & \quad & \mathbf{Meas}(A,B)^{\mathrm{op}} \ar[d]^{G'_1} \\
	\mathbf{MorTens}(A,B) \ar[rr]^{(-)^\vee} &  & \mathbf{TensMor}(A,B)^{\mathrm{op}}
}$$
Theorem~\ref{Theorem(Co)monUniv(Co)measDuality} asserts that 
if $\rho_U^\mathbf{Comeas}$ is the initial object in $\mathbf{Comeas}(A,B)_{G_1}(\rho_U)$,
then $\left(\rho_U^\mathbf{Comeas}\right)^\vee(\varkappa_{\mathcal{A}^\square(\rho_U)}\otimes \id_A)$
is an initial object in $\mathbf{Meas}(A,B)^{\mathrm{op}}_{G_1'}(\rho_U^\vee)$.
\end{remark}
\begin{proof}[Proof of Theorem~\ref{Theorem(Co)monUniv(Co)measDuality}] The morphism $\left(\rho_U^\mathbf{Comeas}\right)^\vee(\varkappa_{\mathcal{A}^\square(\rho_U)}\otimes \id_A)$
	is a measuring by Proposition~\ref{PropositionQcircMeasuring}.
	
	Inspired by~\cite[Remark~1.3]{Tambara},
	we can add the bijection from Lemma~\ref{Lemma(Co)monMeasComeasBijection} to the following bijections natural in the comonoid $P$:
	\begin{equation}\label{EqTambaraBijections}\begin{split}\mathsf{Comon}(\mathcal C)(P, {}_\square \mathcal{C}(\rho_U^\vee))\cong \mathrm{Meas}(P,\rho_U^\vee) \\ \cong \mathrm{Comeas}(P^*,\rho_U)\cong
	\mathsf{Mon}(\mathcal C)(\mathcal{A}^\square(\rho_U), P^*)
	\cong \mathsf{Comon}(\mathcal C)(P, \mathcal{A}^\square(\rho_U)^\circ).\end{split}
	\end{equation}
	
	Now if we substitute for $P$ the comonoid $\mathcal{A}^\square(\rho_U)^\circ$,
	the comonoid homomorphism $\mathcal{A}^\square(\rho_U)^\circ \to {}_\square \mathcal{C}(\rho_U^\vee)$,
	corresponding to $\id_{\mathcal{A}^\square(\rho_U)^\circ}$, will be equal to $\beta$ by Lemma~\ref{LemmaQcircKappa}.
	If we substitute for $P$ the comonoid~${}_\square \mathcal{C}(\rho_U^\vee)$,
	by the naturality,
	the homomorphism $\id_{{}_\square \mathcal{C}(\rho_U^\vee)}$ 
	will correspond to $\beta^{-1}$. Hence $\beta$ is a comonoid isomorphism.
\end{proof}

\begin{corollary}\label{Corollary(Co)monUniv(Co)measDuality}
Let~$\mathcal C$ be a pre-rigid braided monoidal category
satisfying Properties~\ref{PropertySmallLimits}, \ref{PropertySubObjectsSmallSet}, \ref{PropertyMonomorphism},    \ref{PropertySwitchProdTensorIsAMonomorphism}, \ref{PropertyFreeMonoid} of Section~\ref{SubsectionSupportCoactingConditions} and their duals such that	\begin{itemize}
	\item $\alpha_M \otimes \id_N$ and $\theta_{M,N}$ are monomorphisms
	for all objects $M,N$.
\end{itemize}
Let $\rho_U \colon A \to B \otimes U$
be a tensor  epimorphism in~$\mathcal C$
for 
$\Omega$-magmas $A$ and $B$
and an object~$U$ such that $\mathbf{Comeas}(A,B)_{G_1}(\rho_U)$ is not empty.
Then $$\left(\rho_U^\mathbf{Comeas}\right)^\vee(\varkappa_{\mathcal{A}^\square(\rho_U)}\otimes \id_A)
\colon \mathcal{A}^\square(\rho_U)^\circ \otimes A \to B$$ is a measuring and
the unique comonoid homomorphism $\beta \colon \mathcal{A}^\square(\rho_U)^\circ \to {}_\square \mathcal{C}(\rho_U^\vee)$ making
the diagram~\eqref{Eq(Co)monUniv(Co)measDuality}  commutative is a comonoid isomorphism.
\end{corollary}
\begin{proof}
	Apply~\cite[Proposition~4.2 (1), Theorems~4.24, 5.19]{AGV3}, 
	Remark~\ref{RemarkCircSufficientConditionsForExistence} and	Theorem \ref{Theorem(Co)monUniv(Co)measDuality}. 
\end{proof}		

\subsection{Duality theorem for (co)actions}

For given objects $A$ and $B$ in a pre-rigid braided monoidal category $\mathcal C$ denote by $\theta^\mathrm{inv}_{A,B} \colon A^* \otimes B^* \to (A\otimes B)^*$ the morphism corresponding under the bijection~\eqref{EqPrerigidBijection} to the composition $$\xymatrix{A^* \otimes B^* \otimes A\otimes B \ar[rr]^{{\id_A} \otimes {c_{B^*,A}} \otimes {\id_B}}
	& \qquad  & A^* \otimes A\otimes B^* \otimes B \ar[rr]^(0.6){{\mathrm{ev}_A} \otimes {\mathrm{ev}_B}} & &
	\mathbbm{1} \otimes \mathbbm{1} \ar[r]^{\sim}
	& \mathbbm{1}
}$$

(Please notice that in comparison with $\theta_{A,B}$ in  $\theta^\mathrm{inv}_{A,B}$ 
we use the original braiding, not the inverse one.) Again, $\theta^\mathrm{inv}_{A,B}$ is natural in $A$ and $B$.
\begin{remark}\label{RemarkThetaThetaInvAdjoint}
By~\cite[Proposition 4.4]{ArdGoyMen2022},  $(\theta^\mathrm{inv}_{A,B},\iota)$ is the monoidal structure
on the functor $(-)^* \colon \mathcal C^{\mathrm{op}} \to \mathcal C$ that corresponds to the op-monoidal structure $( \theta_{A,B}, \iota)$ on $(-)^* \colon \mathcal C \to \mathcal C^{\mathrm{op}}$ under the adjunction
 $
(-)^\sharp \colon \mathcal C^{\mathrm{op}}(A^*, B) \mathrel{\widetilde\to}
\mathcal C(A, B^*)$.
\end{remark}

\begin{lemma}\label{LemmaVeeCompositionThetaInv} Let $\rho_1 \colon B \to C \otimes Q_1$ and $\rho_2 \colon A \to B \otimes Q_2$ be morphisms in $\mathcal C$ for some objects $A,B,C, Q_1, Q_2$.
	Then the diagram below is commutative:
		$$\xymatrix{ Q_1^* \otimes Q_2^* \otimes A \ar[rrr]^{\id_{Q_1^*} \otimes \rho_2^\vee} 
			\ar[d]_{\id_{Q_1^* \otimes Q_2^*} \otimes \rho_2}
			& & &   Q_1^*\otimes B  \ar[dddd]^{\rho_1^\vee} \\
			Q_1^* \otimes Q_2^* \otimes B \otimes Q_2
			\ar[d]_{\id_{Q_1^* \otimes Q_2^*} \otimes \rho_1 \otimes \id_{Q_2}}
			   \\ 
			Q_1^* \otimes Q_2^* \otimes C \otimes Q_1 \otimes Q_2
						\ar[d]_{c_{Q_1^* \otimes Q_2^*, C} \otimes \id_{Q_1\otimes Q_2}}
					\\ 
			C \otimes Q_1^* \otimes Q_2^* \otimes Q_1 \otimes Q_2
			\ar[d]_{\id_C \otimes
				\theta^{\mathrm{inv}}_{Q_1, Q_2} \otimes \id_{Q_1 \otimes Q_2}}
			\\
			C \otimes (Q_1 \otimes Q_2)^* \otimes Q_1 \otimes Q_2 \ar[rr]_(0.65){\id_C \otimes
				\mathrm{ev}_{Q_1 \otimes Q_2}} &  \quad &  C \otimes \mathbbm{1} \ar[r]^\sim & C 
		}$$
\end{lemma}
\begin{proof}
	Consider the following diagram:
	$$\hspace{-0.7cm}
			\xymatrix{ Q_1^* \otimes Q_2^* \otimes A \ar[rrr]^{\id_{Q_1^*} \otimes \rho_2^\vee} 
				\ar[d]_{\id_{Q_1^* \otimes Q_2^*} \otimes \rho_2}
				& & &   Q_1^*\otimes B \ar@/^1.2pc/[lddd]^(0.65){\id_{Q_1^*}\otimes \rho_1} \ar[ddddd]^{\rho_1^\vee} \\
				 Q_1^* \otimes Q_2^* \otimes B \otimes Q_2
				\ar[d]_{\id_{Q_1^* \otimes Q_2^*} \otimes \rho_1 \otimes \id_{Q_2}}
				\ar[r]^{\id_{Q_1^*}
					\otimes c_{Q_2^*,B} \otimes \id_{Q_2} \phantom{\Bigl|}} & Q_1^* \otimes B \otimes Q_2^* \otimes Q_2
				\ar[d]_{\id_{Q_1^*}\otimes \rho_1 \otimes \id_{ Q_2^* \otimes Q_2}}
				\ar[r]^(0.55){\id_{Q_1^* \otimes B} \otimes \mathrm{ev}_{Q_2}} & Q_1^* \otimes B \otimes \mathbbm{1}\ar[d]^(0.45){\id_{Q_1^*}\otimes \rho_1 \otimes \id_\mathbbm{1}} \ar[ru]^\sim 
				&   \\ 
				 Q_1^* \otimes Q_2^* \otimes C \otimes Q_1 \otimes Q_2
				\ar[r]_{\id_{Q_1^*}
					\otimes c_{Q_2^*,C\otimes Q_1} \otimes \id_{Q_2} \phantom{\bigl|}}
				\ar[d]_{c_{Q_1^* \otimes Q_2^*, C} \otimes \id_{Q_1\otimes Q_2}}
				& Q_1^* \otimes C \otimes Q_1 \otimes Q_2^* \otimes Q_2 
				\ar[d]^(0.45){c_{Q_1^*, C} \otimes \id_{Q_1\otimes Q_2^* \otimes Q_2}}
				\ar[r]^{\qquad\id_{Q_1^* \otimes C \otimes Q_1} \otimes \mathrm{ev}_{Q_2} \phantom{\biggl|}} &
				 Q_1^* \otimes C \otimes Q_1 \otimes \mathbbm{1} \ar[d]^\sim
				\\ 
				 C \otimes Q_1^* \otimes Q_2^* \otimes Q_1 \otimes Q_2
				\ar[r]_{\id_{C \otimes Q_1^*}
					\otimes c_{Q_2^*, Q_1} \otimes \id_{Q_2} \phantom{\bigl|}}
\ar[dd]^{\id_C \otimes
	\theta^{\mathrm{inv}}_{Q_1, Q_2}\otimes \id_{Q_1 \otimes Q_2}}				
				& C \otimes Q_1^* \otimes Q_1 \otimes Q_2^* \otimes Q_2
				 \ar[d]^{\id_{C \otimes Q_1^* \otimes Q_1} \otimes \mathrm{ev}_{Q_2}}  &
				Q^*_1 \otimes C \otimes Q_1
				\ar[d]^{c_{Q_1^*, C} \otimes \id_{Q_1}} \\
				  &  C \otimes Q_1^* \otimes Q_1 \otimes \mathbbm{1} \ar[r]^\sim & C \otimes Q_1^* \otimes Q_1 \ar[d]^{\id_{C} \otimes \mathrm{ev}_{Q_1}}\\
				  C \otimes (Q_1 \otimes Q_2)^* \otimes Q_1 \otimes Q_2 \ar[rr]^{\id_C \otimes
				 \mathrm{ev}_{Q_1 \otimes Q_2}} &  &  C \otimes \mathbbm{1} \ar[r]^\sim & C 
		}
		$$
		The upper polygon is commutative by the definition of $\rho_2^\vee$. The lower right polygon is commutative by the definition of $\rho_1^\vee$. The upper right triangle is commutative by the properties of a monoidal category.
		The lower left polygon is commutative by the definition of $\theta^{\mathrm{inv}}$. The rectangles in the central part of the diagram are commutative by the functoriality of $\otimes$. The squares on the left are commutative by the properties of the braiding. Hence the outer square is commutative too, which implies the lemma.
\end{proof}	

\begin{lemma}\label{LemmaVeeCompositionSwitch} Let $\rho_1 \colon B \to C \otimes Q_1$ and $\rho_2 \colon A \to B \otimes Q_2$ be morphisms in $\mathcal C$ for some objects $A,B,C, Q_1, Q_2$.
	Then $$\left((\rho_1 \otimes \id_{Q_2})\rho_2\right)^\vee (\theta^\mathrm{inv}_{Q_1,Q_2} \otimes \id_A) = 
	\rho_1^\vee (\id_{Q_1^*} \otimes \rho_2^\vee ).$$
\end{lemma}
\begin{proof}
	Apply Lemma~\ref{LemmaVeeCompositionThetaInv} and the naturality of the braiding.
\end{proof}

Suppose now that~$\mathcal C$ is a pre-rigid \textit{symmetric} monoidal category, i.e. 
$\mathcal C$ is braided and $c_{B,A}=c_{A,B}^{-1}$ for every objects $A,B$. In particular, $\theta_{A,B}=\theta^\mathrm{inv}_{A,B}$.

\begin{lemma}\label{LemmaStarComoduleToAModule}
	Let a morphism $\rho \colon M \to M\otimes P$ define on an object $M$ a structure of a right $P$-comodule
	for a comonoid $P$. Then $\rho^\vee \colon P^* \otimes M \to M$ defines on $M$ a structure of a left $P^*$-module.
\end{lemma}
\begin{proof}
Consider the following diagram:
$$\xymatrix{ P^* \otimes M \ar[rrrrr]^{\id_{P^*} \otimes \rho} & & & & & P^*\otimes M \otimes P \ar[d]^{c_{P^*,M}\otimes \id_P} \\
	& \mathbbm{1}\otimes M \otimes P \ar[rrd]_(0.3)\sim \ar[rrrru]^{u \otimes \id_{M\otimes P}}
	\ar[rr]_{c_{\mathbbm{1},M}\otimes \id_P}
	 & & M \otimes {\mathbbm{1}}\otimes P \ar[rr]^{\id_M \otimes u \otimes \id_P}
	 \ar[d]^{\sim}
	  & & M \otimes P^*\otimes P \ar[d]^{\id_M \otimes \mathrm{ev}_P} \\
	& M \ar[rr]^\rho
	 & & M \otimes P \ar[rr]_{\id_M \otimes \varepsilon}& & M \otimes \mathbbm{1}
	 \ar[d]^\sim \\
\mathbbm{1} \otimes M \ar[uuu]^{u\otimes \id_M}\ar[ruu]^{\id_\mathbbm{1} \otimes \rho}
\ar[ru]_\sim
 \ar[rrrrr]^\sim & & & & & M\ar@{=}[llllu]	
}$$
The inner polygons are commutative by the naturality of the transformation $\mathbbm{1}\otimes C \mathrel{\widetilde\to} C$, the counit property of $\rho$, the functoriality of $\otimes$, Lemma~\ref{LemmaPStarPStarPDifferentWays} and the properties of the braiding.
Therefore the outer polygon is commutative too and $\rho^\vee$ satisfies the unit axiom of a module structure.

Now consider the diagram below:
$$\hspace{-7mm}
\xymatrix{ P^* \otimes P^* \otimes M 
	 \ar[ddd]_{\mu \otimes \id_M}
	\ar[r]^(0.45){\id_{P^* \otimes P^*} \otimes \rho\phantom{\Bigr|}}
	\ar[rd]_(0.45){\id_{P^* \otimes P^*} \otimes \rho\quad }
	& P^* \otimes P^* \otimes M \otimes P
	\ar[r]^{\id_{P^* \otimes P^*} \otimes \rho \otimes \id_P \phantom{\Bigr|}}
	& P^* \otimes P^* \otimes M \otimes P \otimes P
	\ar[r]^{c_{P^* \otimes P^*, M} \otimes \id_{P\otimes P}\phantom{\Bigr|}} &
	M \otimes P^* \otimes P^* \otimes P \otimes P
	\ar[d]^{\id_{M}
		\otimes \theta_{P, P} \otimes \id_{P\otimes P}}
	 \\
	& P^* \otimes P^* \otimes M \otimes P
	\ar[ru]^(0.4){\id_{P^* \otimes P^* \otimes M} \otimes \Delta\quad}
	\ar[r]_{c_{P^* \otimes P^*, M} \otimes \id_P\phantom{\bigl|}}
	\ar[d]_{\mu\otimes \id_{M \otimes P}}
	 	&  M \otimes P^* \otimes P^* \otimes P\ar[d]^{\id_M \otimes \mu \otimes \id_P}
	\ar[ru]^(0.4){\id_{M \otimes P^* \otimes P^*} \otimes \Delta\quad}
	& M \otimes (P \otimes P)^* \otimes P \otimes P \ar[d]^{\id_{M} \otimes \mathrm{ev}_{P\otimes P}}
	  \\
	& P^* \otimes M \otimes P \ar[r]_{c_{P^*, M} \otimes \id_P} & M \otimes P^* \otimes P \ar[r]_{\id_M\otimes \mathrm{ev}_P }  & M \otimes \mathbbm{1} \ar[d]^\sim  \\
	P^* \otimes M \ar[ru]^{\id_{P^*}\otimes \rho} \ar[rrr]^{\rho^\vee} & &   & M
}
$$
The inner polygons are commutative by the naturality of the braiding, Lemma~\ref{LemmaPStarPStarPDifferentWays},
the definition of $\rho^\vee$, the coassociativity of $\rho$ and the functoriality of $\otimes$.
Therefore the outer square is commutative too.
By Lemma~\ref{LemmaVeeCompositionThetaInv},
the composition along the upper and the right sides of the outer square equals~$\rho^\vee(\id_{P^*} \otimes \rho^\vee)$. Therefore, $\rho^\vee$ satisfies the associativity axiom of a module structure.
\end{proof}

When $\mathcal C$ is symmetric, by Theorem~\ref{TheoremStarBraidedLaxMonoidal} and Remark~\ref{RemarkThetaThetaInvAdjoint},  $(-)^*$ is a \textit{symmetric} monoidal functor $\mathcal C^\mathrm{op} \to \mathcal C$
(i.e. a braided functor in the case of symmetric categories),
which is right adjoint to itself considered as an op-monoidal functor $(-)^* \colon \mathcal C \to \mathcal C^\mathrm{op}$.
In particular, $(-)^*$ induces a monoidal functor $\mathsf{Mon}(\mathcal C^\mathrm{op}) \to \mathsf{Mon}(\mathcal C)$
and an op-monoidal functor $\mathsf{Comon}(\mathcal C) \to \mathsf{Comon}(\mathcal C^\mathrm{op})$.
We denote both of them again by $(-)^*$. Suppose now that $(-)^* \colon \mathsf{Mon}(\mathcal C^\mathrm{op}) \to \mathsf{Mon}(\mathcal C)$ admits a left adjoint functor $(-)^\circ \colon \mathsf{Mon}(\mathcal C) \to \mathsf{Mon}(\mathcal C^\mathrm{op})$. Then $(-)^\circ$ is right adjoint to the functor $(-)^* \colon \mathsf{Comon}(\mathcal C) \to \mathsf{Comon}(\mathcal C^\mathrm{op})$ if we view $(-)^\circ$ as a functor $\mathsf{Comon}(\mathcal C^\mathrm{op}) \to \mathsf{Comon}(\mathcal C)$. Moreover, $(-)^\circ$ being a left adjoint to a monoidal functor is an op-monoidal functor
$\mathsf{Mon}(\mathcal C) \to \mathsf{Mon}(\mathcal C^\mathrm{op})$. Analogously, 
$(-)^\circ$ is a monoidal functor $\mathsf{Comon}(\mathcal C^\mathrm{op}) \to \mathsf{Comon}(\mathcal C)$.
Now recall that $\mathsf{Comon}(\mathsf{Mon}(\mathcal C))=\mathsf{Mon}(\mathsf{Comon}(\mathcal C))=\mathsf{Bimon}(\mathcal C)$.
Hence $(-)^\circ$ induces functors $\mathsf{Bimon}(\mathcal C) \to \mathsf{Bimon}(\mathcal C^\mathrm{op})$
and $\mathsf{Bimon}(\mathcal C^\mathrm{op}) \to \mathsf{Bimon}(\mathcal C)$,
which we again denote by $(-)^\circ$. By~\cite[Theorem 2.7]{GoyVer}, $(-)^\circ \colon \mathsf{Bimon}(\mathcal C) \to \mathsf{Bimon}(\mathcal C^\mathrm{op})$ is left adjoint to $(-)^\circ \colon \mathsf{Bimon}(\mathcal C^\mathrm{op}) \to \mathsf{Bimon}(\mathcal C)$.

\begin{lemma}\label{LemmaKappaForBimonoids} For every bimonoid $B$ in $\mathcal C$
	the morphism
	$\varkappa_B \colon B^\circ \to B^*$ is a monoid homomorphism.
\end{lemma}
\begin{proof} Applying the proof of~\cite[Theorem 2.7]{GoyVer} to the case of the functors $(-)^*$
	and $(-)^\circ$, we see
	that, given bimonoids $B_1$ and $B_2$,
	the natural bijection 
	$$\mathsf{Bimon}(\mathcal C^\mathrm{op})(B_1^\circ, B_2) \mathrel{\widetilde{\to}}
	\mathsf{Bimon}(\mathcal C)(B_1, B_2^\circ)
	$$ is the composition of (co)restrictions
	of natural bijections
	$$\mathsf{Mon}(\mathcal C^\mathrm{op})(B_1^\circ, B_2) \mathrel{\widetilde{\to}}
	\mathsf{Mon}(\mathcal C)(B_1, B_2^*),
	$$
	$$
(-)^\sharp \colon {\mathcal C}(B_1, B_2^*) \mathrel{\widetilde{\to}} {\mathcal C}^\mathrm{op}(B_1^*, B_2),
$$
$$\mathsf{Comon}(\mathcal C^\mathrm{op})(B_1^*, B_2) \mathrel{\widetilde{\to}} \mathsf{Comon}(\mathcal C)(B_1, B_2^\circ).$$	
In particular, the morphism from ${\mathcal C}^\mathrm{op}(B_1^*, B_2)$ obtained in the second line in fact belongs to 
$\mathsf{Comon}(\mathcal C^\mathrm{op})(B_1^*, B_2)$.
If we let $B_1 = B$ and $B_2 = B^\circ$
and consider $\id_B \in \mathsf{Bimon}(\mathcal C^\mathrm{op})(B^\circ, B^\circ)$,
we get $\varkappa_B \in \mathsf{Comon}(\mathcal C^\mathrm{op})(B^*, B^\circ)$, which proves the lemma.
\end{proof}	

Again, it turns out that under mild conditions
$\mathcal{B}^\square(\rho_U)^\circ \cong {}_\square \mathcal{B}(\rho_U^\vee)$:

\begin{theorem}\label{TheoremBimonUniv(Co)actDuality}
Let~$\mathcal C$ be a pre-rigid symmetric monoidal category such that
\begin{itemize}
\item there exists the functor $(-)^\circ$;
\item $\alpha_M \otimes \id_N$ and $\theta_{M,N}$ are monomorphisms
for all objects $M,N$.
\end{itemize}
Let $\rho_U \colon A \to A \otimes U$
be a tensor epimorphism	defining on an $\Omega$-magma~$A$ a structure of a $U$-comodule for a comonoid $U$ such that there exist	 
$\mathcal{A}^\square(\rho_U)$ and ${}_\square \mathcal{C}(\rho_U^\vee)$.
Then $$\left(\rho_U^\mathbf{Coact}\right)^\vee(\varkappa_{\mathcal{B}^\square(\rho_U)}\otimes \id_A)
\colon \mathcal{B}^\square(\rho_U)^\circ \otimes A \to A$$ is an action and
the unique bimonoid homomorphism $\beta \colon \mathcal{B}^\square(\rho_U)^\circ \to {}_\square \mathcal{B}(\rho_U^\vee)$ making
the diagram
\begin{equation}\label{EqBimonUniv(Co)actDuality}
\xymatrix{
	\mathcal{B}^\square(\rho_U)^\circ \otimes A  \ar@{-->}[d]_{\beta \otimes \id_A}
	\ar[rr]^{\varkappa_{\mathcal{B}^\square(\rho_U)}\otimes \id_A}
	& & \mathcal{B}^\square(\rho_U)^* \otimes A \ar[d]^{\left(\rho_U^\mathbf{Coact}\right)^\vee} \\
	{}_\square \mathcal{B}(\rho_U^\vee) \otimes A \ar[rr]^(0.6){\left(\rho_U^\vee\right)^\mathbf{Act}}& & A \\}
\end{equation}
  commutative is a bimonoid isomorphism.	
	\end{theorem}
\begin{remark} Consider the following (not necessarily commutative) diagram where 
	$G_2$ and $G_2'$ are forgetful functors and $\mathbf{ModStr}(A)$ is the category of module structures on $A$:
	$$\xymatrix{ \mathbf{Coact}(A) \ar[d]_{G_2} \ar[rr]^{(-)^\vee (\varkappa_{(\ldots)} \otimes \id_A)} & \quad & \mathbf{Act}(A)^{\mathrm{op}} \ar[d]^{G'_2} \\
\mathbf{ComodStr}(A) \ar[rr]^{(-)^\vee} & \quad & \mathbf{ModStr}(A)^{\mathrm{op}} 		
	}$$
	Theorem~\ref{TheoremBimonUniv(Co)actDuality} asserts that 
if $\rho_U^\mathbf{Coact}$ is the initial object in $\mathbf{Coact}(A)_{G_2}(\rho_U)$,
then $\left(\rho_U^\mathbf{Coact}\right)^\vee(\varkappa_{\mathcal{B}^\square(\rho_U)}\otimes \id_A)$
is an initial object in $\mathbf{Act}(A)^{\mathrm{op}}_{G_2'}(\rho_U^\vee)$.
\end{remark}
\begin{proof}[Proof of Theorem~\ref{TheoremBimonUniv(Co)actDuality}] By~\cite[Theorems~4.35 and~5.23]{AGV3}
	there exist $\mathcal{B}^\square(\rho_U) = \mathcal{A}^\square(\rho_U)$ and ${}_\square \mathcal{B}(\rho_U^\vee)={}_\square \mathcal{C}(\rho_U^\vee)$.
	The morphism $\left(\rho_U^\mathbf{Coact}\right)^\vee(\varkappa_{\mathcal{B}^\square(\rho_U)}\otimes \id_A)$ is an action by
	Theorem~\ref{Theorem(Co)monUniv(Co)measDuality} and Lemmas~\ref{LemmaStarComoduleToAModule} and~\ref{LemmaKappaForBimonoids}.
	Hence the existence of $\beta$ follows from the definition of ${}_\square \mathcal{B}(\rho_U^\vee)$. By
	Theorem~\ref{Theorem(Co)monUniv(Co)measDuality} the morphism $\beta$ is an isomorphism in $\mathcal C$.
	Combined with the fact that $\beta$ is a bimonoid homomorphism, this implies that $\beta$ is a bimonoid isomorphism.
\end{proof}
\begin{corollary}
	Let~$\mathcal C$ be a pre-rigid symmetric monoidal category satisfying
	Properties~\ref{PropertySmallLimits}, \ref{PropertySubObjectsSmallSet}, \ref{PropertyMonomorphism}, \ref{PropertySwitchProdTensorIsAMonomorphism}, \ref{PropertyFreeMonoid} of Section~\ref{SubsectionSupportCoactingConditions} and their duals 	such that
	\begin{itemize}
		\item $\alpha_M \otimes \id_N$ and $\theta_{M,N}$ are monomorphisms
		for all objects $M,N$.
	\end{itemize}
	Let $\rho_U \colon A \to A \otimes U$
	be a tensor epimorphism	defining on an $\Omega$-magma~$A$ a structure of a $U$-comodule for a comonoid $U$.
	Then $$\left(\rho_U^\mathbf{Coact}\right)^\vee(\varkappa_{\mathcal{B}^\square(\rho_U)}\otimes \id_A)
	\colon \mathcal{B}^\square(\rho_U)^\circ \otimes A \to A$$ is an action and
	the unique bimonoid homomorphism $\beta \colon \mathcal{B}^\square(\rho_U)^\circ \to {}_\square \mathcal{B}(\rho_U^\vee)$ making
	the diagram~\eqref{EqBimonUniv(Co)actDuality}  commutative is a bimonoid isomorphism.	
\end{corollary}
\begin{proof}
	Apply~\cite[Proposition~4.2 (1), Theorems~4.24, 5.19]{AGV3}, 
Remark~\ref{RemarkCircSufficientConditionsForExistence} and	Theorem \ref{TheoremBimonUniv(Co)actDuality}. 
\end{proof}	

Note that since $(-)^* \colon \mathcal C^\mathrm{op} \to \mathcal C$ is a  a symmetric functor, the induced functor
$(-)^* \colon \mathsf{Mon}(\mathcal C^\mathrm{op}) \to \mathsf{Mon}(\mathcal C)$
 commutes with the functor $(-)^{\mathrm{op}}$, i.e. $(C^\mathrm{cop})^*=(C^*)^\mathrm{op}$ for every comonoid $C$ in $\mathcal C$. This implies that for the left adjoint functor $(-)^\circ \colon
 \mathsf{Mon}(\mathcal C) \to \mathsf{Mon}(\mathcal C^\mathrm{op})$
 there exists an isomorphism $(A^\circ)^\mathrm{cop} \cong (A^\mathrm{op})^\circ$ natural in the monoid $A$ in $\mathcal C$. 
  If $H$ is a Hopf monoid in $\mathcal C$, then its antipode $S$
 is a monoid homomorphism $H \to H^\mathrm{op}$. Now~\cite[Proposition 31]{PorstStreet} implies
  that if $\varkappa_H$ is a monomorphism in $\mathcal C$,
  then $H^\circ$ is a Hopf monoid too where the antipode is
   the comonoid homomorphism $(H^\circ)^\mathrm{cop} \to H^\circ$
  corresponding to $S^\circ \colon (H^\mathrm{op})^\circ \to H^\circ$.
  As it was noticed in~\cite{GoyVer}, the fact that any bimonoid homomorphism between Hopf monoids is
  in fact a Hopf monoid homomorphism implies that the adjunction
  $$\mathsf{Bimon}(\mathcal C)(H_2, H_1^\circ) \cong \mathsf{Bimon}(\mathcal C)(H_1, H_2^\circ)$$
  can be restricted to an adjunction
  $$\mathsf{Hopf}(\mathcal C)(H_2, H_1^\circ) \cong \mathsf{Hopf}(\mathcal C)(H_1, H_2^\circ)$$
  provided that $\varkappa_H$ is a monomorphism in $\mathcal C$ for every Hopf monoid $H$.
  
  Denote by $H_l$ and $H_r$, respectively, the left and the right adjoint functors (if they exist) to the embedding functor $\mathsf{Hopf}(\mathcal C) \subseteq \mathsf{Bimon}(\mathcal C)$.

\begin{theorem}\label{TheoremHopfMonUniv(Co)actDuality}
Let~$\mathcal C$ be a pre-rigid symmetric monoidal category such that
\begin{itemize}
\item there exist functors $(-)^\circ$, $H_l$ and $H_r$;
\item the morphisms
 $\alpha_M \otimes \id_N$ and $\theta_{M,N}$ are monomorphisms
for all objects $M,N$;
\item $\varkappa_H$ is a monomorphism in $\mathcal C$ for every Hopf monoid $H$.
\end{itemize}
Let $\rho_U \colon A \to A \otimes U$
be a tensor epimorphism	defining on an $\Omega$-magma~$A$ a structure of a $U$-comodule for a comonoid $U$ such that there exist	 
$\mathcal{A}^\square(\rho_U)$ and ${}_\square \mathcal{C}(\rho_U^\vee)$.
Then $$\left(\rho_U^\mathbf{HCoact}\right)^\vee(\varkappa_{\mathcal{H}^\square(\rho_U)}\otimes \id_A)
\colon \mathcal{H}^\square(\rho_U)^\circ \otimes A \to A$$ is an action and
the unique Hopf monoid homomorphism $\beta^{\mathbf{Hopf}} \colon \mathcal{H}^\square(\rho_U)^\circ \to {}_\square \mathcal{H}(\rho_U^\vee)$ making
the diagram
\begin{equation}\label{EqHopfMonUniv(Co)actDuality}
\xymatrix{
	\mathcal{H}^\square(\rho_U)^\circ \otimes A  \ar@{-->}[d]_{\beta^{\mathbf{Hopf}} \otimes \id_A}
	\ar[rr]^{\varkappa_{\mathcal{H}^\square(\rho_U)}\otimes \id_A}
	& & \mathcal{H}^\square(\rho_U)^* \otimes A \ar[d]^{\left(\rho_U^\mathbf{HCoact}\right)^\vee} \\
	{}_\square \mathcal{H}(\rho_U^\vee) \otimes A \ar[rr]^(0.6){\left(\rho_U^\vee\right)^\mathbf{HAct}}& & A \\}
\end{equation}
commutative is a Hopf monoid isomorphism.		
\end{theorem}
\begin{remark} Consider the following (not necessarily commutative) diagram where 
	$G_2G_4$ and $G_2'G_4'$ are forgetful functors:
	$$\xymatrix{ \mathbf{HCoact}(A) \ar[d]_{G_2G_4} \ar[rr]^{(-)^\vee (\varkappa_{(\ldots)} \otimes \id_A)} & \quad & \mathbf{HAct}(A)^{\mathrm{op}} \ar[d]^{G_2'G'_4} \\
		\mathbf{ComodStr}(A) \ar[rr]^{(-)^\vee} & \quad & \mathbf{ModStr}(A)^{\mathrm{op}} 		
	}$$
	Theorem~\ref{TheoremHopfMonUniv(Co)actDuality} asserts that 
if $\rho_U^\mathbf{HCoact}$ is the initial object in $\mathbf{HCoact}(A)_{G_2 G_4}(\rho_U)$,
then $\left(\rho_U^\mathbf{HCoact}\right)^\vee(\varkappa_{\mathcal{H}^\square(\rho_U)}\otimes \id_A)$
is an initial object in $\mathbf{HAct}(A)^{\mathrm{op}}_{G_2' G_4'}(\rho_U^\vee)$.
\end{remark}
\begin{proof}[Proof of Theorem~\ref{TheoremHopfMonUniv(Co)actDuality}]Let $H$ be an arbitrary Hopf monoid in $\mathcal C$.
	By~\cite[Theorems~4.42 and~5.27]{AGV3} 
	 and Theorem~\ref{TheoremBimonUniv(Co)actDuality} we have natural bijections
	\begin{equation*}
		\begin{split}
			\mathsf{Hopf}(\mathcal C)(H, {}_\square \mathcal{H}(\rho_U^\vee)) = 
			\mathsf{Hopf}(\mathcal C)(H, H_r({}_\square \mathcal{B}(\rho_U^\vee)))\cong
			\mathsf{Bimon}(\mathcal C)(H, {}_\square \mathcal{B}(\rho_U^\vee)) \\
			\cong \mathsf{Bimon}(\mathcal C)(H, \mathcal{B}^\square(\rho_U)^\circ)
			\cong \mathsf{Bimon}(\mathcal C)(\mathcal{B}^\square(\rho_U), H^\circ)
			\cong \mathsf{Hopf}(\mathcal C)(H_l(\mathcal{B}^\square(\rho_U)), H^\circ) \\
			= \mathsf{Hopf}(\mathcal C)(\mathcal{H}^\square(\rho_U), H^\circ)
			\cong \mathsf{Hopf}(\mathcal C)(H, \mathcal{H}^\square(\rho_U)^\circ).\end{split}\end{equation*}
	Hence $\mathcal{H}^\square(\rho_U)^\circ \cong {}_\square \mathcal{H}(\rho_U^\vee)$
	under the isomorphism that corresponds to $$\id_{\mathcal{H}^\square(\rho_U)^\circ}
	\in \mathsf{Hopf}(\mathcal C)(\mathcal{H}^\square(\rho_U)^\circ, \mathcal{H}^\square(\rho_U)^\circ)$$
	if we take $H=\mathcal{H}^\square(\rho_U)^\circ$.
	But the corresponding element of $\mathsf{Hopf}(\mathcal C)(\mathcal{H}^\square(\rho_U)^\circ, {}_\square \mathcal{H}(\rho_U^\vee))$ is precisely $\beta^\mathbf{Hopf}$ resulting from the universal properties 
	of ${}_\square \mathcal{B}(\rho_U^\vee)$ and ${}_\square \mathcal{H}(\rho_U^\vee)$:
	$$
	\xymatrix{ & & & \mathcal{H}^\square(\rho_U)^* \otimes A \ar[d] \ar@/^4pc/[rddd]^{\left(\rho_U^\mathbf{HCoact}\right)^\vee}   \\
	\mathcal{H}^\square(\rho_U)^\circ \otimes A \ar[r]
	\ar[rrru]^{\varkappa_{\mathcal{H}^\square(\rho_U)}\otimes \id_A}
	 \ar@{-->}[d]_{\beta^{\mathbf{Hopf}} \otimes \id_A} &	\mathcal{B}^\square(\rho_U)^\circ \otimes A  \ar[d]^{\beta \otimes \id_A}
		\ar[rr]^{\varkappa_{\mathcal{B}^\square(\rho_U)}\otimes \id_A}
		& & \mathcal{B}^\square(\rho_U)^* \otimes A \ar[rdd]^{\left(\rho_U^\mathbf{Coact}\right)^\vee} \\
		{}_\square \mathcal{H}(\rho_U^\vee) \otimes A \ar[r]
		\ar@/_2pc/[rrrrd]_{\left(\rho_U^\vee\right)^\mathbf{HAct}} & {}_\square \mathcal{B}(\rho_U^\vee) \otimes A \ar[rrrd]^(0.6){\left(\rho_U^\vee\right)^\mathbf{Act}}& & \\
		& & & & A \\}
	$$
	 (The right triangle in the diagram above is commutative by Proposition~\ref{PropositionVeeUnderComposition}.)
	
	The uniqueness of $\beta^\mathbf{Hopf}$ follows from the universal property of ${}_\square \mathcal{H}(\rho_U^\vee)$ too.
\end{proof}
\begin{corollary}\label{CorollaryHopfMonUniv(Co)actDuality}
	Let~$\mathcal C$ be a pre-rigid symmetric monoidal category satisfying
	Properties~\ref{PropertySmallLimits}, \ref{PropertySubObjectsSmallSet}, \ref{PropertyMonomorphism},    \ref{PropertySwitchProdTensorIsAMonomorphism}, \ref{PropertyFreeMonoid} of Section~\ref{SubsectionSupportCoactingConditions} and their duals such that
	\begin{itemize}
		\item the morphisms
		$\alpha_M \otimes \id_N$ and $\theta_{M,N}$ are monomorphisms
		for all objects $M,N$;
		\item $\varkappa_H$ is a monomorphism in $\mathcal C$ for every Hopf monoid $H$.
	\end{itemize}
	Let $\rho_U \colon A \to A \otimes U$
	be a tensor epimorphism	defining on an $\Omega$-magma~$A$ a structure of a $U$-comodule for a comonoid $U$.
	Then $$\left(\rho_U^\mathbf{HCoact}\right)^\vee(\varkappa_{\mathcal{H}^\square(\rho_U)}\otimes \id_A)
	\colon \mathcal{H}^\square(\rho_U)^\circ \otimes A \to A$$ is an action and
	the unique Hopf monoid homomorphism $\beta^{\mathbf{Hopf}} \colon \mathcal{H}^\square(\rho_U)^\circ \to {}_\square \mathcal{H}(\rho_U^\vee)$ making
	the diagram~\eqref{EqHopfMonUniv(Co)actDuality}
	commutative is a Hopf monoid isomorphism.		
\end{corollary}
\begin{proof}
	Apply~\cite[Proposition~4.2 (1), Theorems~4.6, 4.24, 5.10, 5.19]{AGV3}, 
Remark~\ref{RemarkCircSufficientConditionsForExistence} and	Theorem~\ref{TheoremHopfMonUniv(Co)actDuality}. 
\end{proof}	

\section{Cosupport and duality in closed monoidal categories }\label{SectionCosupportDualityInMonoidalClosedCategories}

\subsection{Cosupports in closed monoidal categories} When the base category $\mathcal C$ is closed monoidal,
we can reduce the Lifting Problem to the Lifting Problem
for a forgetful functor between such comma categories that the notion of the absolute value coincides with the notion of the cosupport. This allows us to work with subobjects instead of tensor monomorphisms, which in some situations (say, for $\mathcal C = \mathbf{Vect}_\mathbbm{k}$) is easier.

Let $\mathcal C$ be a closed monoidal category. (See examples in Section~\ref{SectionApplications} below.) Recall that by $[A,-]$ we denote the right adjoint functor to $(-)\otimes A$
  and by $\mathrm{ev}_{A,B} \colon [A,B]\otimes A \to B$ the counit of this adjunction. Let $(-)^* := [-,\mathbbm{1}]$. Then $\mathrm{ev}_{A} := \mathrm{ev}_{A,\mathbbm{1}}$. 

Given objects $A,B$ in $\mathcal C$, consider the comma category $(\mathcal C \downarrow [A,B])$,
i.e. the category where
\begin{itemize}
	\item objects are morphisms $\zeta \colon P \to [A,B]$ for objects $P$ in $\mathcal C$;
	\item morphisms between $\zeta_1 \colon P_1 \to [A,B]$ and $\zeta_2 \colon P_2 \to [A,B]$
	are morphisms $\tau \colon P_1 \to P_2$ such that $\zeta_2\tau = \zeta_1$:
	$$\xymatrix{ P_1 \ar[d]_\tau \ar[r]^(0.4){\zeta_1} & [A,B] \\
		P_2 \ar[ru]_{\zeta_2}
	}$$
\end{itemize} 

\begin{remark}
Objects in $\LIO\left((\mathcal C \downarrow [A,B])^{\mathrm{op}} \right)$
are exactly monomorphisms $P \rightarrowtail [A,B]$. Moreover, if $\mathcal C$
satisfies the property dual to Property~\ref{PropertyEpiExtrMonoFactorizations} of Section~\ref{SubsectionSupportCoactingConditions},
then in $(\mathcal C \downarrow [A,B])^{\mathrm{op}}$ there exist all absolute values, since if
 $\zeta = i \tau$
is an (ExtrEpi, Mono)-factorization of $\zeta \colon P \to [A,B]$, then $|\zeta|=i$.
\end{remark}

Denote by $K \colon  \mathbf{TensMor}(A,B) \mathrel{\widetilde\to} (\mathcal C \downarrow [A,B])$
the isomorphism of categories, corresponding to the natural bijection
\begin{equation*}\mathcal C(P\otimes A, B) \cong \mathcal C(P, [A,B]).\end{equation*}

Being an isomorphism, the functor $K$ maps $\LIO(\mathbf{TensMor}(A,B)^\mathrm{op})$
onto $\LIO((\mathcal C \downarrow [A,B])^\mathrm{op})$, commutes with taking absolute values and preserves the preorder.
In other words, we get the following proposition and remarks:

\begin{proposition}\label{PropositionMonoTensorMonoClosed} The morphism $\psi \colon P \otimes A \to B$ is a tensor monomorphism in $\mathcal C$  for some objects $P,A,B$
	if and only if the corresponding morphism $K\psi \colon P \to [A,B]$ is a monomorphism.
\end{proposition}	

\begin{remarks}\label{Cosupp_duality_rmk}
 \hspace{0.1cm}
\begin{enumerate}
\item\label{RemarkCosupportIsTheImage}
	Suppose that $\mathcal C$ satisfies the property dual to Property~\ref{PropertyEpiExtrMonoFactorizations}.
	Let $\psi \colon P \otimes A \to B$ be a morphism
	and let $K\psi = i \pi$ be the (ExtrEpi, Mono)-factorization.
	Then there exists $|\psi| = K^{-1} i$, i.e. $i$ is a morphism $\cosupp \psi \rightarrowtail [A,B]$
	and $\psi = |\psi|(\pi \otimes \id_A)$:
	\begin{equation*}\xymatrix{
			P \otimes A \ar[r]^\psi  \ar[d]_(0.45){\pi\otimes \id_A} & B  \\
			(\cosupp \psi) \otimes A \ar[ru]_(0.6){|\psi|} &
	}\end{equation*}
	In particular, $|K\psi|=K|\psi|$ and the cosupport of $\psi$ is just the image of $K\psi$ in $[A,B]$;
\item\label{RemarkCosupportSubobjectPartialOrder} Under the same assumptions as in (\ref{RemarkCosupportIsTheImage}) above,
	we have $\psi_1 \succcurlyeq \psi_2$ (with respect to $X=(\mathcal C \downarrow [A,B])^\mathrm{op}$) if and only if $\cosupp \psi_1$ is a subobject of $\cosupp \psi_2$.
	\end{enumerate}
\end{remarks}

Proposition~\ref{PropositionMonoidActionMonoidHomomorphism} below is 
verified directly using the universal property of $\mathrm{ev}$ dually to the proof of
\cite[Theorem~4.31]{AGV3}:
\begin{proposition}\label{PropositionMonoidActionMonoidHomomorphism} 
	For every object $A$ in $\mathcal C$ the object $[A,A]$ admits a unique structure of a monoid
	turning $A$ into a left $[A,A]$-module via $\mathrm{ev}_{A,A}$.
	Namely, the multiplication $\mu_{[A,A]} \colon [A,A]\otimes [A,A] \to [A,A]$
	and the unit $u_{[A,A]} \colon \mathbbm{1}\to [A,A]$
	are the unique morphisms making the diagrams below commutative:
	$$\xymatrix{
		[A,A] \otimes [A,A] \otimes A \ar[rr]^(0.6){\id_{[A,A]} \otimes  \mathrm{ev}_{A,A}} \ar@{-->}[d]_{\mu_{[A,A]}\otimes \id_A} &\quad & [A,A] \otimes A \ar[d]^{\mathrm{ev}_{A,A}}\\
		[A,A] \otimes A \ar[rr]^{\mathrm{ev}_{A,A}} & & A
	}
	\qquad
	\xymatrix{ \mathbbm{1} \otimes A \ar@{-->}[dd]_{u_{[A,A]}\otimes \id_A} \ar[rd]^\sim &     \\
		& A \\
		[A,A] \otimes A \ar[ru]^{\mathrm{ev}_{A,A}} &
	} 
	$$
	Moreover, if $\psi \colon P \otimes A \to A$ is a morphism 
	for some object $A$ and a monoid $(P,\mu,u)$ in $\mathcal C$,
	then $\psi$
	defines on $A$ a structure of a $P$-module if and only if
	the corresponding morphism $K\psi \colon P \to [A,A]$ is a monoid homomorphism.
	In particular, the isomorphism of categories $$K \colon  \mathbf{TensMor}(A,A) \mathrel{\widetilde\to} (\mathcal C \downarrow [A,A])$$ restricts to
	an isomorphism of categories
	$$\mathbf{ModStr}(A) \mathrel{\widetilde\to} (\mathsf{Mon}(\mathcal C) \downarrow [A,A]),$$
	which we again denote by the same letter $K$.
\end{proposition}

\subsection{Universal measuring comonoids and universal acting bimonoids and Hopf monoids in closed monoidal categories}
\label{SubsectionMeasuringActingClosed}

The propositions and remarks above make it possible to provide sufficient conditions for existence
of universal (co)measuring (co)monoids in terms of subobjects in $[A,B]$ and  universal (co)acting 
bimonoids and Hopf monoids in terms of submonoids in $[A,A]$.

Let $\mathcal C$ be a braided closed monoidal category satisfying the property dual to Property~\ref{PropertyEpiExtrMonoFactorizations} of Section~\ref{SubsectionSupportCoactingConditions}.

\begin{remark}\label{RemarkInBraidedClosedMonoidalManyPropertiesHoldAutomatically}
	Note that in a braided closed monoidal category the functor $M\otimes (-)$ is isomorphic to the functor $(-)\otimes M$,
	which is a left adjoint for every object $M$. Therefore both functors preserve all colimits and epimorphisms.
	In particular, the properties dual to Properties~\ref{PropertyMonomorphism}--\ref{PropertyEqualizers} of Section~\ref{SubsectionSupportCoactingConditions} hold automatically.
	Property~\ref{PropertyFreeMonoid} follows from~\cite[\S 7.3, Theorem 2]{MacLaneCatWork}.
\end{remark}
 
Fix $\Omega$-magmas $A$ and $B$ and a subobject $i \colon V \rightarrowtail [A,B]$. 
The comonoid ${}_\square \mathcal{C}(A,B,V)$ corresponding to an initial object $\psi_{A,B,V}$ in $\mathbf{Meas}(A,B)^\mathrm{op}_{K G_1'}(i)$ (where $G_1'$ is the forgetful functor $\mathbf{Meas}(A,B)\to \mathbf{TensMor}(A,B)$)
is called the \textit{$V$-universal measuring comonoid} from $A$ to $B$.
 
 In other words, ${}_\square\mathcal{C}(A,B,V)$ is a $V$-universal measuring comonoid if for every measuring $\psi \colon P \otimes A \to B$,
 such that $\cosupp \psi$ is a subobject of $V$, there exists a unique comonoid homomorphism
 $\varphi \colon P \to {}_\square\mathcal{C}(A,B,V)$ making the diagram below commutative:
 
 \begin{equation*}\xymatrix{
 		{P} \otimes A \ar[r]^(0.6){\psi}  \ar[d]_{\varphi\otimes{\id_A}} & B  \\
 		{{}_\square\mathcal{C}(A,B,V)} \otimes A \ar[ru]_(0.6){\psi_{A,B,V}} &
 }\end{equation*}

\begin{theorem}\label{TheoremDUALComonUnivMeasExistenceClosed}
Let $A$ and $B$ be $\Omega$-magmas and let $i \colon V \rightarrowtail [A,B]$ be a subobject
in a braided closed monoidal category $\mathcal C$ satisfying Properties~\ref{PropertySubObjectsSmallSet}, \ref{PropertyMonomorphism} and the properties dual to
Properties~\ref{PropertySmallLimits}, \ref{PropertyEpiExtrMonoFactorizations},   \ref{PropertyFreeMonoid} of Section~\ref{SubsectionSupportCoactingConditions}.
	 Then the initial object in $\mathbf{Meas}(A,B)^\mathrm{op}_{K G_1'}(i)$ indeed exists if $\mathbf{Meas}(A,B)^\mathrm{op}_{K G_1'}(i)$ is not empty. 
\end{theorem}
\begin{proof} Use the fact that $K$ is an isomorphism of categories and apply~\cite[Theorem~5.19]{AGV3} 
	and Remark~\ref{RemarkInBraidedClosedMonoidalManyPropertiesHoldAutomatically}.
\end{proof}	
\begin{remark}
If we take $i=\id_{[A,B]}$, then we get the measuring that is universal among all measurings from $A$ to $B$. In particular,
the corresponding comonoid is a generalization of the Sweedler universal measuring coalgebra.	
\end{remark}
\begin{corollary}[{\cite[Theorem 3.10]{AGV2}}]
	Let $A$ and $B$ be $\Omega$-algebras over a field $\mathbbm{k}$ and let $V\subseteq \mathbf{Vect}_\mathbbm{k}(A,B)$
	be a subspace. Then there exists $({}_\square \mathcal{C}(A,B,V),\psi_{A,B,V})$,
	which is called the \textit{$V$-universal measuring coalgebra} from $A$ to $B$.
\end{corollary}

\begin{theorem}\label{TheoremDUALBimonUnivActingExistenceClosed}
	Let $A$ be an $\Omega$-magma and let $i \colon V \to [A,A]$ be a submonoid (homomorphism of monoids that is a monomorphism in $\mathcal C$) in
	 a braided closed monoidal category $\mathcal C$ satisfying the property dual to Property~\ref{PropertyEpiExtrMonoFactorizations} of Section~\ref{SubsectionSupportCoactingConditions} such that there exists
	 the $V$-universal measuring comonoid ${}_\square \mathcal{C}(A,A,V)$.
	Then  ${}_\square \mathcal{B}(A,V):={}_\square \mathcal{C}(A,A,V)$ admits a structure of a bimonoid such that for any bimonoid $B$ and any action
	$\psi \colon B \otimes A \to A$, such that $\cosupp \psi$ is a submonoid of $V$, the unique comonoid homomorphism $\varphi$
	making the diagram below commutative is in fact a bimonoid homomorphism:
	\begin{equation*}\xymatrix{ B \otimes A \ar[r]^(0.6){\psi} \ar@{-->}[d]_{\varphi \otimes \id_A} & A \\
			{}_\square \mathcal{B}(A,V) \otimes A  \ar[ru]_{\psi_{A,V}}  } \end{equation*}
	(Here $\psi_{A,V} := \psi_{A,A,V}$.)
	In other words, ${}_\square \mathcal{B}(A,V)$ is the \textit{$V$-universal acting bimonoid on $A$}.
\end{theorem}
\begin{proof} 	Apply Proposition~\ref{PropositionMonoidActionMonoidHomomorphism}
	and~\cite[Theorem~5.23]{AGV3}.
\end{proof}	

\begin{theorem}\label{TheoremDUALBimonUnivActingExistenceExplicit} Suppose that a braided closed monoidal category $\mathcal C$
	satisfies Properties~\ref{PropertySubObjectsSmallSet}, \ref{PropertyMonomorphism} and the properties dual to
	Properties~\ref{PropertySmallLimits}, \ref{PropertyEpiExtrMonoFactorizations},   \ref{PropertyFreeMonoid} of Section~\ref{SubsectionSupportCoactingConditions}.
	Let $A$ be an $\Omega$-magma in $\mathcal C$ and let $i \colon V \to [A,A]$ be a submonoid.
	Then ${}_\square\mathcal{B}(A,V):={}_\square\mathcal{C}(A,A,V)$ admits a unique monoid structure turning
	$\psi_{A,V} := \psi_{A,A,V}$ into an action, which is the initial object in $\mathbf{Act}(A)^{\mathrm{op}}_{K G_1'G_3'}(i)$
	where $G_3'$ is the forgetful functor $\mathbf{Act}(A)\to \mathbf{Meas}(A,A)$.
\end{theorem}
\begin{proof} 
	Apply \cite[Remark 5.22, Corollary~5.24]{AGV3}, 
	Proposition~\ref{PropositionMonoidActionMonoidHomomorphism}
	and Remark~\ref{RemarkInBraidedClosedMonoidalManyPropertiesHoldAutomatically}.
\end{proof}

\begin{corollary}[{\cite[Theorem 4.2]{AGV2}}]
	Let $A$ be an $\Omega$-algebra over a field $\mathbbm{k}$ and let $V\subseteq \End_\mathbbm{k}(A)$ be a unital subalgebra.
	Then the $V$-universal measuring coalgebra ${}_\square \mathcal{B}(A,V) := {}_\square \mathcal{C}(A,A,V)$ admits a structure of a bialgebra such that for any bialgebra $B$ and any action
	$\psi \colon B \otimes A \to A$ such that $\cosupp \psi \subseteq V$ the unique coalgebra homomorphism $\varphi$
	making the diagram below commutative is in fact a bialgebra homomorphism:
	\begin{equation*}\xymatrix{ B \otimes A \ar[r]^(0.6){\psi} \ar@{-->}[d]_{\varphi \otimes \id_A} & A \\
		{}_\square \mathcal{B}(A,V) \otimes A  \ar[ru]_{\psi_{A,V}}  } \end{equation*}
	(Here $\psi_{A,V} := \psi_{A,A,V}$.)
\end{corollary}

Again, let $\mathcal C$ be a braided closed monoidal category satisfying the property dual to Property~\ref{PropertyEpiExtrMonoFactorizations} of Section~\ref{SubsectionSupportCoactingConditions}. Fix an $\Omega$-magma $A$ and a submonoid $i \colon V \to [A,A]$. The Hopf monoid ${}_\square\mathcal{H}(A,V)$
corresponding to an initial object $\psi_{A,V}^\mathbf{Hopf}$ in $\mathbf{HAct}(A)^\mathrm{op}_{K G_1'G_3'G_4'}(i)$
(where $G_4'$ is the forgetful functor $\mathbf{HAct}(A)\to \mathbf{Act}(A)$)
is called the \textit{$V$-universal acting Hopf monoid} on $A$.

In other words, ${}_\square\mathcal{H}(A,V)$ is a $V$-universal measuring Hopf monoid if for every action $\psi \colon H \otimes A \to H$ of a Hopf monoid $H$,
such that $\cosupp \psi$ is a subobject of $V$, there exists a unique Hopf monoid homomorphism
$\varphi \colon H \to {}_\square\mathcal{H}(A,V)$ making the diagram below commutative:

\begin{equation*}\xymatrix{
		{H} \otimes A \ar[r]^(0.6){\psi}  \ar[d]_{\varphi\otimes{\id_A}} & A  \\
		{{}_\square\mathcal{H}(A,V)} \otimes A \ar[ru]_(0.6){\psi_{A,V}^\mathbf{Hopf}} &
}\end{equation*}

\begin{theorem}\label{TheoremDUALHopfMonUnivActingExistenceClosed}
		Let $A$ be an $\Omega$-magma and let $i \colon V \to [A,A]$ be a submonoid in
	a braided closed monoidal category $\mathcal C$ satisfying the property dual to Property~\ref{PropertyEpiExtrMonoFactorizations} of Section~\ref{SubsectionSupportCoactingConditions}. 	Suppose that the forgetful functor $\mathbf{Hopf}(\mathcal C) \to \mathbf{Bimon}(\mathcal C)$
	admits a right adjoint functor $H_r \colon \mathbf{Bimon}(\mathcal C) \to \mathbf{Hopf}(\mathcal C)$
	and there exists ${}_\square\mathcal{B}(V)$.
	Then the initial object in $\mathbf{HAct}(A)^\mathrm{op}_{K G_1'G_3'G_4'}(i)$ indeed exists.
\end{theorem}
\begin{proof}
		Apply Proposition~\ref{PropositionMonoidActionMonoidHomomorphism}
	and~\cite[Theorem~5.27]{AGV3}. 
\end{proof}	
\begin{theorem}\label{TheoremDUALHopfMonUnivActingExistenceClosedExplicit}
	Let $A$ be an $\Omega$-magma and let $i \colon V \to [A,A]$ be a submonoid in
	a braided closed monoidal category $\mathcal C$ satisfying  
	Properties~\ref{PropertySubObjectsSmallSet}, \ref{PropertyMonomorphism} and the properties dual to
	Properties~\ref{PropertySmallLimits}--\ref{PropertyEpiExtrMonoFactorizations},   \ref{PropertyFreeMonoid} of Section~\ref{SubsectionSupportCoactingConditions}.
	 Then the initial object in $\mathbf{HAct}(A)^\mathrm{op}_{K G_1'G_3'G_4'}(i)$ indeed exists.
\end{theorem}
\begin{proof}
	Apply~\cite[Corollary~5.28]{AGV3}, Proposition~\ref{PropositionMonoidActionMonoidHomomorphism}
	and Remark~\ref{RemarkInBraidedClosedMonoidalManyPropertiesHoldAutomatically}.
\end{proof}	

\begin{remark}
	If we take $i=\id_{[A,A]}$, then we get the action that is universal among all Hopf monoid actions on $A$.
\end{remark}

Again, applying Theorem~\ref{TheoremDUALHopfMonUnivActingExistenceClosedExplicit} above, we recover the existence theorem for $V$-universal acting Hopf algebras proved in~\cite{AGV2}.

\subsection{Universal comeasuring monoids and universal coacting bimonoids and Hopf monoids in closed monoidal categories}
\label{SubsectionComeasuringCoactingClosed}
Let $\mathcal C$ be a braided closed monoidal category satisfying
the property dual to Property~\ref{PropertyEpiExtrMonoFactorizations} of Section~\ref{SubsectionSupportCoactingConditions}.

Fix $\Omega$-magmas $A$ and $B$ and a subobject $i \colon V \rightarrowtail [A,B]$. 
The monoid $\mathcal{A}^\square(A,B,V)$ corresponding to an initial object $\rho_{A,B,V}$ in $\mathbf{Comeas}(A,B)_{K (-)^\vee G_1}(i)$ (where $G_1$ is the forgetful functor $\mathbf{Comeas}(A,B)\to \mathbf{MorTens}(A,B)$) is called the \textit{$V$-universal comeasuring monoid} from $A$ to $B$.

In other words, $\mathcal{A}^\square(A,B,V)$ is a $V$-universal measuring comonoid if for every comeasuring $\rho \colon  A \to B \otimes Q$,
such that $\cosupp \left(\rho^\vee\right)$ is a subobject of $V$, there exists a unique monoid homomorphism
$\varphi \colon \mathcal{A}^\square(A,B,V) \to Q$ making the diagram below commutative:

$$\xymatrix{ A \ar[rr]^(0.3){\rho_{A,B,V}} 
	\ar[rrd]_{\rho}
	& & B \otimes \mathcal{A}^\square(A,B,V) \ar[d]^{\id_B \otimes \varphi} \\
	& & B \otimes Q} $$

\begin{theorem}\label{TheoremMonUnivComeasExistenceClosed}
	Let $\mathcal C$ be a braided closed monoidal category satisfying
	Properties~\ref{PropertySmallLimits}, \ref{PropertySubObjectsSmallSet}--\ref{PropertyLimitsOfSubobjectsArePreserved},
		 \ref{PropertySwitchProdTensorIsAMonomorphism}, \ref{PropertyEqualizers}
		 and the properties dual to Properties~\ref{PropertySmallLimits} and \ref{PropertySubObjectsSmallSet}
		  of Sections~\ref{SubsectionSupportCoactingConditions}. 
	 Suppose, in addition, that
	 \begin{itemize} \item the functor $(-)^*$ maps extremal monomorphisms to extremal epimorphisms;
	\item
	$\alpha_P$ is an extremal monomorphism and $\theta_{B^*,P}$ is a monomorphism for all objects~$P$.
	\end{itemize}
	Let $A$ and $B$ be $\Omega$-magmas  and let  $i\colon V \rightarrowtail [A,B]$
	be a subobject such that $i=K(\rho_U^\vee)$ for some morphism $\rho_U \colon A \to B \otimes U$ where $U$ is an object in $\mathcal C$. Then there exists an initial object in $\mathbf{Comeas}(A,B)_{K (-)^\vee G_1}(i)$ if $\mathbf{Comeas}(A,B)_{K (-)^\vee G_1}(i)$ is not empty.
\end{theorem}
\begin{proof} Recall that by~\cite[Proposition 4.2 (1)]{AGV3} 
	and Remark~\ref{RemarkInBraidedClosedMonoidalManyPropertiesHoldAutomatically} the category $\mathcal C$ satisfies Properties~\ref{PropertyEpiExtrMonoFactorizations}, \ref{PropertyFreeMonoid} and the properties dual to Properties~\ref{PropertyEpiExtrMonoFactorizations} and~\ref{PropertyMonomorphism}--\ref{PropertyEqualizers}. By Theorems~\ref{TheoremSupportsUnderVee}--\ref{TheoremVeeReflectsCoarserFiner} and Remarks~\ref{Cosupp_duality_rmk} (\ref{RemarkCosupportSubobjectPartialOrder}) we are under the assumptions of Proposition~\ref{PropositionLIOKPreservesPreorder}.
Now we apply~\cite[Theorem~4.24]{AGV3}. 
\end{proof}	

\begin{remark} If $V$ is an arbitrary subobject of $[A,B]$, say, if $V=[A,B]$, then ${}_\square \mathcal{C}(A,B,V)$
	may not exist, see~\cite[Section~4.5]{AGV2} for the corresponding examples.
\end{remark}
\begin{corollary}[{\cite[Theorem 3.16]{AGV2}}]
Let $A$ and $B$ be $\Omega$-algebras over a field $\mathbbm{k}$ and let $V\subseteq \mathbf{Vect}_\mathbbm{k}(A,B)$
be a pointwise finite dimensional subspace closed in the finite topology (see the definitions in~\cite[Section 2.2]{AGV2}). Then there exists the \textit{$V$-universal comeasuring algebra} ${}_\square \mathcal{C}(A,B,V)$ from $A$ to $B$.
\end{corollary}
\begin{proof}
	By~\cite[Theorem 2.11]{AGV2} and Theorem~\ref{TheoremSupportsUnderVee} there exists a morphism $\rho_U \colon A \to B \otimes U$ such that $i=K(\rho_U^\vee)$ if and only if $V$ is pointwise finite dimensional and closed in the finite topology.
\end{proof}	
\begin{remark}
In~\cite{AGV2} the notion of the cosupport was introduced for linear maps $\rho \colon A \to B \otimes Q$ too. Namely, by the definition, $\cosupp \rho := \cosupp(\rho^\vee)$. 
\end{remark}

In order to proceed to coactions we need the following lemmas:
\begin{lemma}\label{LemmaComeasVUnity}
	Let $A$ be an $\Omega$-magma in a braided closed monoidal category~$\mathcal C$
	and let $i\colon V \to [A,A]$ be a submonoid. Then the morphism $A	\mathrel{\widetilde\to} A  \otimes \mathbbm{1}$ is an object in $\mathbf{Comeas}(A,A)_{K (-)^\vee G_1}(i)$.
\end{lemma}
\begin{proof}
	By~\cite[Lemma 4.33]{AGV3}, 
	 $\rho_0 \colon A	\mathrel{\widetilde\to} A  \otimes \mathbbm{1}$ is a comeasuring.
	 Denote by $u \colon \mathbbm{1} \to V$ the unit of $V$ and
	 consider the diagram	 
	 $$\xymatrix{\mathbbm{1}^* \otimes A \ar[r]^(0.4){\mathbbm{1}^* \otimes \rho_0} \ar[dd]_{r_{\mathbbm{1}^*} \otimes \id_{A}} &  \mathbbm{1}^* \otimes A \otimes \mathbbm{1} \ar[rr]^{c_{\mathbbm{1}^* \otimes A} \otimes \id_{\mathbbm{1}}} & &  A \otimes \mathbbm{1}^* \otimes \mathbbm{1}
	 	\ar[dd]^{\id_A \otimes \mathrm{ev_{\mathbbm{1}}}} \\ & & \\
	 	\mathbbm{1}^* \otimes \mathbbm{1} \otimes A \ar[ruu]^(0.6){\id_{\mathbbm{1}^*} \otimes c_{\mathbbm{1}, A}\,} \ar[d]^{\mathrm{ev}_{\mathbbm{1}} \otimes \id_{A}}
	 	\ar[rrruu]_{\ c_{\mathbbm{1}^* \otimes \mathbbm{1}, A}}
	 	 & & & A\otimes \mathbbm{1} \ar[d]^\sim \\
	 	\mathbbm{1} \otimes A \ar[rd]_{u\otimes A} \ar[rrr]_\sim \ar[rrru]_{c_{\mathbbm{1}, A}} & & & A \\
	 	&  V \otimes A \ar[rru]_(0.6){K^{-1}i} & 
	 }$$
 
 The lower triangle is commutative since $A$ is a $V$-module by Proposition~\ref{PropositionMonoidActionMonoidHomomorphism}.
 The other inner polygons are commutative by the properties of the braiding. Hence the outer polygon is commutative too.
 Now we notice the composition on the upper and the right edges equals $\rho_0^\vee \colon \mathbbm{1}^* \otimes A	\to A$.
 Therefore, the diagram below is commutative too:
 $$\xymatrix{ \mathbbm{1}^*\otimes A \ar[rd]_{(u\,\mathrm{ev}_{\mathbbm{1}}\, r_{\mathbbm{1}^*})\otimes \id_{A}} \ar[rr]^{\rho_0^\vee}  & & A \\
 	&  V \otimes A \ar[ru]_(0.6){K^{-1}i} &
 }$$ In particular, $\rho_0 \colon A	\mathrel{\widetilde\to} A  \otimes \mathbbm{1}$ is an object in $\mathbf{Comeas}(A,A)_{K (-)^\vee G_1}(i)$. 
\end{proof}	
 
\begin{lemma}\label{LemmaComeasVComposition}
	Let $A$ be an $\Omega$-magma in a braided closed monoidal category~$\mathcal C$
	satisfying Properties~\ref{PropertySmallLimits},
	 \ref{PropertySubObjectsSmallSet}--\ref{PropertyLimitsOfSubobjectsArePreserved} and \ref{PropertyEqualizers}
	and the properties dual to Properties~\ref{PropertySmallLimits} and~\ref{PropertySubObjectsSmallSet}
	 of Section~\ref{SubsectionSupportCoactingConditions}
	and let $i\colon V \to [A,A]$ be a submonoid such that $i=K(\rho_U^\vee)$ for some morphism $\rho_U \colon A \to A \otimes U$ where $U$ is an object in $\mathcal C$. Suppose, in addition, that
	\begin{itemize}
	 \item the functor $(-)^*$ maps extremal monomorphisms to extremal epimorphisms;
	\item $\alpha_A \otimes \id_{P^*}$ and $\theta_{A^*,P}$ are monomorphisms for all objects $P$.
	\end{itemize}
	 If $\rho_1 \colon A \to A \otimes Q_1$ and $\rho_2 \colon A \to A \otimes Q_2$
	are objects in $\mathbf{Comeas}(A,A)_{K (-)^\vee G_1}(i)$ such that
	$\left(\theta^\mathrm{inv}_{Q_1,Q_2}\right)^\flat \colon Q_1\otimes Q_2 \to (Q_1^* \otimes Q_2^*)^*$ is an extremal monomorphism in $\mathcal C$, then $(\rho_1 \otimes {\id_{Q_2}} )\rho_2$ is an object in
	$\mathbf{Comeas}(A,A)_{K (-)^\vee G_1}(i)$ too.
\end{lemma}
\begin{proof}  By~\cite[Lemma 4.33]{AGV3}, 
	 the morphism $(\rho_1 \otimes {\id_{Q_2}} )\rho_2$
	is a comeasuring.
	
By Lemma~\ref{LemmaVeeCompositionSwitch},
\begin{equation}\label{EqComeasV}\left((\rho_1 \otimes \id_{Q_2})\rho_2\right)^\vee (\theta^\mathrm{inv}_{Q_1,Q_2} \otimes \id_A) = 
\rho_1^\vee (\id_{Q_1^*} \otimes \rho_2^\vee ).\end{equation}
Applying the dual of~\cite[Lemma~4.34]{AGV3} 
 we obtain that the right hand side of \eqref{EqComeasV}
is an object in $\mathbf{Meas}(A,A)^\mathrm{op}_{K G_1'}(i)$. In particular, 
$\rho_1^\vee (\id_{Q_1^*} \otimes \rho_2^\vee ) = \rho_U^\vee(f\otimes \id_A)$
for some morphism  $f \colon Q_1^* \otimes Q_2^* \to V$ where $V=U^*$.

In virtue of Proposition~\ref{PropositionDualityNablaVeeSharp},
we have 
\begin{equation*}\begin{split}\bigl((\id_A \otimes \left( \theta^\mathrm{inv}_{Q_1,Q_2}\right)^\flat)(\rho_1 \otimes \id_{Q_2})\rho_2\bigr)^\nabla = \left((\rho_1 \otimes \id_{Q_2})\rho_2\right)^\vee (\theta^\mathrm{inv}_{Q_1,Q_2} \otimes \id_A)=\\=
\rho_1^\vee (\id_{Q_1^*} \otimes \rho_2^\vee ) = \rho_U^\vee(f\otimes \id_A)=
\left((\id_A \otimes f^\flat)\rho_U\right)^\nabla.
\end{split}
\end{equation*}

By Corollary~\ref{CorollaryRhoABPStarEqualAfterNabla}, 
$$(\id_A \otimes \left( \theta^\mathrm{inv}_{Q_1,Q_2}\right)^\flat)(\rho_1 \otimes \id_{Q_2})\rho_2 = (\id_A \otimes f^\flat)\rho_U.$$

By~\cite[Proposition~4.18]{AGV3}, 
$\supp \left( (\id_A \otimes f^\flat)\rho_U \right) = \supp (\rho_1 \otimes \id_{Q_2})\rho_2$
and $(\rho_1 \otimes \id_{Q_2})\rho_2 \preccurlyeq \rho_U$.
By Theorem~\ref{TheoremVeePreservesCoarserFiner}, 
$\bigl((\rho_1 \otimes \id_{Q_2})\rho_2\bigr)^\vee \preccurlyeq \rho_U^\vee$
and $(\rho_1 \otimes \id_{Q_2})\rho_2$ is an object in $\mathbf{Comeas}(A,A)_{K (-)^\vee G_1}(i)$.
\end{proof}	

\begin{theorem}\label{TheoremBimonUnivCoactingExistenceClosed}
	Let $A$ be an $\Omega$-magma in a braided closed monoidal category~$\mathcal C$
	satisfying 	
	Properties~\ref{PropertySmallLimits},
	\ref{PropertySubObjectsSmallSet}--\ref{PropertyLimitsOfSubobjectsArePreserved} and \ref{PropertyEqualizers}
	and the properties dual to Properties~\ref{PropertySmallLimits} and~\ref{PropertySubObjectsSmallSet}
	of Section~\ref{SubsectionSupportCoactingConditions}
		and let $i\colon V \to [A,A]$ be a submonoid such that $i=K(\rho_U^\vee)$ for some morphism $\rho_U \colon A \to A \otimes U$ where $U$ is an object in $\mathcal C$. Suppose, in addition, that 
	\begin{itemize}
		\item the functor $(-)^*$ maps extremal monomorphisms to extremal epimorphisms;
		\item $\alpha_A \otimes \id_{P^*}$ and $\theta_{A^*,P}$ are monomorphisms for all objects $P$;
		\item  there exists the $V$-universal comeasuring monoid $\mathcal{B}^\square(A,V) := \mathcal{A}^\square(A,A,V)$;
		\item $\left(\theta^\mathrm{inv}_{\mathcal{B}^\square(A,V),\mathcal{B}^\square(A,V)}\right)^\flat \colon \mathcal{B}^\square(A,V)\otimes \mathcal{B}^\square(A,V) \to (\mathcal{B}^\square(A,V)^* \otimes \mathcal{B}^\square(A,V)^*)^*$ is an extremal monomorphism in $\mathcal C$.
	\end{itemize}
	Then $\mathcal{B}^\square(A,V)$ admits a structure of a bimonoid such that for any bimonoid $B$ and any coaction $\rho \colon  A \to A \otimes B$, such that $\cosupp(\rho^\vee)$ is a submonoid of $V$, the unique monoid homomorphism $\varphi$
	making the diagram below commutative is in fact a bimonoid homomorphism:
	\begin{equation*}\xymatrix{ A \ar[r]^(0.3){\rho_{A,V}} \ar[rd]_\rho & A \otimes \mathcal{B}^\square(A,V) \ar@{-->}[d]^{\id_A\otimes
				\varphi} \\
			& A\otimes B} \end{equation*}
	(Here $\rho_{A,V} := \rho_{A,A,V}$.)
	In other words, $\mathcal{B}^\square(A,V)$ is the \textit{$V$-universal coacting bimonoid on $A$}.
\end{theorem}
\begin{proof} By Proposition~\ref{PropositionVeeUnderComposition}, Lemmas~\ref{LemmaComeasVUnity} and~\ref{LemmaComeasVComposition} the category $\mathcal D = \mathbf{Comeas}(A,A)_{K (-)^\vee G_1}(i)$
	satisfies the assumptions of~\cite[Theorem~4.31]{AGV3}. 
\end{proof}	

\begin{corollary}\label{CorollaryBimonUnivCoactingExistenceClosed}  Let $A$ be an $\Omega$-magma in a braided closed monoidal category~$\mathcal C$
	satisfying 
	Properties~\ref{PropertySmallLimits},
	\ref{PropertySubObjectsSmallSet}--\ref{PropertyLimitsOfSubobjectsArePreserved},
	\ref{PropertySwitchProdTensorIsAMonomorphism}, \ref{PropertyEqualizers}
	and the properties dual to Properties~\ref{PropertySmallLimits} and~\ref{PropertySubObjectsSmallSet}
	of Section~\ref{SubsectionSupportCoactingConditions}
		and let $i\colon V \to [A,A]$ be a submonoid such that $i=K(\rho_U^\vee)$ for some morphism $\rho_U \colon A \to A \otimes U$ where $U$ is an object in $\mathcal C$. Suppose, in addition, that 
	\begin{itemize}
		\item the functor $(-)^*$ maps extremal monomorphisms to extremal epimorphisms;
		\item $\alpha_P$ is an extremal monomorphism and $\theta_{A^*,P}$ is a monomorphism for all objects $P$;
		\item $\left(\theta^\mathrm{inv}_{B,B}\right)^\flat \colon B\otimes B \to (B^* \otimes B^*)^*$  is an extremal monomorphism in $\mathcal C$
		for all bimonoids~$B$.
	\end{itemize}
	Then $\mathcal{B}^\square(A,V) := \mathcal{A}^\square(A,A,V)$ admits a unique comonoid structure turning
	$\rho_{A,V} := \rho_{A,A,V}$ into a coaction, which is the initial object in $\mathbf{Coact}(A)_{K (-)^\vee G_1 G_3}(i)$ where $G_3$ is the forgetful functor $\mathbf{Coact}(A)\to \mathbf{Comeas}(A,A)$.
\end{corollary}
\begin{proof} Apply Theorems~\ref{TheoremMonUnivComeasExistenceClosed} and~\ref{TheoremBimonUnivCoactingExistenceClosed}.
\end{proof}	

Again fix an $\Omega$-magma $A$ in $\mathcal C$.
Let  $i\colon V \to [A,A]$ be a submonoid.
We call the Hopf monoid 
corresponding to the initial object
in $\mathbf{HCoact}(A)_{K (-)^\vee G_1 G_3 G_4}(i)$ (if it exists) the \textit{$V$-universal coacting Hopf monoid} on $A$.

\begin{theorem}\label{TheoremHopfMonUnivCoactingExistenceClosed}
		Let $A$ be an $\Omega$-magma in a braided closed monoidal category~$\mathcal C$ 
	and let $i\colon V \to [A,A]$ be a submonoid
	such that there exists the $V$-universal comeasuring bimonoid $\mathcal{B}^\square(A,V)$.
	Suppose, in addition, that the forgetful functor $\mathbf{Hopf}(\mathcal C) \to \mathbf{Bimon}(\mathcal C)$
	admits a left adjoint functor $H_l \colon  \mathbf{Bimon}(\mathcal C) \to \mathbf{Hopf}(\mathcal C)$. Then the initial object in $\mathbf{HCoact}(A)_{K (-)^\vee G_1 G_3 G_4}(i)$ indeed exists.
\end{theorem}
\begin{proof} Apply~\cite[Theorem 4.41]{AGV3}.
\end{proof}	

Applying Theorem~\ref{TheoremHopfMonUnivCoactingExistenceClosed} above, we recover here the existence theorem for $V$-universal coacting Hopf algebras proved in~\cite{AGV2}.

\begin{corollary}\label{CorollaryHopfMonUnivCoactingExistenceClosed}  Let $A$ be an $\Omega$-magma in a braided closed monoidal category~$\mathcal C$
	satisfying Properties~\ref{PropertySmallLimits},
	\ref{PropertySubObjectsSmallSet}--\ref{PropertyLimitsOfSubobjectsArePreserved}, \ref{PropertySwitchProdTensorIsAMonomorphism}, \ref{PropertyEqualizers}
	and the properties dual to Properties~\ref{PropertySmallLimits} and~\ref{PropertySubObjectsSmallSet}
	of Section~\ref{SubsectionSupportCoactingConditions}
	and let $i\colon V \to [A,A]$ be a submonoid such that $i=K(\rho_U^\vee)$ for some morphism $\rho_U \colon A \to A \otimes U$ where $U$ is an object in $\mathcal C$. Suppose, in addition, that 
	\begin{itemize}
		\item the functor $(-)^*$ maps extremal monomorphisms to extremal epimorphisms;
		\item $\alpha_P$ is an extremal monomorphism and $\theta_{A^*,P}$ is a monomorphism for all objects $P$;
		\item $\left(\theta^\mathrm{inv}_{B,B}\right)^\flat \colon B\otimes B \to (B^* \otimes B^*)^*$  is an extremal monomorphism in $\mathcal C$
		for all bimonoids~$B$.
	\end{itemize}
	Then the initial object in $\mathbf{HCoact}(A)_{K (-)^\vee G_1 G_3 G_4}(i)$ indeed exists.
\end{corollary}
\begin{proof}
	Apply~\cite[Theorem~4.6]{AGV3}, 
	Theorem~\ref{TheoremHopfMonUnivCoactingExistenceClosed}, Remark~\ref{RemarkInBraidedClosedMonoidalManyPropertiesHoldAutomatically} and Corollary~\ref{CorollaryBimonUnivCoactingExistenceClosed}.
\end{proof}	

\subsection{Duality in closed monoidal categories}\label{SubsectionDualityClosed}

\begin{theorem}\label{Theorem(Co)monUniv(Co)measDualityClosed}
	Let $\mathcal C$ be a braided closed monoidal category satisfying
	Properties~\ref{PropertySmallLimits},
	\ref{PropertySubObjectsSmallSet}--\ref{PropertyLimitsOfSubobjectsArePreserved} and~\ref{PropertyEqualizers}
	and the properties dual to Properties~\ref{PropertySmallLimits} and~\ref{PropertySubObjectsSmallSet}
	of Section~\ref{SubsectionSupportCoactingConditions}.
Suppose, in addition, that
\begin{itemize} \item the functor $(-)^*$ maps extremal monomorphisms to extremal epimorphisms;
	\item there exists the functor $(-)^\circ$;
	\item 
$\alpha_M$ is an extremal monomorphism and $\theta_{M,N}$ is a monomorphism
for all objects $M,N$.
\end{itemize}
Let $A$ and $B$ be $\Omega$-magmas  and let  $i\colon V \rightarrowtail [A,B]$
be a subobject such that $i=K(\rho_U^\vee)$ for some morphism $\rho_U \colon A \to B \otimes U$ where $U$ is an object in $\mathcal C$ and there exist	 
	$\mathcal{A}^\square(A,B,V)$ and ${}_\square \mathcal{C}(A,B,V)$.
	Then $$\rho_{A,B,V}^\vee(\varkappa_{\mathcal{A}^\square(A,B,V)}\otimes \id_A)
	\colon \mathcal{A}^\square(A,B,V)^\circ \otimes A \to B$$ is a measuring and
	the unique comonoid homomorphism $\beta \colon \mathcal{A}^\square(A,B,V)^\circ \to {}_\square \mathcal{C}(A,B,V)$ making
	the diagram
	\begin{equation}\label{Eq(Co)monUniv(Co)measDualityClosed}
	 \xymatrix{
		\mathcal{A}^\square(A,B,V)^\circ \otimes A  \ar@{-->}[d]_{\beta \otimes \id_A}
		\ar[rr]^{\varkappa_{\mathcal{A}^\square(A,B,V)}\otimes \id_A}
		& & \mathcal{A}^\square(A,B,V)^* \otimes A \ar[d]^{\rho_{A,B,V}^\vee} \\
		{}_\square \mathcal{C}(A,B,V) \otimes A \ar[rr]^(0.6){\psi_{A,B,V}}& \quad & B \\} 
	\end{equation}  commutative is a comonoid isomorphism.
\end{theorem}
\begin{proof}
By Theorems~\ref{TheoremSupportsUnderVee}--\ref{TheoremVeeReflectsCoarserFiner} and Remarks~\ref{Cosupp_duality_rmk} (\ref{RemarkCosupportSubobjectPartialOrder}) and~\ref{RemarkInBraidedClosedMonoidalManyPropertiesHoldAutomatically} we are under the assumptions of Proposition~\ref{PropositionLIOKPreservesPreorder}.
Now we apply Theorem~\ref{Theorem(Co)monUniv(Co)measDuality}.
\end{proof}	

\begin{corollary}\label{Corollary(Co)monUniv(Co)measDualityClosed}
	Let~$\mathcal C$ be a  braided closed monoidal category
	satisfying Properties~\ref{PropertySmallLimits},
	\ref{PropertySubObjectsSmallSet}--\ref{PropertyLimitsOfSubobjectsArePreserved}, \ref{PropertySwitchProdTensorIsAMonomorphism}, \ref{PropertyEqualizers}
	and the properties dual to Properties~\ref{PropertySmallLimits}, \ref{PropertySubObjectsSmallSet} and~\ref{PropertyFreeMonoid}
	of Section~\ref{SubsectionSupportCoactingConditions}
	such that 	\begin{itemize}
		\item the functor $(-)^*$ maps extremal monomorphisms to extremal epimorphisms;
		\item $\alpha_M$ is an extremal monomorphism and $\theta_{M,N}$ is a monomorphism
for all objects $M,N$;
	\end{itemize}
Let $A$ and $B$ be $\Omega$-magmas  and let  $i\colon V \rightarrowtail [A,B]$
be a subobject such that $i=K(\rho_U^\vee)$ for some morphism $\rho_U \colon A \to B \otimes U$ where $U$ is an object in $\mathcal C$ and $\mathbf{Comeas}(A,B)_{K (-)^\vee G_1}(i)$ is not empty.
Then $$\rho_{A,B,V}^\vee(\varkappa_{\mathcal{A}^\square(A,B,V)}\otimes \id_A)
\colon \mathcal{A}^\square(A,B,V)^\circ \otimes A \to B$$ is a measuring and
the unique comonoid homomorphism $\beta \colon \mathcal{A}^\square(A,B,V)^\circ \to {}_\square \mathcal{C}(A,B,V)$ making
the diagram~\eqref{Eq(Co)monUniv(Co)measDualityClosed} commutative is a comonoid isomorphism.
\end{corollary}
\begin{proof}
	Apply Theorems~\ref{TheoremDUALComonUnivMeasExistenceClosed}, \ref{TheoremMonUnivComeasExistenceClosed}, \ref{Theorem(Co)monUniv(Co)measDualityClosed}
	and Remark~\ref{RemarkCircSufficientConditionsForExistence}. 
\end{proof}		

\begin{corollary}[{\cite[Theorem 3.20]{AGV2}}] Let $A$ and $B$ be $\Omega$-algebras  over a field $\mathbbm{k}$ and let $V \subseteq \mathbf{Vect}_\mathbbm{k}(A,B)$ be a pointwise finite dimensional subspace closed in the finite topology.  
	Then $\beta$ from the diagram~\eqref{Eq(Co)monUniv(Co)measDualityClosed} is a coalgebra isomorphism.
\end{corollary}

\begin{theorem}\label{TheoremBimonUniv(Co)actDualityClosed}
	Let $A$ be an $\Omega$-magma in a
	symmetric closed monoidal category~$\mathcal C$
	satisfying
	Properties~\ref{PropertySmallLimits},
	\ref{PropertySubObjectsSmallSet}--\ref{PropertyLimitsOfSubobjectsArePreserved} and~\ref{PropertyEqualizers}
	and the properties dual to Properties~\ref{PropertySmallLimits} and~\ref{PropertySubObjectsSmallSet}
	of Section~\ref{SubsectionSupportCoactingConditions}
	and let $i\colon V \to [A,A]$ be a submonoid such that $i=K(\rho_U^\vee)$ for some morphism $\rho_U \colon A \to A \otimes U$ where $U$ is an object in $\mathcal C$. Suppose, in addition, that 
	\begin{itemize} \item the functor $(-)^*$ maps extremal monomorphisms to extremal epimorphisms;
		\item there exists the functor $(-)^\circ$;
	\item 
$\alpha_M$ is an extremal monomorphism and $\theta_{M,N}$ is a monomorphism
for all objects $M,N$;
\item there exist	 
$\mathcal{B}^\square(A,V)=\mathcal{A}^\square(A,A,V)$ and ${}_\square \mathcal{B}(A,V)={}_\square \mathcal{C}(A,A,V)$;
\item $\left(\theta^\mathrm{inv}_{\mathcal{B}^\square(A,V),\mathcal{B}^\square(A,V)}\right)^\flat \colon \mathcal{B}^\square(A,V)\otimes \mathcal{B}^\square(A,V) \to (\mathcal{B}^\square(A,V)^* \otimes \mathcal{B}^\square(A,V)^*)^*$ is an extremal monomorphism in $\mathcal C$.
	\end{itemize}
	Then $$\rho_{A,V}^\vee(\varkappa_{\mathcal{B}^\square(A,V)}\otimes \id_A)
	\colon \mathcal{B}^\square(A,V)^\circ \otimes A \to A$$ is an action and
	the unique bimonoid homomorphism $\beta \colon \mathcal{B}^\square(A,V)^\circ \to {}_\square \mathcal{B}(A,V)$ making
	the diagram
	\begin{equation}\label{EqBimonUniv(Co)actDualityClosed}
	\xymatrix{
		\mathcal{B}^\square(A,V)^\circ \otimes A  \ar@{-->}[d]_{\beta \otimes \id_A}
		\ar[rr]^{\varkappa_{\mathcal{B}^\square(A,V)}\otimes \id_A}
		& & \mathcal{B}^\square(A,V)^* \otimes A \ar[d]^{\rho_{A,V}^\vee} \\
		{}_\square \mathcal{B}(A,V) \otimes A \ar[rr]^(0.6){\psi_{A,V}}& & A \\}
	\end{equation}
	commutative is a bimonoid isomorphism.	
\end{theorem}
\begin{proof}
By Theorems~\ref{TheoremDUALBimonUnivActingExistenceClosed} and~\ref{TheoremBimonUnivCoactingExistenceClosed}
the bimonoids $\mathcal{B}^\square(A,V) = \mathcal{A}^\square(A,A,V)$ and ${}_\square \mathcal{B}(A,V)={}_\square \mathcal{C}(A,A,V)$ are indeed $V$-universal.
The morphism $\rho_{A,V}^\vee(\varkappa_{\mathcal{B}^\square(A,V)}\otimes \id_A)$ is an action by Theorem~\ref{Theorem(Co)monUniv(Co)measDualityClosed} and Lemmas~\ref{LemmaStarComoduleToAModule} and~\ref{LemmaKappaForBimonoids}.
Hence the existence of $\beta$ follows from the definition of ${}_\square \mathcal{B}(A,V)$. By
Theorem~\ref{Theorem(Co)monUniv(Co)measDualityClosed} the morphism $\beta$ is an isomorphism in~$\mathcal C$.
Combined with the fact that $\beta$ is a bimonoid homomorphism, this implies that $\beta$ is a bimonoid isomorphism.
\end{proof}	

\begin{corollary}
	Let~$\mathcal C$ be a  symmetric closed monoidal category
	satisfying Properties~\ref{PropertySmallLimits},
	\ref{PropertySubObjectsSmallSet}--\ref{PropertyLimitsOfSubobjectsArePreserved}, \ref{PropertySwitchProdTensorIsAMonomorphism}, \ref{PropertyEqualizers}
	and the properties dual to Properties~\ref{PropertySmallLimits}, \ref{PropertySubObjectsSmallSet} and~\ref{PropertyFreeMonoid}
	of Section~\ref{SubsectionSupportCoactingConditions}
	such that
	\begin{itemize}
		\item the functor $(-)^*$ maps extremal monomorphisms to extremal epimorphisms;
		\item 
		$\alpha_M$ is an extremal monomorphism and $\theta_{M,N}$ is a monomorphism
		for all objects $M,N$;
		\item $\left(\theta^\mathrm{inv}_{B,B}\right)^\flat \colon B\otimes B \to (B^* \otimes B^*)^*$  is an extremal monomorphism in $\mathcal C$
		for all bimonoids~$B$
	\end{itemize}
	and let $i\colon V \to [A,A]$ be a submonoid such that $i=K(\rho_U^\vee)$ for some morphism $\rho_U \colon A \to A \otimes U$ where $U$ is an object in $\mathcal C$.
	Then $$\rho_{A,V}^\vee(\varkappa_{\mathcal{B}^\square(A,V)}\otimes \id_A)
	\colon \mathcal{B}^\square(A,V)^\circ \otimes A \to A$$ is an action and
	the unique bimonoid homomorphism $\beta \colon \mathcal{B}^\square(A,V)^\circ \to {}_\square \mathcal{B}(A,V)$ making
	the diagram~\eqref{EqBimonUniv(Co)actDualityClosed} commutative is a bimonoid isomorphism.	
\end{corollary}
\begin{proof}
	Apply Theorems~\ref{TheoremDUALComonUnivMeasExistenceClosed}, \ref{TheoremMonUnivComeasExistenceClosed}, \ref{TheoremBimonUniv(Co)actDualityClosed}
	and Remark~\ref{RemarkCircSufficientConditionsForExistence}. 
\end{proof}

\begin{corollary}[{\cite[Theorem 4.14]{AGV2}}]
	Let $A$ be an $\Omega$-algebra  over a field $\mathbbm{k}$ and let $V\subseteq \End_\mathbbm{k}(A)$ be a unital pointwise finite dimensional subalgebra  closed in the finite topology.
	Then $\beta$ from the diagram~\eqref{EqBimonUniv(Co)actDualityClosed}
	 is a bialgebra isomorphism.
\end{corollary}

\begin{theorem}\label{TheoremHopfMonUniv(Co)actDualityClosed}
	Let~$\mathcal C$ be a symmetric closed monoidal category 
	satisfying
	Properties~\ref{PropertySmallLimits},
	\ref{PropertySubObjectsSmallSet}--\ref{PropertyLimitsOfSubobjectsArePreserved} and~\ref{PropertyEqualizers}
	and the properties dual to Properties~\ref{PropertySmallLimits} and~\ref{PropertySubObjectsSmallSet}
	of Section~\ref{SubsectionSupportCoactingConditions}
	and let $i\colon V \to [A,A]$ be a submonoid such that $i=K(\rho_U^\vee)$ for some morphism $\rho_U \colon A \to A \otimes U$ where $U$ is an object in $\mathcal C$. Suppose, in addition, that
	\begin{itemize} \item the functor $(-)^*$ maps extremal monomorphisms to extremal epimorphisms;
	\item there exist functors $(-)^\circ$, $H_l$ and $H_r$;
	\item 
		$\alpha_M$ is an extremal monomorphism and $\theta_{M,N}$ is a monomorphism
		for all objects $M,N$;
		\item there exist	 
		$\mathcal{B}^\square(A,V)=\mathcal{A}^\square(A,A,V)$ and ${}_\square \mathcal{B}(A,V)={}_\square \mathcal{C}(A,A,V)$;
		\item $\left(\theta^\mathrm{inv}_{\mathcal{B}^\square(A,V),\mathcal{B}^\square(A,V)}\right)^\flat \colon \mathcal{B}^\square(A,V)\otimes \mathcal{B}^\square(A,V) \to (\mathcal{B}^\square(A,V)^* \otimes \mathcal{B}^\square(A,V)^*)^*$ is an extremal monomorphism in $\mathcal C$;
		\item $\varkappa_H$ is a monomorphism in $\mathcal C$ for every Hopf monoid $H$.
	\end{itemize}
Denote by 
$\rho_{A,V}^\mathbf{Hopf} \colon A \to A \otimes \mathcal{H}^\square(A,V) $
and $\psi_{A,V}^\mathbf{Hopf} \colon {}_\square \mathcal{H}(A,V) \otimes A \to A$
the initial objects in $\mathbf{HCoact}(A)_{K (-)^\vee G_1 G_3 G_4}(i)$ and $\mathbf{HAct}(A)^\mathrm{op}_{K G_1'G_3'G_4'}(i)$, respectively.
	Then $$\left(\rho_{A,V}^\mathbf{Hopf}\right)^\vee(\varkappa_{\mathcal{H}^\square(A,V)}\otimes \id_A)
	\colon \mathcal{H}^\square(A,V)^\circ \otimes A \to A$$ is an action and
	the unique Hopf monoid homomorphism $\beta^{\mathbf{Hopf}} \colon \mathcal{H}^\square(A,V)^\circ \to {}_\square \mathcal{H}(A,V)$ making
	the diagram
	\begin{equation}
	\label{EqHopfMonUniv(Co)actDualityClosed}
	\xymatrix{
		\mathcal{H}^\square(A,V)^\circ \otimes A  \ar@{-->}[d]_{\beta^{\mathbf{Hopf}} \otimes \id_A}
		\ar[rr]^{\varkappa_{\mathcal{H}^\square(A,V)}\otimes \id_A}
		& & \mathcal{H}^\square(A,V)^* \otimes A \ar[d]^{\left(\rho_{A,V}^\mathbf{Hopf}\right)^\vee} \\
		{}_\square \mathcal{H}(A,V) \otimes A \ar[rr]^(0.6){\psi_{A,V}^\mathbf{Hopf}}& & A \\}
	\end{equation}
	commutative is a Hopf monoid isomorphism.		
\end{theorem}
\begin{proof}
We repeat verbatim the proof of Theorem~\ref{TheoremHopfMonUniv(Co)actDuality}
using Theorems~\ref{TheoremDUALHopfMonUnivActingExistenceClosed}, \ref{TheoremHopfMonUnivCoactingExistenceClosed} and~\ref{TheoremBimonUniv(Co)actDualityClosed} instead of~\cite[Theorems~4.42 and~5.27]{AGV3}
  and Theorem~\ref{TheoremBimonUniv(Co)actDuality}.
\end{proof}	

\begin{corollary}\label{CorollaryHopfMonUniv(Co)actDualityClosed}
	Let~$\mathcal C$ be a symmetric closed monoidal category
	satisfying Properties~\ref{PropertySmallLimits},
	\ref{PropertySubObjectsSmallSet}--\ref{PropertyLimitsOfSubobjectsArePreserved}, \ref{PropertySwitchProdTensorIsAMonomorphism}, \ref{PropertyEqualizers}
	and the properties dual to Properties~\ref{PropertySmallLimits}, \ref{PropertySubObjectsSmallSet} and~\ref{PropertyFreeMonoid}
	of Section~\ref{SubsectionSupportCoactingConditions}
	such that
	\begin{itemize}
		\item the functor $(-)^*$ maps extremal monomorphisms to extremal epimorphisms;
		\item $\alpha_M$ is an extremal monomorphism and $\theta_{M,N}$ is a monomorphism
		for all objects $M,N$;
		\item $\left(\theta^\mathrm{inv}_{B,B}\right)^\flat \colon B\otimes B \to (B^* \otimes B^*)^*$  is an extremal monomorphism in $\mathcal C$
		for all bimonoids~$B$;
		\item $\varkappa_H$ is a monomorphism in $\mathcal C$ for every Hopf monoid $H$
	\end{itemize}
	and let $i\colon V \to [A,A]$ be a submonoid such that $i=K(\rho_U^\vee)$ for some morphism $\rho_U \colon A \to A \otimes U$ where $U$ is an object in $\mathcal C$.
	Then $$\left(\rho_{A,V}^\mathbf{Hopf}\right)^\vee(\varkappa_{\mathcal{H}^\square(A,V)}\otimes \id_A)
	\colon \mathcal{H}^\square(A,V)^\circ \otimes A \to A$$ is an action and
	the unique Hopf monoid homomorphism $\beta^{\mathbf{Hopf}} \colon \mathcal{H}^\square(A,V)^\circ \to {}_\square \mathcal{H}(A,V)$ making
	the diagram~\eqref{EqHopfMonUniv(Co)actDualityClosed}
	commutative is a Hopf monoid isomorphism.		
\end{corollary}
\begin{proof}
	Apply~\cite[Theorems 4.6 and 5.10]{AGV3}, 
	 Theorems~\ref{TheoremDUALComonUnivMeasExistenceClosed}, \ref{TheoremMonUnivComeasExistenceClosed}, \ref{TheoremHopfMonUniv(Co)actDualityClosed}
	and Remarks~\ref{RemarkCircSufficientConditionsForExistence} and~\ref{RemarkInBraidedClosedMonoidalManyPropertiesHoldAutomatically}. 
\end{proof}	

\begin{corollary}[{\cite[Theorem 4.15]{AGV2}}]
	Let $A$ be an $\Omega$-algebra  over a field $\mathbbm{k}$ and let $V\subseteq \End_\mathbbm{k}(A)$ be a unital pointwise finite dimensional subalgebra  closed in the finite topology.
		Then $\beta^{\mathbf{Hopf}}$ from the diagram~\eqref{EqHopfMonUniv(Co)actDualityClosed}
	is a Hopf algebra isomorphism.
\end{corollary}

\section{Applications}\label{SectionApplications}

In this section we provide examples of monoidal categories satisfying Properties~\ref{PropertySmallLimits}--\ref{PropertyFreeMonoid}, \ref{PropertyExtrMonomorphism} 
of Section~\ref{SubsectionSupportCoactingConditions} 
and their duals as well as some other properties required in Sections~\ref{SectionDuality} and~\ref{SectionCosupportDualityInMonoidalClosedCategories}.

Throughout the section we use the standard Sweedler notation:

$\Delta c = c_{(1)}\otimes c_{(2)}$, $c\in C$,
for the comultiplication $\Delta \colon C \to C \otimes C$ in a coalgebra $C$ over a field $\mathbbm k$;

$\delta(m) = m_{(-1)}\otimes m_{(0)}$, $m\in M$, for the linear map 
$\delta \colon M \to C \otimes M$ defining on a $\mathbbm k$-vector space $M$ a structure of a left $C$-comodule;

$\rho(m) = m_{(0)}\otimes m_{(1)}$, $m\in M$, for the linear map 
$\rho \colon M \to M \otimes C$ defining on a $\mathbbm k$-vector space $M$ a structure of a right $C$-comodule.

Recall that by $V^*$ we denote the dual to an object $V$ in a pre-rigid braided monoidal category $\mathcal C$
and by  $(-)^\circ \colon \mathsf{Mon}(\mathcal C) \to \mathsf{Comon}(\mathcal C)$ the functor adjoint to the induced functor $(-)^* \colon \mathsf{Comon}(\mathcal C) \to \mathsf{Mon}(\mathcal C)$.
Below we will have to use the traditional duals (i.e. where $\mathcal C$ is replaced with $\mathbf{Vect}_\mathbbm k$
for a field $\mathbbm{k}$) too.
In order to distinguish between them, we will use for the latter a special notation. Namely,
for a vector space $V$ over $\mathbbm{k}$ we denote
its dual vector space $\Hom_\mathbbm{k}(V,\mathbbm{k})$ by $V^*_\mathbbm{k}$ and for a $\mathbbm{k}$-algebra $A$ 
we denote its traditional finite (or Sweedler) dual by $A^\circ_\mathbbm{k}$.

\subsection{Modules over Hopf algebras}

Let $B$ be a bialgebra over a field $\mathbbm k$. The category ${}_B \mathsf{Mod}$ of left $B$-modules
is monoidal where the monoidal product coincides with the tensor product $\otimes$ over $\mathbbm k$.
It turns out that all the properties from Section~\ref{SubsectionSupportCoactingConditions} as well as their duals
hold in ${}_B \mathsf{Mod}$:

\begin{lemma}\label{LemmaBialgebraMod1-10Star}
The category ${}_B \mathsf{Mod}$
satisfies Properties~\ref{PropertySmallLimits}--\ref{PropertyFreeMonoid}, \ref{PropertyExtrMonomorphism}
of Section~\ref{SubsectionSupportCoactingConditions} as well as their duals.
In addition, all monomorphisms and epimorphisms in ${}_B \mathsf{Mod}$ are extremal. 
\end{lemma}
\begin{proof}
The forgetful functor ${}_B \mathsf{Mod} \to \mathbf{Vect}_{\mathbbm k}$ is a strict monoidal
functor that creates small limits and colimits as well as limits of subobjects and colimits of quotient objects. Moreover, the free (tensor) algebra $T(M)$ of a $B$-module $M$ inherits the structure of a $B$-module, which makes $T(M)$ a $B$-module algebra. Finally, by~\cite[Proposition 4.1]{AbdulIovanov} the forgetful functor $\mathsf{Comon}({}_B \mathsf{Mod}) \to {}_B \mathsf{Mod} $
admits a right adjoint.
\end{proof}	

Let now $H$ be a Hopf algebra over a field $\mathbbm k$. Then the category ${}_H \mathsf{Mod}$
is closed where for every $H$-modules $M$ and $N$ we have $[M,N]=\Hom_\mathbbm{k}(M,N)$,
\begin{equation}\label{EqHActionOn[M,N]}
(hf)(m):=h_{(1)} f(Sh_{(2)} m)\text{ for  }h\in H,\ f\in \Hom_\mathbbm{k}(M,N),\ m\in M.
\end{equation}
In particular, ${}_H \mathsf{Mod}$ is pre-rigid where for every $H$-module $M$
we have $M^* = M^*_\mathbbm{k}$ and $(hf)(m):= f((Sh) m)$ for all  $h\in H$, $f\in M^*$, $m\in M$.

\subsection{Comodules over Hopf algebras}\label{SubsectionApplicationsComodules}
Now consider the category $\mathsf{Comod}^B$ of right $B$-comodules for a bialgebra $B$ over a field $\mathbbm k$ where the monoidal product again coincides with the tensor product $\otimes$ over $\mathbbm k$.
It turns out that all properties from Section~\ref{SubsectionSupportCoactingConditions} as well as their duals hold in $\mathsf{Comod}^B$ too:
\begin{lemma}\label{LemmaBialgebraComod1-10Star}
	Let $B$ be a bialgebra over a field $\mathbbm k$. Then the category $\mathsf{Comod}^B$
	satisfies Properties~\ref{PropertySmallLimits}--\ref{PropertyFreeMonoid}, \ref{PropertyExtrMonomorphism}
	of Section~\ref{SubsectionSupportCoactingConditions} and their duals.
In addition, all monomorphisms and epimorphisms in $\mathsf{Comod}^B$ are extremal. 
\end{lemma}
\begin{proof}
	The forgetful functor $\mathsf{Comod}^B \to \mathbf{Vect}_{\mathbbm k}$ creates finite limits and small colimits as well as limits of subobjects and colimits of quotient objects.
	If $M_\alpha$, $\alpha \in \Lambda$,
	are right $B$-comodules, then their product in $\mathsf{Comod}^B$
	is the subspace of their Cartesian product $\prod\limits_{\alpha \in \Lambda} M_\alpha$
	consisting of all tuples $(m_\alpha)_{\alpha \in \Lambda}$, $m_\alpha \in M_\alpha$,
	such that for each tuple there exists a single finite dimensional subcoalgebra $C\subseteq B$
	where $\rho(m_\alpha) \in M_\alpha \otimes C$ for all $\alpha \in \Lambda$. Again, the free algebra $T(M)$ of a $B$-comodule $M$ inherits the structure of a $B$-comodule, which makes $T(M)$ a $B$-comodule algebra.
	Finally, by~\cite[Proposition 4.1]{AbdulIovanov} the forgetful functor $\mathsf{Comon}(\mathsf{Comod}^B) \to \mathsf{Comod}^B$
	admits a right adjoint.
\end{proof}	

Let now $H$ be a Hopf algebra over a field $\mathbbm k$. Then the category $\mathsf{Comod}^H$
is closed 
 where for every $H$-comodules $M$ and $N$ the $H$-comodule  $[M,N]$ is the subspace of $\Hom_\mathbbm{k}(M,N)$
consisting of such $\mathbbm{k}$-linear maps $f \colon M \to N$ that 
$f(m_{(0)})_{(0)}\otimes f(m_{(0)})_{(1)}Sm_{(1)} \in M \otimes C$ where $C \subseteq H$ is a fixed finite dimensional subcoalgebra that may depend on $f$ but does not depend on $m\in M$.
Then $\rho \colon [M,N] \to [M,N] \otimes H$ is defined by $\rho(f):= f_{(0)}\otimes f_{(1)}$,
$f_{(0)} (m) \otimes f_{(1)} := f(m_{(0)})_{(0)}\otimes f(m_{(0)})_{(1)}Sm_{(1)}$ for $m\in M$ and $f\in [M,N]$.
In particular, $\mathsf{Comod}^H$ is pre-rigid where for every $H$-comodule $M$
the dual comodule $M^*$ is the subspace of $M^*_{\mathbbm k}$
consisting of such $\mathbbm{k}$-linear functions $f \colon M \to \mathbbm{k}$ that 
$f(m_{(0)}) Sm_{(1)} \in C$ where $C \subseteq H$ is a fixed finite dimensional subcoalgebra that may depend on $f$ but does not depend on $m\in M$. Then
 $f_{(0)}(m) f_{(1)}:= f(m_{(0)}) Sm_{(1)}$ for $f\in M^*$, $m\in M$.
 
\subsection{Left Yetter~-- Drinfel'd modules}

Let $H$ be a Hopf algebra over a field $\mathbbm k$ with an invertible antipode $S$. Denote by ${}_H^H\mathcal{YD}$
the category of \textit{left Yetter~--- Drinfel'd modules} (or \textit{${}_H^H\mathcal{YD}$-modules} for short), i.e. left $H$-modules and $H$-comodules $M$ such that
the $H$-action and the $H$-coaction $\delta \colon M \to H \otimes M$ 
satisfy the following compatibility condition:
$$\delta(hm)=h_{(1)}m_{(-1)}Sh_{(3)} \otimes h_{(2)}m_{(0)} \text{ for every } m\in M \text{ and } h\in H.$$
For more details on the general theory of Yetter~-- Drinfel'd modules and their applications we refer to e.g.~\cite[Section 11.6]{Radford}, \cite[Section 4.4]{CMZ}. 

An \textit{${}_H^H\mathcal{YD}$-module algebra} is an algebra $A$ over a field $\mathbbm k$
that is an ${}_H^H\mathcal{YD}$-module, an $H$-module algebra and an $H$-comodule algebra at the same time.
The algebra $A$
is called a \textit{unital ${}_H^H\mathcal{YD}$-module algebra} if $A$ is unital, $\delta(1_A) = 1_H \otimes 1_A$ and $h 1_A = \varepsilon(h) 1_A$. 

The category ${}_H^H\mathcal{YD}$ is braided monoidal where the monoidal product of  ${}_H^H\mathcal{YD}$-modules
$M$ and $N$ is their usual tensor product $M\otimes N$ over $\mathbbm k$ with the induced structures of a left $H$-module and a left $H$-comodule. The braiding $c_{M,N} \colon M\otimes N \to N \otimes M$ is defined
by the formula $$c_{M,N}(m\otimes n) := m_{(-1)}n \otimes m_{(0)}\text{ for }m\in M,\ n\in N.$$
Its inverse $c_{M,N}^{-1} \colon  N \otimes M \to M\otimes N$
is defined by $$c_{M,N}^{-1}(n\otimes m) := m_{(0)} \otimes \left(S^{-1}m_{(-1)}\right)n\text{ for }m\in M,\ n\in N.$$
Monoids in ${}_H^H\mathcal{YD}$ are just unital associative ${}_H^H\mathcal{YD}$-module algebras.
Comonoids in ${}_H^H\mathcal{YD}$ are called \textit{${}_H^H\mathcal{YD}$-module coalgebras}.

\begin{lemma}\label{LemmaYD1-10Star}
	Let $H$ be a Hopf algebra over a field $\mathbbm k$ with an invertible antipode $S$. Then the category ${}_H^H\mathcal{YD}$
	satisfies Properties~\ref{PropertySmallLimits}--\ref{PropertyFreeMonoid}, \ref{PropertyExtrMonomorphism}
	of Section~\ref{SubsectionSupportCoactingConditions} and their duals.
In addition, all monomorphisms and epimorphisms in ${}_H^H\mathcal{YD}$ are extremal. 
\end{lemma}
\begin{proof}
	The forgetful functor ${}_H^H\mathcal{YD} \to {}^H\mathsf{Comod}$ creates small limits and colimits as well as limits of subobjects and colimits of quotient objects. Both categories are abelian. In addition, monomorphisms and epimorphisms are just injective and surjective homomorphisms, respectively. Hence, by Lemma~\ref{LemmaBialgebraComod1-10Star},
	the category ${}_H^H\mathcal{YD}$
	satisfies Properties~\ref{PropertySmallLimits}--\ref{PropertyEqualizers}, \ref{PropertyExtrMonomorphism} and their duals.
	Moreover, the free algebra $T(M)$ of an ${}_H^H\mathcal{YD}$-module $M$ inherits the structure of an ${}_H^H\mathcal{YD}$-module, which makes $T(M)$ an ${}_H^H\mathcal{YD}$-module algebra, whence we get Property~\ref{PropertyFreeMonoid}.	Now we have to prove its dual.
	
	By~\cite[Proposition 2.3]{AbdulIovanov},
	the category $\mathsf{Comon}\left({}_H^H\mathcal{YD}\right)$ is cowellpowered. By~\cite[Theorem 4.5]{AbdulIovanov} every
	element of an $H$-comodule coalgebra is contained in a finite dimensional $H$-comodule subcoalgebra, whence
	every element of an ${}_H^H\mathcal{YD}$-module coalgebra is contained in a ${}_H^H\mathcal{YD}$-module
	subcoalgebra finitely generated as an $H$-module. In particular, the small set of all ${}_H^H\mathcal{YD}$-module
	coalgebras finitely generated as $H$-modules generates the category $\mathsf{Comon}\left({}_H^H\mathcal{YD}\right)$.
	
	The forgetful functor $\mathsf{Comon}\left({}_H^H\mathcal{YD}\right) \to {}_H^H\mathcal{YD}$ preserves small coproducts.
	Let $f,g \colon A \to B$ be ${}_H^H\mathcal{YD}$-module coalgebra homomorphisms.
	Then $$\Delta(f(a)-g(a))=\bigl(f(a_{(1)})-g(a_{(1)})\bigr) \otimes f(a_{(2)}) + 
	g(a_{(1)}) \otimes \bigl(f(a_{(2)})-g(a_{(2)})\bigr) 
	\text{ for every } a\in A$$
	implies that the set $I=\lbrace f(a)-g(a) \mid a \in A \rbrace$ is a coideal in $B$ that is, in addition, an ${}_H^H\mathcal{YD}$-subcomodule. Hence the surjective homomorphism $\pi \colon B \to B/I$ is the coequilizer
	of $f$ and $g$ both in $\mathsf{Comon}\left({}_H^H\mathcal{YD}\right)$ and ${}_H^H\mathcal{YD}$.
	Therefore, the forgetful functor $\mathsf{Comon}\left({}_H^H\mathcal{YD}\right) \to {}_H^H\mathcal{YD}$ preserves all small colimits.
	Now we apply the Special Adjoint Functor Theorem.
\end{proof}	

The category ${}_H^H\mathcal{YD}$ is closed where for every ${}_H^H\mathcal{YD}$-modules $M$ and $N$ the ${}_H^H\mathcal{YD}$-module  $[M,N]$ is the subspace of $\Hom_\mathbbm{k}(M,N)$
consisting of such $\mathbbm{k}$-linear maps $f \colon M \to N$ that 
$f(m_{(0)})_{(-1)}S^{-1}m_{(-1)} \otimes f(m_{(0)})_{(0)} \in C \otimes M$ where $C \subseteq H$ is a fixed finite dimensional subcoalgebra that may depend on $f$ but does not depend on $m\in M$.
Then $\delta \colon [M,N] \to H \otimes [M,N] $ is defined by $\delta(f):= f_{(-1)}\otimes f_{(0)}$,
$$f_{(-1)}  \otimes f_{(0)} (m) := f(m_{(0)})_{(-1)}S^{-1}m_{(-1)} \otimes f(m_{(0)})_{(0)}$$
for $m\in M$ and $f\in [M,N]$
and the $H$-action $H \otimes [M,N] \to [M,N]$ is defined by~\eqref{EqHActionOn[M,N]}.
In particular, ${}_H^H\mathcal{YD}$ is pre-rigid where for every ${}_H^H\mathcal{YD}$-module $M$
the ${}_H^H\mathcal{YD}$-module $M^*$ is the subspace of $M^*_{\mathbbm k}$
consisting of such $\mathbbm{k}$-linear functions $f \colon M \to \mathbbm{k}$ that 
$f(m_{(0)}) S^{-1}m_{(-1)} \in C$ where $C \subseteq H$ is a fixed finite dimensional subcoalgebra that may depend on $f$ but does not depend on $m\in M$.
Then $\delta \colon M^* \to H \otimes M^* $ is defined by $\delta(f):= f_{(-1)}\otimes f_{(0)}$, $f_{(0)}(m) f_{(-1)}:= f(m_{(0)}) S^{-1}m_{(-1)}$
and $(hf)(m)=f((Sh)m)$  for $f\in M^*$, $h \in H$ and $m\in M$. The functor $(-)^*$ on ${}_H^H\mathcal{YD}$-module homomorphisms
coincides with the usual $(-)^*$ on $\mathbbm{k}$-linear maps.

\begin{remark}\label{RemarkYDfBelongsToStar} For a ${}_H^H\mathcal{YD}$-module $M$
		define the bilinear map $\diamond \colon M^*_{\mathbbm k} \times M \to H$
	by $f \diamond m := f(m_{(0)}) m_{(-1)}$ for all $f\in M^*_{\mathbbm k}$ and $m\in M$.
	Now for $f \in M^*_{\mathbbm k}$ the condition $f \in M^*$ is equivalent to $\dim (f\diamond M) < +\infty$.
\end{remark}

\begin{lemma}\label{LemmaYDSf(-1)Multiply}
	Let $M$ be an ${}_H^H\mathcal{YD}$-module.  Define the linear maps $\xi \colon M \to M$
	and $\zeta \colon M \to M$ by
	$$\xi(m)=\left(S^{-2}m_{(-1)}\right) m_{(0)} \text{ and }\zeta(m)=\left(Sm_{(-1)}\right) m_{(0)}
	\text{ for }m \in M.$$ Then $$\xi\zeta=\zeta\xi=\id_M.$$
\end{lemma}
\begin{proof} The definition of an ${}_H^H\mathcal{YD}$-module implies that
	$$\xi(hm) = (S^{-2}h) \xi(m)\text{ and }\zeta(hm) = (S^2 h) \zeta(m)$$ for every $h\in H$ and $m\in M$.
	Hence $\xi\zeta=\zeta\xi=\id_M$.
\end{proof}	

Note that for any ${}_H^H\mathcal{YD}$-modules $M$ and $N$ we have
$$\theta_{M,N}(f\otimes g)(m\otimes n) = f(m_{(0)})g(m_{(-1)}n),$$
$$\theta^\mathrm{inv}_{M,N}(f\otimes g)(m\otimes n) = f\bigl((S^{-1}n_{(-1)})m\bigr)g(n_{(0)}),$$
$$\varphi^\flat(n)(m)=\varphi(m_{(0)})\bigl((S^{-1}m_{(-1)})n\bigr),$$
$$\varphi^\sharp(n)(m)=\varphi(n_{(-1)}m)(n_{(0)})$$
for all $m\in N$, $n\in M$, $f\in M^*$, $g\in N^*$, $\varphi \in {}_H^H\mathcal{YD}(M, N^*)$.

\begin{lemma}\label{LemmaYDThetaInjective}
	For every ${}_H^H\mathcal{YD}$-modules $M$ and $N$ the map $\theta_{M,N} \colon M^*\otimes N^* \to (M\otimes N)^*$ is injective.
\end{lemma}	
\begin{proof}
The map $\theta_{M,N} \colon M^*\otimes N^* \to (M\otimes N)^*$ is injective since the map $m\otimes n\mapsto m_{(0)} \otimes m_{(-1)}n$ is a bijection.
\end{proof}
\begin{lemma}\label{LemmaYDAlphaInjective}
	Let $M$ be an ${}_H^H\mathcal{YD}$-module. Then the map $\alpha_M \colon M \to M^{**}$ is injective
	if and only if $\bigcap\limits_{f \in M^*} \Ker f = 0$.
\end{lemma}	
\begin{proof}
	For every $m\in M$ and $f\in M^*$ we have
	$$\alpha_M(m)(f) := (\id_{M^*})^\flat (m)(f) = f_{(0)}\bigl((S^{-1}f_{(-1)})m\bigr)
	= \bigl((S^{-2}f_{(-1)})f_{(0)}\bigr)(m).$$
	Now Lemma~\ref{LemmaYDSf(-1)Multiply} implies that $\alpha_M(m) = 0$ for some $m\in M$
	if and only if $f(m)=0$ for all $f\in M^*$.
\end{proof}

\begin{lemma}\label{LemmaYDThetaInvFlatInjective}
	Let $M$ and $N$ be nonzero ${}_H^H\mathcal{YD}$-modules.
	Then the map $$\left(\theta^\mathrm{inv}_{M,N}\right)^\flat
	\colon M\otimes N \to (M^*\otimes N^*)^*$$ is injective if and only if 
	$\bigcap\limits_{f \in M^*} \Ker f = 0$  and $\bigcap\limits_{g \in N^*} \Ker g = 0$.
\end{lemma}	
\begin{proof}
Given $m\in M$, $n\in N$, $f\in M^*$, $g\in N^*$,
we have 
\begin{equation*}\begin{split}\left(\theta^\mathrm{inv}_{M,N}\right)^\flat (m\otimes n)(f\otimes g)=
\theta^\mathrm{inv}_{M,N}\bigl((f\otimes g)_{(0)}\bigr)\Bigl(\left(S^{-1}(f\otimes g)_{(-1)}\right)(m\otimes n)\Bigr)
\\=\theta^\mathrm{inv}_{M,N}\Bigl(\left(S^{-2}(f\otimes g)_{(-1)}\right)(f\otimes g)_{(0)}\Bigr)(m\otimes n)
\\= \Bigl(\left(S^{-2}(f\otimes g)_{(-1)}\right)(f\otimes g)_{(0)}\Bigr)\bigl((S^{-1}n_{(-1)})m \otimes n_{(0)}\bigr).\end{split}\end{equation*}
Here we have used the fact that for every $H$-module homomorphism $\varphi \colon A\otimes B \to \mathbbm{k}$
we have $$\varphi(ha \otimes b) = \varphi(h_{(1)}a \otimes h_{(2)} (Sh_{(3)}) b)=
\varepsilon(h_{(1)}) \varphi\bigl(a \otimes (Sh_{(2)}) b) = \varphi(a \otimes (Sh)b)$$ for all $a\in A$, $b\in B$.	

Now we notice that the map $m \otimes n \mapsto (S^{-1}n_{(-1)})m \otimes n_{(0)}$ is bijective
since $m \otimes n \mapsto n_{(-1)}m \otimes n_{(0)}$ is its inverse.
Combining this with Lemma~\ref{LemmaYDSf(-1)Multiply}, we obtain that the map $\left(\theta^\mathrm{inv}_{M,N}\right)^\flat
\colon M\otimes N \to (M^*\otimes N^*)^*$ is not injective only if
there exists a nonzero element $\sum\limits_{i=1}^k m_i \otimes n_i \in M \otimes N$ such that
$$\sum\limits_{i=1}^k f(m_i) g(n_i)=0\text{ for all }f\in M^*,\ g\in N^*.$$
Without loss of generality we may assume that the $m_i$'s are linearly independent.
Suppose $\bigcap\limits_{f \in M^*} \Ker f = 0$  and $\bigcap\limits_{g \in N^*} \Ker g = 0$.
 Then 
$$\sum\limits_{i=1}^k f(m_i) g(n_i) =  f\left( \sum\limits_{i=1}^k g(n_i) m_i\right)=0$$
implies $\sum\limits_{i=1}^k g(n_i) m_i = 0$ and $g(n_i)=0$ for all $1\leqslant i \leqslant k$ and $g\in N^*$.
Hence all $n_i = 0$, $\sum\limits_{i=1}^k m_i \otimes n_i = 0$
and  $\left(\theta^\mathrm{inv}_{M,N}\right)^\flat$ is injective.

Conversely, if $m \in M$, $n\in N$, where $m\ne 0$, $n\ne 0$, and either $m \in \bigcap\limits_{f \in M^*} \Ker f$
or $n\in \bigcap\limits_{g \in N^*} \Ker g$,
then $\left(\theta^\mathrm{inv}_{M,N}\right)^\flat( n_{(-1)}m \otimes n_{(0)})(f\otimes g)= 0$
for all $f\in M^*,\ g\in N^*$.
\end{proof}
\begin{remark}
	An example of ${}_H^H\mathcal{YD}$-modules $M$ and $N$ with non-injective $\alpha_M$ and $\left(\theta^\mathrm{inv}_{M,N}\right)^\flat$ will be given in Remark~\ref{left_YD_rem} (\ref{RemarkYDThetaInvFlatIsNotMono}) below.
\end{remark}

By Lemma~\ref{LemmaYD1-10Star} and Remark~\ref{RemarkCircSufficientConditionsForExistence}, the functor $(-)^* \colon \mathsf{Comon}\left({}_H^H\mathcal{YD}\right) \to \mathsf{Mon}\left({}_H^H\mathcal{YD}\right)$ admits an adjoint functor $(-)^\circ \colon \mathsf{Mon}\left({}_H^H\mathcal{YD}\right) \to \mathsf{Comon}\left({}_H^H\mathcal{YD}\right)$. However, in order to prove that $\varkappa_A$ is a monomorphism for every ${}_H^H\mathcal{YD}$-module algebra $A$, we have to provide an explicit construction.
We first need the following lemma:

\begin{lemma}\label{LemmaPhiMonACStarEquivalence} Let $(A,\mu,u)$ be a monoid and let $(C,\Delta,\varepsilon)$ be a comonoid in a braided pre-rigid monoidal category $\mathcal C$. Then
	for $\varphi \in \mathcal C(A, C^*)$ we have $\varphi \in \mathsf{Mon}(\mathcal C)(A, C^*)$
	if and only if the diagrams below are commutative:
	\begin{equation}\label{EqMonACStarEquivalence} \xymatrix{ C \ar[d]_{\Delta} \ar[rrrr]^{\varphi^\sharp} & & & & A^*\ar[d]^{\mu^*} \\
		C\otimes C \ar[rr]^{\varphi^\sharp\otimes \varphi^\sharp} & & A^*\otimes A^* \ar[rr]^{\theta^\mathrm{inv}_{A,A}} & & (A\otimes A)^*}
	\qquad \xymatrix{ C \ar[d]_\varepsilon \ar[r]^{\varphi^\sharp} & A^* \ar[d]^{u^*} \\
	                  \mathbbm{1} \ar[r]^\iota &  \mathbbm{1}^*}
	\end{equation}
\end{lemma}	
\begin{proof} Apply Remark~\ref{RemarkThetaThetaInvAdjoint} and~\cite[Lemma 2.5]{GoyVer}.
\end{proof}	

Below we present a candidate for $A^\circ$ in  ${}_H^H\mathcal{YD}$.

\begin{lemma}\label{LemmaYDAcirc} Let $H$ be a Hopf algebra over a field $\mathbbm k$ with an invertible antipode $S$.
	Given a unital associative ${}_H^H\mathcal{YD}$-module algebra $A$, the subspace $A^* \cap A^\circ_\mathbbm{k}$
	is a ${}_H^H\mathcal{YD}$-submodule of~$A^*$. Moreover, $A^* \cap A^\circ_\mathbbm{k}$
	is an \textrm{${}_H^H\mathcal{YD}$-module coalgebra}, i.e. a comonoid in ${}_H^H\mathcal{YD}$,
	where the comultiplication $\Delta_{A^* \cap A^\circ_\mathbbm{k}}f:=f_{[1]}\otimes f_{[2]}$ and the counit are defined by
	$$f_{[1]}(a) f_{[2]}(b) := f\bigl((b_{(-1)}a)b_{(0)}\bigr), \quad \varepsilon_{A^* \cap A^\circ_\mathbbm{k}} (f) := f(1_A),$$
	for all  $f\in A^* \cap A^\circ_\mathbbm{k}$, $a$, $b \in A$.
\end{lemma}	
\begin{proof} 	Note that for every $f \in A^*$ and $b \in A$ we have $\dim \bigl(f(b(-))\diamond A\bigr) < +\infty$
	and $\dim \bigl(f((-)b)\diamond A\bigr) < +\infty$ (recall that $\diamond$ was defined in Remark~\ref{RemarkYDfBelongsToStar} above),
	since $$f\bigl(b(-)\bigr)\diamond a = f(ba_{(0)}) a_{(-1)}= f(b_{(0)} a_{(0)}) (S b_{(-2)}) b_{(-1)}a_{(-1)}=
	(S b_{(-1)})(f \diamond (b_{(0)} a))$$
	and $$f\bigl((-)b\bigr)\diamond a = f(a_{(0)}b) a_{(-1)}= f(a_{(0)}b_{(0)}) a_{(-1)}b_{(-1)}S^{-1}b_{(-2)}=
	(f \diamond (ab_{(0)}))S^{-1}b_{(-1)}.$$
	Suppose $f \in A^* \cap A^\circ_\mathbbm{k}$. Then there exist $s \in\mathbb Z_+$ and $f_i, g_i \in A^\circ_\mathbbm{k}$, $1\leqslant i \leqslant s$, such that $f(ab)=\sum\limits_{i=1}^s f_i(a) g_i(b)$ for all $a,b \in A$. Without loss of generality we may assume that each of the systems $f_1, \ldots, f_s$ and $g_1, \ldots, g_s$ of linear functions is linearly independent. Performing Gaussian elimination, one can find $a_i,b_i \in A$, $1\leqslant i \leqslant s$, such that $f_i(a_j)=g_i(b_j)=\delta_{ij}$ for all  $1\leqslant i,j \leqslant s$.
	In particular, $f_i=f((-)b_i)$ and $g_i=f(a_i(-))$ for all $1\leqslant i \leqslant s$.
	Hence $f \in A^* \cap A^\circ_\mathbbm{k}$ implies that all $f_i, g_i \in A^* \cap A^\circ_\mathbbm{k}$.
	
	Moreover, for every $h\in H$ and $a,b\in A$ we have
	\begin{equation*}\begin{split}(hf)(ab)=f((Sh)(ab))=f\bigl(((Sh_{(2)})a)((Sh_{(1)})b)\bigr) \\ =\sum\limits_{i=1}^s f_i((Sh_{(2)})a) g_i((Sh_{(1)})b) 
			= \sum\limits_{i=1}^s (h_{(2)}f_i)(a) (h_{(1)}g_i)(b),\end{split}\end{equation*}
	\begin{equation*}\begin{split}f_{(0)}(ab)f_{(-1)}=f((ab)_{(0)})S^{-1}\bigl((ab)_{(-1)}\bigr)=f(a_{(0)}b_{(0)})S^{-1}(b_{(-1)}) S^{-1}(a_{(-1)})\\=
			\sum\limits_{i=1}^s f_i(a_{(0)})g_i(b_{(0)})S^{-1}(b_{(-1)}) S^{-1}(a_{(-1)}) =
			\sum\limits_{i=1}^s f_{i(0)}(a)g_{i(0)}(b)g_{i(-1)} f_{i(-1)}.\end{split}\end{equation*}
	Therefore, by~\cite[Proposition 1.5.6, 3)]{DNR}, the subspace $A^* \cap A^\circ_\mathbbm{k}$ is an ${}_H^H\mathcal{YD}$-submodule of $A^*$.

	Now \begin{equation*}\begin{split}f_{[1]}(a) f_{[2]}(b) := f\bigl((b_{(-1)}a)b_{(0)}\bigr)=\sum_{i=1}^s f_i(b_{(-1)}a) g_i(b_{(0)})
	\\ = \sum_{i=1}^s  \bigl((S^{-1}b_{(-1)})f_i\bigr)(a) g_i(b_{(0)})  =
	 \sum_{i=1}^s  \bigl(g_{i(-1)}f_i\bigr)(a) g_{i(0)}(b) \text{ for all }a,b\in A.\end{split}\end{equation*}
	Therefore, $f_{[1]} \otimes f_{[2]} = \sum\limits_{i=1}^s g_{i(-1)}f_i \otimes g_{i(0)} \in (A^* \cap A^\circ_\mathbbm{k}) \otimes (A^* \cap A^\circ_\mathbbm{k})$
	for all $f \in A^* \cap A^\circ_\mathbbm{k}$ and $\Delta_{A^* \cap A^\circ_\mathbbm{k}}$ is a comultiplication on $A^* \cap A^\circ_\mathbbm{k}$.
		
	An explicit verification shows that $(A^* \cap A^\circ_\mathbbm{k}, \Delta_{A^* \cap A^\circ_\mathbbm{k}}, \varepsilon_{A^* \cap A^\circ_\mathbbm{k}})$ is indeed a coalgebra.
	
	Note that	
	\begin{equation*}\begin{split} (hf)_{[1]}(a) (hf)_{[2]}(b)= (hf)\bigl((b_{(-1)}a)b_{(0)}\bigr)  
	= f\bigl((Sh)((b_{(-1)}a)b_{(0)})\bigr) \\=f\bigl(((Sh_{(2)})b_{(-1)}a)((Sh_{(1)})b_{(0)})\bigr)
		 =f\bigl((((Sh_{(4)})b_{(-1)} (S^2 h_{(2)}) Sh_{(1)})a) (Sh_{(3)})b_{(0)}\bigr)
	\\=f\bigl((((Sh_{(2)})_{(1)}b_{(-1)} (S(Sh_{(2)})_{(3)}) Sh_{(1)})a) (Sh_{(2)})_{(2)}b_{(0)}\bigr) \\=
	f\bigl(((((Sh_{(2)})b)_{(-1)} Sh_{(1)})a) ((Sh_{(2)})b)_{(0)}\bigr)
	\\= f_{[1]}\bigl((Sh_{(1)})a\bigr) f_{[2]}\bigl((Sh_{(2)})b\bigr) =	(h_{(1)}f_{[1]})(a) (h_{(2)}f_{[2]})(b)
\end{split}\end{equation*} 
and $$\varepsilon_{A^* \cap A^\circ_\mathbbm{k}}(hf)=(hf)(1_A)=f((Sh) 1_A)=\varepsilon(h)f(1_A)=
\varepsilon(h) \varepsilon_{A^* \cap A^\circ_\mathbbm{k}}(f)$$
	for all $h\in H$, $f\in A^* \cap A^\circ_\mathbbm{k}$, $a,b\in A$.
	Hence $A^* \cap A^\circ_\mathbbm{k}$ is an $H$-module coalgebra.
	
	At the same time,
	\begin{equation*}\begin{split}	f_{(0)[1]}(a) f_{(0)[2]}(b) f_{(-1)}
			= f_{(0)}\bigl((b_{(-1)}a)b_{(0)}\bigr) f_{(-1)}  = 
			f\bigl(((b_{(-1)}a)b_{(0)})_{(0)}\bigr) S^{-1}((b_{(-1)}a)b_{(0)})_{(-1)} \\ =
			f\bigl((b_{(-2)}a)_{(0)}b_{(0)}\bigr) S^{-1} ((b_{(-2)}a)_{(-1)}b_{(-1)})  =
			f\bigl((b_{(-3)}a_{(0)})b_{(0)}\bigr) S^{-1} (b_{(-4)}a_{(-1)} (Sb_{(-2)}) b_{(-1)}) 
			\\ =
			f\bigl((b_{(-1)}a_{(0)})b_{(0)}\bigr) S^{-1} (b_{(-2)}a_{(-1)})
			 = f\bigl((b_{(-1)} a_{(0)}) b_{(0)}\bigr)(S^{-1}a_{(-1)})(S^{-1}b_{(-2)})
            \\ = f_{[1]}(a_{(0)}) f_{[2]}(b_{(0)}) (S^{-1}a_{(-1)})(S^{-1}b_{(-1)})
			 = f_{[1](0)}(a) f_{[2](0)}(b) f_{[1](-1)}f_{[2](-1)}
\end{split}\end{equation*} 
and $$\varepsilon_{A^* \cap A^\circ_\mathbbm{k}}(f_{(0)})f_{(-1)}=f_{(0)}(1_A)f_{(-1)}=f(1_A)1_H=\varepsilon_{A^* \cap A^\circ_\mathbbm{k}}(f)1_H$$
for all $f\in A^* \cap A^\circ_\mathbbm{k}$, $a,b\in A$.
Therefore, $A^* \cap A^\circ_\mathbbm{k}$ is an $H$-comodule coalgebra.
\end{proof}	

\begin{lemma}\label{LemmaYDFiniteDual} Let $H$ be a Hopf algebra over a field $\mathbbm k$ with an invertible antipode $S$. Then the finite dual $A^\circ$ in ${}_H^H\mathcal{YD}$ coincides with $A^* \cap A^\circ_\mathbbm{k}$ endowed with the coalgebra structure from Lemma~\ref{LemmaYDAcirc}.
	Moreover,
	\begin{itemize}
		\item $\varkappa_A$ is a monomorphism for every unital associative ${}_H^H\mathcal{YD}$-module algebra $A$.
	\end{itemize} 
\end{lemma}	
\begin{proof}
By Lemma~\ref{LemmaPhiMonACStarEquivalence}, for
every ${}_H^H\mathcal{YD}$-module algebra homomorphism $\varphi \colon A\to C^*$
where $A$ is a unital associative ${}_H^H\mathcal{YD}$-module algebra 
and $C$ is an ${}_H^H\mathcal{YD}$-module coalgebra, 
$\mu^*\varphi^\sharp(C)\in \theta^\mathrm{inv}_{A,A}(A^*\otimes A^*)$.

Recall that
\begin{equation*}\begin{split}\theta^\mathrm{inv}_{A,A}(f\otimes g)(a\otimes b)=f\bigl((S^{-1}b_{(-1)})a\bigr)g(b_{(0)})
 = \bigl((S^{-2}b_{(-1)})f\bigr)(a)g(b_{(0)})
\\ = \bigl((S^{-1}g_{(-1)})f\bigr)(a)g_{(0)}(b)= \bigl((S^{-1}g_{(-1)})f \otimes g_{(0)}\bigr)(a\otimes b)
\end{split}
\end{equation*}
 for all $f,g \in A^*$, $a,b \in A$.
 In other words, $\theta^\mathrm{inv}_{A,A}(A^*\otimes A^*)$ coincides with
 the image of $A^*\otimes A^*$ in $(A\otimes A)^*$
 under the restriction of the standard embedding $A^*_{\mathbbm k}\otimes A^*_{\mathbbm k} \hookrightarrow (A\otimes A)_{\mathbbm k}^*$
  in $\mathbf{Vect}_\mathbbm{k}$
 since the map $f\otimes g \mapsto (S^{-1}g_{(-1)})f \otimes g_{(0)}$
 is a bijection. Hence for every $c\in C$ the linear function $\varphi^\sharp (c)$
 belongs to $A^* \cap A^\circ_\mathbbm{k}$. An explicit verification shows that the diagram below
 is commutative:
 $$\xymatrix{
 	A^* \cap A^\circ_\mathbbm{k} \ar[d]_{\Delta_{A^* \cap A^\circ_\mathbbm{k}}}
 	\ar@{^{(}->}[rrr] & & & A^*\ar[d]^{\mu^*} \\
 	(A^* \cap A^\circ_\mathbbm{k})\otimes (A^* \cap A^\circ_\mathbbm{k}) \ar@{^{(}->}[r] &  A^*\otimes A^* \ar[rr]^{\theta^\mathrm{inv}_{A,A}} & & (A\otimes A)^* \\
 }$$

 The injectivity of $\theta^\mathrm{inv}_{A,A}$ together with Lemma~\ref{LemmaPhiMonACStarEquivalence} imply that an ${}_H^H\mathcal{YD}$-module homomorphism $\varphi \colon A\to C^*$ 
is an ${}_H^H\mathcal{YD}$-module algebra homomorphism if and only if the corestriction of $\varphi^\sharp$ to $A^* \cap A^\circ_\mathbbm{k}$ is an ${}_H^H\mathcal{YD}$-module coalgebra homomorphism. In particular, the functor $(-)^* \colon \mathsf{Comon}\bigl({}_H^H\mathcal{YD}\bigr)\to 
 \mathsf{Mon}\bigl({}_H^H\mathcal{YD}\bigr)$ admits an adjoint functor
 $(-)^\circ \colon \mathsf{Mon}\bigl({}_H^H\mathcal{YD}\bigr)\to 
 \mathsf{Comon}\bigl({}_H^H\mathcal{YD}\bigr)$ where $A^\circ = A^* \cap A^\circ_\mathbbm{k}$.
  Note that $\varkappa_A$
 coincides with the restriction of the standard embedding $A^\circ_\mathbbm{k} \hookrightarrow A^*_\mathbbm{k}$
 and is a monomorphism.
\end{proof}

Finally, we get
\begin{theorem}\label{TheoremYDSatisfiesProperties}
	Let $H$ be a Hopf algebra over a field $\mathbbm k$ with an invertible antipode $S$. 
	Then ${}_H^H\mathcal{YD}$ is a braided closed monoidal category satisfying  Properties~\ref{PropertySmallLimits}--\ref{PropertyFreeMonoid}, \ref{PropertyExtrMonomorphism}
	of Section~\ref{SubsectionSupportCoactingConditions} and their duals.
		Moreover, all monomorphisms and epimorphisms in ${}_H^H\mathcal{YD}$ are extremal and
	\begin{itemize}
		\item $\theta_{M,N}$ is a monomorphism
		for any ${}_H^H\mathcal{YD}$-modules $M,N$;
		\item $\alpha_M$ is a monomorphism for an ${}_H^H\mathcal{YD}$-module $M$
if and only if $\bigcap\limits_{f \in M^*} \Ker f = 0$;
		\item $\left(\theta^\mathrm{inv}_{M,N}\right)^\flat$ is a monomorphism
		for nonzero ${}_H^H\mathcal{YD}$-modules $M$ and $N$ if and only if $\bigcap\limits_{f \in M^*} \Ker f = 0$  and $\bigcap\limits_{g \in N^*} \Ker g = 0$;
		\item $\varkappa_A$ is a monomorphism for every unital associative ${}_H^H\mathcal{YD}$-module algebra $A$.
	\end{itemize}
\end{theorem}
\begin{proof}
	Apply Lemmas~\ref{LemmaYD1-10Star}, \ref{LemmaYDThetaInjective}--\ref{LemmaYDThetaInvFlatInjective} and \ref{LemmaYDFiniteDual}.
\end{proof}	

\begin{remarks}\label{left_YD_rem}
 \hspace{0.1cm}
 \begin{enumerate}
 	\item If the Hopf algebra $H$ is finite dimensional, then $M^*=M^*_{\mathbbm k}$ and $\bigcap\limits_{f \in M^*} \Ker f = 0$ for every ${}_H^H\mathcal{YD}$-module $M$.
\item\label{RemarkYDPolyMonoStarIsNotEpi} Note that the functor $(-)^*$ does not necessarily map monomorphisms to epimorphisms. 	Let $\mathbbm k$ be a field and let $H=\mathbbm k[x_i \mid i\in \mathbb N]$ where $\Delta_H x_i := x_i \otimes 1_H  + 1_H \otimes x_i$, $\varepsilon_H(x_i):=0$. Then $H$ is a left $H$-comodule where $\delta_H := \Delta_H$. 
	Consider $H$-subcomodules $\mathbbm k 1_H$ and $V:= \left\langle\left. x_i \right| i\in \mathbb N\right\rangle_{\mathbbm k}\oplus \mathbbm k 1_H$ of $H$.
	 Define on $V$ and $\mathbbm k 1_H$ the structure of trivial left $H$-modules, i.e. $hm := \varepsilon(h)m$ for all $m\in V$. The commutativity of $H$ implies that both $V$ and $\mathbbm k 1_H$ are ${}_H^H\mathcal{YD}$-modules.
	Denote by $\varphi$ the  ${}_H^H\mathcal{YD}$-module embedding $\varphi \colon \mathbbm k1_H \hookrightarrow V$.	
Let  $f\in V^*_{\mathbbm k}$, $\lambda := f(1_H)\ne 0$, $\lambda_i := f(x_i)$.
Then $$f \diamond V = \left\langle\left. f \diamond x_i \right| i\in \mathbb N\right\rangle_{\mathbbm k} + f \diamond \mathbbm k 1_H =
 \left\langle\left. x_i + \frac{\lambda_i}{\lambda}1_H \right| i\in \mathbb N\right\rangle_{\mathbbm k}+\mathbbm k 1_H.$$ In particular, $\dim(f \diamond V) = + \infty$.
Hence $f(1_H)=0$ for all $f\in V^*$ where the dual is taken in ${}_H^H\mathcal{YD}$.
Therefore, $\varphi^*=0$. At the same time, ${\mathbbm k}^* \ne 0$. Hence $\varphi^* \colon V^* \to {\mathbbm k}^*$ is not an epimorphism for the (extremal) monomorphism $\varphi$ in  ${}_H^H\mathcal{YD}$.
\item\label{RemarkYDThetaInvFlatIsNotMono}
	Let $V$ be the same as in~(\ref{RemarkYDPolyMonoStarIsNotEpi}) above.
	Then the proof of Lemmas~\ref{LemmaYDAlphaInjective} and~\ref{LemmaYDThetaInvFlatInjective} implies that
$\alpha_V(1_H)=0$ and $\left(\theta^\mathrm{inv}_{V,V}\right)^\flat(1_H \otimes 1_H)=0$. In particular,
$\alpha_V$ and $\left(\theta^\mathrm{inv}_{V,V}\right)^\flat$ are not monomorphisms.
\item  By Theorem~\ref{TheoremYDSatisfiesProperties} we may apply to the category $\mathcal C = {}_H^H\mathcal{YD}$
all the results of~\cite[Sections 4, 5]{AGV3} 
and Section~\ref{SubsectionMeasuringActingClosed} of the present paper.
Unfortunately, by the reasons mentioned in Remarks~\ref{RemarkYDPolyMonoStarIsNotEpi} and~\ref{RemarkYDThetaInvFlatIsNotMono} above, for an arbitrary $H$, we cannot apply to ${}_H^H\mathcal{YD}$ the results of Sections~\ref{SectionDuality}, \ref{SubsectionComeasuringCoactingClosed} and~\ref{SubsectionDualityClosed}.
Another reason, why we cannot apply the duality theorems for bimonoid and Hopf monoids (co)actions, is that those theorems have been proved only for symmetric categories.
\end{enumerate}
\end{remarks}

\subsection{Right Yetter~-- Drinfel'd modules}
\label{SubsectionYDRightSatisfiesPropertiesToo}

Let $F \colon \mathcal C \to \mathcal D$ be a braided strong monoidal functor between braided monoidal categories.
Suppose $F$ is an isomorphism of ordinary categories, $\mathcal C$ is pre-rigid and
there exists $(-)^\circ \colon \mathsf{Mon}(\mathcal C)\to \mathsf{Comon}(\mathcal C)$.
Then $\mathcal D$ is pre-rigid too and there exists  $(-)^\circ \colon \mathsf{Mon}(\mathcal D)\to \mathsf{Comon}(\mathcal D)$
where $G(-)^*\cong (G(-))^*$ and $G(-)^\circ \cong (G(-))^\circ$.
Moreover, for every  objects $A,B$ in $\mathcal C$, $M$ in $\mathsf{Mon}(\mathcal C)$
and a morphism $\varphi \colon A \to B^*$ the morphisms
  $\theta_{GA,GB}$, $\theta^\mathrm{inv}_{GA,GB}$,
$(G\varphi)^\flat$, $(G\varphi)^\sharp$, $\varkappa_{GM}$ in $\mathcal D$ can be identified with, respectively,
$G\theta_{A,B}$, $G\theta^\mathrm{inv}_{A,B}$,
$G\left(\varphi^\flat\right)$, $G\left(\varphi^\sharp\right)$ and $G\varkappa_M$ via the corresponding isomorphisms.
 Hence $\mathcal C$ and $\mathcal D$ share the same properties. In particular, this is true when $\mathcal D$ is just $\mathcal C$ with the opposite monoidal product and the functor $F$ is identical on objects and morphisms. This observation has the following important application.
 
  Let $H$ be a Hopf algebra over a field $\mathbbm k$ with an invertible antipode $S$. Consider the category $\mathcal{YD}_H^H$ of \textit{right Yetter~--- Drinfel'd modules} (or \textit{$\mathcal{YD}_H^H$-modules} for short), i.e. right $H$-modules and $H$-comodules $M$ such that
the $H$-action and the $H$-coaction $\rho \colon M \to M \otimes H$ 
satisfy the following compatibility condition:
$$\rho(mh)= m_{(0)} h_{(2)} \otimes  (Sh_{(1)})m_{(1)} h_{(3)} \text{ for every } m\in M \text{ and } h\in H.$$
The braiding on $\mathcal{YD}_H^H$ is defined by $c_{M,N}(m\otimes n) := n_{(0)}\otimes mn_{(1)}$
for all $m\in M$, $n\in N$ and $\mathcal{YD}_H^H$-modules $M$ and $N$.
(Note that in contrast with~\cite[Exercise 11.6.21]{Radford} we consider the standard $H$-(co)action
on $M \otimes N$ induced from those in ${}_H\mathsf{Mod}$ and $\mathsf{Comod}^H$, respectively. As a consequence, our formula for the braiding is different too.)
Then $\mathcal{YD}_H^H$ can be identified with ${}_{H^\mathrm{op,cop}}^{H^\mathrm{op,cop}}\mathcal{YD}$
where $H^\mathrm{op,cop}$ is the Hopf algebra $H$ with the opposite product and coproduct and the monoidal product is opposite too.
As a consequence, the analog of Theorem~\ref{TheoremYDSatisfiesProperties} holds for $\mathcal{YD}_H^H$ too.
The only difficulty that one can encounter on this way could be verifying the necessary and sufficient conditions
in the analogs of Lemmas~\ref{LemmaYDAlphaInjective} and~\ref{LemmaYDThetaInvFlatInjective} since the description of $A^*$ in $\mathcal{YD}_H^H$
in terms of ${}_{H^\mathrm{op,cop}}^{H^\mathrm{op,cop}}\mathcal{YD}$ becomes complicated due to the braiding involved
and $\mathrm{ev_{A}}$ no longer being tautological. For this reason, below we give
a natural description of the internal hom and the functor $(-)^*$ in $\mathcal{YD}_H^H$ as well as explicit formulas
for the related maps.

Given $\mathcal{YD}_H^H$-modules $M$ and $N$, the $\mathcal{YD}_H^H$-module $[M,N]$ as a right $H$-comodule coincides with the one defined in Section~\ref{SubsectionApplicationsComodules}.
The right $H$-action on $[M,N]$ is defined by $(fh)(m):=f(m S^{-1}h_{(2)})h_{(1)}$ for $h\in H$, $m\in M$.
In this description, $[M,N]$ is a subspace of $\Hom_\mathbbm k (M,N)$ and the evaluation map is tautological,
i.e. $f \otimes m \mapsto f(m)$ for $f\in [M,N]$ and $m\in M$. As usual, $M^* := [M, \mathbbm k]$.

For any ${}_H^H\mathcal{YD}$-modules $M$ and $N$ we have
$$\theta^\mathrm{inv}_{M,N}(f\otimes g)(m\otimes n) = f(m_{(0)})g\bigl(n S^{-1}m_{(1)}\bigr),$$
$$\varphi^\flat(n)(m)=\varphi\bigl(m S^{-1}n_{(1)}\bigr)(n_{(0)})$$
for all $m\in N$, $n\in M$, $f\in M^*$, $g\in N^*$, $\varphi \in {}_H^H\mathcal{YD}(M, N^*)$.

Lemmas~\ref{LemmaYDRightAlphaInjective} and~\ref{LemmaYDRightThetaInvFlatInjective} below complete the proof of the analog of Theorem~\ref{TheoremYDSatisfiesProperties}  for $\mathcal{YD}_H^H$:

\begin{lemma}\label{LemmaYDRightAlphaInjective}
	Let $M$ be an $\mathcal{YD}_H^H$-module. Then the map $\alpha_M \colon M \to M^{**}$ is injective
	if and only if $\bigcap\limits_{f \in M^*} \Ker f = 0$.
\end{lemma}	
\begin{proof}
	For every $m\in M$ and $f\in M^*$ we have
	$$\alpha_M(m)(f) := (\id_{M^*})^\flat (m)(f) = \left(f S^{-1}m_{(1)}\right)\bigl(m_{(0)}\bigr)
	= f\bigl(m_{(0)}S^{-2}m_{(1)}\bigr).$$
	Now we apply the analog of Lemma~\ref{LemmaYDSf(-1)Multiply}.
\end{proof}

\begin{lemma}\label{LemmaYDRightThetaInvFlatInjective}
	Let $M$ and $N$ be nonzero $\mathcal{YD}_H^H$-modules.
	Then the map $\left(\theta^\mathrm{inv}_{M,N}\right)^\flat
	\colon M\otimes N \to (M^*\otimes N^*)^*$ is injective if and only if $\bigcap\limits_{f \in M^*} \Ker f = 0$
	and $\bigcap\limits_{g \in N^*} \Ker g = 0$.
\end{lemma}	
\begin{proof} Define linear maps $\xi, \nu \colon M\otimes N \to M\otimes N$
	by $$\xi(m\otimes n)= (m\otimes n)_{(0)} S^{-2}(m\otimes n)_{(1)},$$
	$$\nu(m\otimes n)= m_{(0)}\otimes n\, S^{-1}m_{(1)}$$
	for $m\in M$, $n\in N$.
	
	Fix $m\in M$, $n\in N$, $f\in M^*$, $g\in N^*$. Then 
	\begin{equation*}\begin{split}\left(\theta^\mathrm{inv}_{M,N}\right)^\flat (m\otimes n)(f\otimes g)=
			\theta^\mathrm{inv}_{M,N}\bigl((f\otimes g)S^{-1}(m\otimes n)_{(1)}\bigr)\bigl((m\otimes n)_{(0)}\bigr)
			\\=	\theta^\mathrm{inv}_{M,N}(f\otimes g)\bigl((m\otimes n)_{(0)} S^{-2}(m\otimes n)_{(1)}\bigr)
			=	\theta^\mathrm{inv}_{M,N}(f\otimes g)\bigl(\xi(m\otimes n)\bigr)=
			(f\otimes g)\bigl(\nu\xi(m\otimes n)\bigr).\end{split}\end{equation*}
	Here we have used that for every right $H$-module homomorphism $\varphi \colon A\otimes B \to \mathbbm{k}$
	we have $$\varphi(ah \otimes b) = \varphi(ah_{(1)} \otimes b(S^{-1}h_{(3)})h_{(2)})=
	\varepsilon(h_{(1)}) \varphi\bigl(a \otimes bS^{-1}h_{(2)}) = \varphi(a \otimes b\,S^{-1}h)$$ for all $a\in A$, $b\in B$.	
	By the analog of Lemma~\ref{LemmaYDSf(-1)Multiply},
	the map $\xi$ is bijective. The map $\nu$ is bijective too since 
	$\nu^{-1}(m\otimes n)= m_{(0)}\otimes n m_{(1)}$.
Now we use the same argument as in Lemma~\ref{LemmaYDThetaInvFlatInjective}.
\end{proof}

\subsection{Modules over quasitriangular Hopf algebras}
Let $H$ be a \textit{quasitriangular} Hopf algebra over a field $\mathbbm k$, i.e. the category ${}_H \mathsf{Mod}$ is braided.
Then $c_{M,N}(u)=(Ru)^{21}$ for all left $H$-modules $M,N$ and $u\in M\otimes N$  where $R \in H \otimes H$ is a fixed invertible element called the \textit{$R$-matrix}, $(m\otimes n)^{21}:=n\otimes m$.
We refer to~\cite[Section 8.3]{EGNObook}, \cite[Chapter 2]{Majid} or~\cite[Section 10.1]{Montgomery} for more details.

\begin{example}
	Every cocommutative Hopf algebra $H$ is quasitriangular where $R=1_H \otimes 1_H$
	and $c_{M,N}$ is the ordinary swap.
\end{example}

Recall that in every quasitriangular Hopf
algebra $H$ the antipode is invertible~\cite[Proposition 2.1.8]{Majid}.
Moreover, for every left $H$-module $M$ the map $\delta \colon M \to H \otimes M$, where $\delta m := R^{21}(1_H \otimes m)$
for $m\in M$, defines on $M$ a structure of a left $H$-comodule that turns $M$ into a ${}_H^H\mathcal{YD}$-module.
The braiding induced by the ${}_H^H\mathcal{YD}$-module structure coincides with the original braiding on ${}_H \mathsf{Mod}$.
In other words, ${}_H \mathsf{Mod}$ can be identified with a full subcategory in ${}_H^H\mathcal{YD}$.
The embedding functor ${}_H \mathsf{Mod} \hookrightarrow {}_H^H\mathcal{YD}$ is a strict braided monoidal functor that commutes with the functor $(-)^*$ and taking small limits and colimits. This implies that $\theta_{A,B}$, $\theta^\mathrm{inv}_{A,B}$,
$\varphi^\flat$, $\varphi^\sharp$ in ${}_H \mathsf{Mod}$ coincide with those in ${}_H^H\mathcal{YD}$.
An explicit verification shows that the restriction of the functor $(-)^\circ \colon \mathsf{Mon}\bigl({}_H^H\mathcal{YD}\bigr)\to 
\mathsf{Comon}\bigl({}_H^H\mathcal{YD}\bigr)$
to $\mathsf{Mon}\bigl({}_H \mathsf{Mod}\bigr)$ takes values again in $\mathsf{Mon}\bigl({}_H \mathsf{Mod}\bigr)$. Moreover, for $H$-module algebras $A$
we have $A^\circ = A^\circ_\mathbbm{k}$ as vector spaces, but the comultiplication is still twisted.
Finally, the equality $M^*=M^*_\mathbbm{k}$ for every $H$-module $M$ implies that the functor $(-)^*$ maps monomorphisms to epimorphisms. By Theorem~\ref{TheoremYDSatisfiesProperties}, we get 
\begin{theorem}\label{TheoremHModSatisfiesProperties}
	Let $H$ be a quasitriangular Hopf algebra over a field $\mathbbm k$. 
	Then ${}_H \mathsf{Mod}$ is a braided closed monoidal category satisfying Properties~\ref{PropertySmallLimits}--\ref{PropertyFreeMonoid}, \ref{PropertyExtrMonomorphism} 
	of Section~\ref{SubsectionSupportCoactingConditions} 
	and their duals.
	Moreover, all monomorphisms and epimorphisms in ${}_H \mathsf{Mod}$ are extremal and
	\begin{itemize}
		\item the functor $(-)^*$ maps monomorphisms to epimorphisms;
		\item 
		$\alpha_M$, $\theta_{M,N}$, $\left(\theta^\mathrm{inv}_{M,N}\right)^\flat$ are monomorphisms
		for any $H$-modules $M,N$;
		\item $\varkappa_A$ is a monomorphism for every unital associative $H$-module algebra $A$.
	\end{itemize}
\end{theorem}

\begin{remark} By Theorem~\ref{TheoremHModSatisfiesProperties} we may apply to the category $\mathcal C = {}_H \mathsf{Mod}$
	all the results of~\cite[Sections 4 and 5]{AGV3},
	 Sections~\ref{SubsectionMeasuringActingClosed},  \ref{SubsectionComeasuringCoactingClosed} of the present paper
	as well as the duality theorems for (co)measurings (Theorems~\ref{Theorem(Co)monUniv(Co)measDuality}, \ref{Theorem(Co)monUniv(Co)measDualityClosed}, Corollaries~\ref{Corollary(Co)monUniv(Co)measDuality} and~\ref{Corollary(Co)monUniv(Co)measDualityClosed}).
	If, moreover, ${}_H \mathsf{Mod}$ is symmetric, i.e. the Hopf algebra $H$ is \textit{triangular},
	then we may apply to ${}_H \mathsf{Mod}$
	all the results of Sections~\ref{SectionDuality} and~\ref{SectionCosupportDualityInMonoidalClosedCategories},
	in particular, the duality theorems for bimonoid and Hopf monoids (co)actions.
\end{remark}	

\subsection{Comodules over coquasitriangular Hopf algebras}\label{SubsectionComodCoquasitriangularHopfAlgebras}

Let $H$ be a \textit{coquasitriangular} Hopf algebra over a field $\mathbbm k$, i.e. the category $\mathsf{Comod}^H$ is braided.
Then $c_{M,N}(m\otimes n)=R(m_{(1)}, n_{(1)})\, n_{(0)}\otimes m_{(0)}$ for all right $H$-comodules $M,N$, $m\in M$ and $n\in N$ where $R \colon H \otimes H \to \mathbbm k$ is a fixed linear map called the \textit{$R$-form}, which is invertible in $(H\otimes H)^*$.
More details can be found in e.g.~\cite[Section 8.3]{EGNObook}, \cite[Chapter 2]{Majid} or~\cite[Section 10.2]{Montgomery}.

\begin{example}
	Every commutative Hopf algebra $H$ is coquasitriangular where $R(h,t)=\varepsilon(h)\varepsilon(t)$ for all $h,t\in H$
	and $c_{M,N}$ is the ordinary swap.
\end{example}

Recall that in every coquasitriangular Hopf
algebra $H$ the antipode is invertible too~\cite[Proposition 2.2.4]{Majid}.
Moreover, every right $H$-comodule $M$
is simultaneously a right $H$-module $M$ too where $mh := R(m_{(1)}, h) m_{(0)}$ for all $m\in M$, $h\in H$,
which turns $M$ into a $\mathcal{YD}_H^H$-module.
The braiding induced by the $\mathcal{YD}_H^H$-module structure coincides with the original braiding on $\mathsf{Comod}^H$.
In other words, $\mathsf{Comod}^H$ can be identified with a full subcategory in $\mathcal{YD}_H^H$.
The embedding functor $\mathsf{Comod}^H \hookrightarrow \mathcal{YD}_H^H$ is a strict braided monoidal functor that commutes with the functor $(-)^*$ and taking small limits and colimits. This implies that $\theta_{A,B}$, $\theta^\mathrm{inv}_{A,B}$,
$\varphi^\flat$, $\varphi^\sharp$ in $\mathsf{Comod}^H$ coincide with those in $\mathcal{YD}_H^H$.

Note that the embedding $\mathsf{Comod}^H \subseteq \mathcal{YD}_H^H$ admits the left adjoint functor
$M \mapsto M/I(M)$ where $I(M)$ is the $\mathcal{YD}_H^H$-submodule of the $\mathcal{YD}_H^H$-module $M$
generated by all elements $mh-R(m_{(1)}, h) m_{(0)}$ where $m\in M$, $h\in H$.
If $C$ is a $\mathcal{YD}_H^H$-module coalgebra, then $I(C)$ is automatically a coideal.
Therefore, if $C$ is a limit in $\mathsf{Comon}\bigl(\mathcal{YD}_H^H\bigr)$ of a diagram in
$\mathsf{Comon}\bigl(\mathsf{Comod}^H\bigr)$, then $C/I(C)$, which now belongs to $\mathsf{Comon}\bigl(\mathsf{Comod}^H\bigr)$, must be a limit of this diagram too
and the embedding $\mathsf{Comon}\bigl(\mathsf{Comod}^H\bigr) \subseteq \mathsf{Comon}\bigl(\mathcal{YD}_H^H\bigr)$
preserves all limits. Now Remark~\ref{RemarkCircSufficientConditionsForExistence} and the proof of~\cite[Theorem 1.5]{ArdGoyMen1}
imply that the restriction of the functor $(-)^\circ \colon \mathsf{Mon}\bigl(\mathcal{YD}_H^H\bigr)\to 
\mathsf{Comon}\bigl(\mathcal{YD}_H^H\bigr)$
to $\mathsf{Mon}\bigl(\mathsf{Comod}^H\bigr)$ takes values in $\mathsf{Comon}\bigl(\mathsf{Comod}^H\bigr)$. By the remarks made in Section~\ref{SubsectionYDRightSatisfiesPropertiesToo}, the following theorem holds: 
\begin{theorem}\label{TheoremHComodSatisfiesProperties}
	Let $H$ be a coquasitriangular Hopf algebra over a field $\mathbbm k$. 
	Then $\mathsf{Comod}^H$ is a braided closed monoidal category satisfying Properties~\ref{PropertySmallLimits}--\ref{PropertyFreeMonoid}, \ref{PropertyExtrMonomorphism} 
	of Section~\ref{SubsectionSupportCoactingConditions} 
	and their duals.
	Moreover, all monomorphisms and epimorphisms in $\mathsf{Comod}^H$ are extremal and
	\begin{itemize}
	\item $\theta_{M,N}$ is a monomorphism
	for any $H$-comodules $M,N$;
	\item $\alpha_M$ is a monomorphism for an $H$-comodule $M$
	if and only if $\bigcap\limits_{f \in M^*} \Ker f = 0$;
	\item $\left(\theta^\mathrm{inv}_{M,N}\right)^\flat$ is a monomorphism
	for nonzero $H$-comodules $M$ and $N$ if and only if $\bigcap\limits_{f \in M^*} \Ker f = 0$  and $\bigcap\limits_{g \in N^*} \Ker g = 0$;
	\item $\varkappa_A$ is a monomorphism for every unital associative $H$-comodule algebra $A$.
\end{itemize}
\end{theorem}

\begin{remarks}\label{Comod_coquasi_rmk}
 \hspace{0.1cm}
 \begin{enumerate}
\item
 		If the Hopf algebra $H$ is finite dimensional, then $M^*=M^*_{\mathbbm k}$ and $\bigcap\limits_{f \in M^*} \Ker f = 0$ for every $H$-comodule $M$.
 \item\label{RemarkPolyMonoStarIsNotEpiRComod}  In $\mathsf{Comod}^H$ the functor $(-)^*$ does not necessarily map monomorphisms to epimorphisms either.	Let $\mathbbm k$ be a field and let $H=\mathbbm k[x_i \mid i\in \mathbb N]$, the algebra of polynomials in the variables $x_i$, where $\Delta_H x_i := x_i \otimes 1_H  + 1_H \otimes x_i$, $\varepsilon_H(x_i):=0$. Then $H$ is a right $H$-comodule where $\rho_H := \Delta_H$. 
	Consider $H$-subcomodules $\mathbbm k 1_H$ and $V:= \left\langle\left. x_i \right| i\in \mathbb N\right\rangle_{\mathbbm k}\oplus \mathbbm k 1_H$ of $H$. Denote by $\varphi$ the  $H$-comodule embedding $\varphi \colon \mathbbm k1_H \hookrightarrow V$.	 	Then the same argument as in Remark~\ref{left_YD_rem} (\ref{RemarkYDPolyMonoStarIsNotEpi})
	shows that $\varphi^*$ is not an epimorphism for the (extremal) monomorphism $\varphi$ in $\mathsf{Comod}^H$.
\item\label{RemarkComodThetaInvFlatIsNotMono} Let $V$ be the same as in~(\ref{RemarkPolyMonoStarIsNotEpiRComod}) above.
	Then the proof of Lemmas~\ref{LemmaYDRightAlphaInjective} and~\ref{LemmaYDRightThetaInvFlatInjective} implies that
	$\alpha_V(1_H)=0$ and $\left(\theta^\mathrm{inv}_{V,V}\right)^\flat(1_H \otimes 1_H)=0$. In particular,
	$\alpha_V$ and $\left(\theta^\mathrm{inv}_{V,V}\right)^\flat$ are not monomorphisms.
\item By Theorem~\ref{TheoremHComodSatisfiesProperties} we may apply to the category $\mathcal C = \mathsf{Comod}^H$
	all the results of~\cite[Sections 4 and 5]{AGV3} and 
	Section~\ref{SubsectionMeasuringActingClosed} of the present paper.
	If, moreover, $\mathsf{Comod}^H$ is symmetric, i.e. the Hopf algebra $H$ is \textit{cotriangular},
and $\bigcap\limits_{f \in M^*} \Ker f = 0$ for every $H$-comodule $M$,
then we may apply to $\mathsf{Comod}^H$ the duality theorems for (co)measurings and (co)actions from Section~\ref{SectionDuality}.
In particular, if $H=\mathbbm k G$ for an abelian group $G$, i.e. $\mathsf{Comod}^H$ is just the category of $G$-graded vector spaces,
then we may apply to $\mathsf{Comod}^{H}$
	all the results of Sections~\ref{SectionDuality} and~\ref{SectionCosupportDualityInMonoidalClosedCategories}.
	By the reasons mentioned in~(\ref{RemarkPolyMonoStarIsNotEpiRComod})
	and~(\ref{RemarkComodThetaInvFlatIsNotMono}) above, we cannot apply to $\mathsf{Comod}^H$ the results of Sections~\ref{SectionDuality}, \ref{SubsectionComeasuringCoactingClosed} and~\ref{SubsectionDualityClosed}
	for arbitrary~$H$.
\end{enumerate}
\end{remarks}

\subsection{Differential graded vector spaces}

Let $\mathbbm{k}$ be a field. Let $\mathbf{dgVect}_\mathbbm{k}$ be the category of differential $\mathbb{Z}$-graded vector spaces (or dg-vector spaces for short) or, in another terminology, chain complexes in $\mathbf{Vect}_\mathbbm{k}$. Objects in $\mathbf{dgVect}_\mathbbm{k}$
are families $(V_n)_{n\in\mathbb Z}$ of vector spaces $V_n$ equipped with linear maps $d \colon V_n \to V_{n-1}$, $n\in\mathbb Z$,
such that $d^2 = 0$. The maps $d$ are called \textit{differentials}. Every family $(V_n)_{n\in\mathbb Z}$ can be identified with its $\mathbb{Z}$-graded total space $V=\bigoplus\limits_{n\in\mathbb Z} V_n$. Moreover $d$ extends to a graded linear map $V \to V$ of degree $(-1)$ such that $d^2 = 0$. Morphisms in $\mathbf{dgVect}_\mathbbm{k}$ are grading preserving (= graded of degree $0$) linear maps commuting with $d$.
Note that $\mathbf{dgVect}_\mathbbm{k}$ is an abelian category where limits and colimits are computed componentwise.

Let $U=\bigoplus\limits_{k\in\mathbb Z} U_k$ and $V=\bigoplus\limits_{m\in\mathbb Z} V_m$ be two dg-vector spaces. Then the monoidal product $W=U\otimes V$ in $\mathbf{dgVect}_\mathbbm{k}$ is defined by $W:=\bigoplus\limits_{n\in\mathbb Z} W_n$ where $$W_n := \bigoplus\limits_{k\in\mathbb Z} U_k \otimes V_{n-k}.$$ The differentials $d \colon W_n \to W_{n-1}$
are defined by $$d(u\otimes v) := du \otimes v + (-1)^{k} u\otimes dv\text{ for }u\in U_k\text{ and }v\in V_m,\ k,m\in\mathbb Z.$$
The monoidal unit in $\mathbf{dgVect}_\mathbbm{k}$ is $\mathbbm{k}$ regarded as a chain complex concentrated in degree $0$ with zero differential.
The category $\mathbf{dgVect}_\mathbbm{k}$ is symmetric where the swap $c \colon U_k \otimes V_m \mathbin{\widetilde\to} V_m \otimes U_k$ is defined by $c(u\otimes v) := (-1)^{km} v \otimes u$ for all $u\in U_k$, $v\in V_m$, $k,m\in\mathbb Z$.

Monoids in $\mathbf{dgVect}_\mathbbm{k}$ are just unital associative differential graded algebras (or dg-algebras for short).

\begin{theorem}\label{TheoremDgVectSatisfiesProperties}
	Let $\mathbbm k$ be a field. 
	Then $\mathbf{dgVect}_\mathbbm{k}$ is a symmetric closed monoidal category satisfying Properties~\ref{PropertySmallLimits}--\ref{PropertyFreeMonoid}, \ref{PropertyExtrMonomorphism} 
	of Section~\ref{SubsectionSupportCoactingConditions} 
	and their duals.
	Moreover, all monomorphisms and epimorphisms in $\mathbf{dgVect}_\mathbbm{k}$ are extremal and
	\begin{itemize}
		\item the functor $(-)^*$ maps monomorphisms to epimorphisms;
		\item 
		$\alpha_V$, $\theta_{U,V}$, $\left(\theta^\mathrm{inv}_{U,V}\right)^\flat$ are monomorphisms
		for any dg-vector spaces $U,V$;
		\item $\varkappa_A$ is a monomorphism for every unital associative dg-algebra $A$.
	\end{itemize}
\end{theorem}
\begin{proof}
	Consider the Hopf algebra $H$ over $\mathbbm k$ with the basis $c^k v^\ell$, where $k\in \mathbb Z$, $\ell = 0,1$,
	$vc = -cv$, $v^2=0$. $\Delta v = c\otimes v + v \otimes 1$, $\Delta c = c\otimes c$, $Sc= c^{-1}$, $Sv=-c^{-1} v$.
	Then $\mathbf{dgVect}_\mathbbm{k}$ can be identified with $\mathsf{Comod}^H$
	where for every dg-vector space $(V_m)_{m\in\mathbb Z}$ 
	the structure of a right $H$-comodule on $\bigoplus_{m\in \mathbb Z} V_m$
	is given by
	$\rho(a) := a\otimes c^{-m} + da \otimes vc^{-m}$ for $a \in V_m$, $m\in\mathbb Z$,
	and if $V$ is a right $H$-comodule, then $V_m := \lbrace a \in V \mid \lambda_m(a_{(1)}) a_{(0)}=a \rbrace$,
	$da := \mu(a_{(1)}) a_{(0)}$. Here $\lambda_m, \mu \in H^*$ are defined by
	$\mu(c^k v^\ell) := \delta_{\ell 1}$, $\lambda_m(c^k v^\ell) := \delta_{k,-m}\delta_{\ell 0}$.
	(Note that $\mu^2 = 0$ and $\lambda_{m-1} \mu = \mu \lambda_m$.)
	
	Recall that $\mathbf{dgVect}_\mathbbm{k}$ is closed. An explicit formula for $[U,V]$ can be found e.g. in ~\cite{AneJoy},
	but the closedness of $\mathbf{dgVect}_\mathbbm{k}$ follows from the above identification of $\mathbf{dgVect}_\mathbbm{k}$
	with $\mathsf{Comod}^H$ too.
	For a dg-vector space $V$ we obtain that 
	$V^* := [V, \mathbbm k]$ consists of all linear functions $V \to \mathbbm k$ that are nonzero only on a finite number of components.
	Each monomorphism $\varphi \colon U \to V$ in $\mathbf{dgVect}_\mathbbm{k}$ is just an embedding of graded components
	 compatible with the differentials. Since each linear function $f \in U^*$ can be extended to a linear function
	that belongs to $V^*$, the map $\varphi^* \colon V^* \to U^*$ is surjective, i.e. the functor $(-)^*$ maps monomorphisms to epimorphisms.
	
	The rest of the properties follow from Theorem~\ref{TheoremHComodSatisfiesProperties}.
\end{proof}	

\begin{remark} By Theorem~\ref{TheoremDgVectSatisfiesProperties} we may apply to the category $\mathcal C = \mathbf{dgVect}_\mathbbm{k}$
	all the results of~\cite[Sections 4 and 5]{AGV3} 
	and  Sections~\ref{SectionDuality} and~\ref{SectionCosupportDualityInMonoidalClosedCategories}
	of the present paper.
\end{remark}

\end{document}